\RequirePackage{fix-cm}%
\documentclass[a4paper,oneside,reqno]{amsart}
\usepackage[latin1]{inputenc}
\usepackage[T2A,T1]{fontenc}
\usepackage{lmodern,fixltx2e}%
\def\SetBibliographyCyrillicFontfamily{\fontfamily{cmr}}%
\usepackage{etex}%
\usepackage[german,british]{babel}
\usepackage{hyphxcpt}
\usepackage[a4paper,verbose,textwidth=338.58778pt]{geometry}

\newif\ifforceblacklinks\forceblacklinksfalse

\def\PublicationTitle{Galois theory for semiclones}%
\def\CorrespondingAuthor{Mike Behrisch}%
\def\TUWname{\foreignlanguage{german}{%
             Tech\-ni\-sche Uni\-ver\-si\-t\"{a}t Wien}}%
\def\InstitutCL{\foreignlanguage{german}{%
             In\-sti\-tut f\"{u}r Com\-pu\-ter\-spra\-chen}}%
\def\PostleitzahlWien{\mbox{A-1040} Vienna}%
\def\PublicationTopic{Galois theory for semiclones %
                      (iterative algebras)}%
\def\PublicationKeywords{iterative algebra,
                         semiclone,
                         relation pair clone,
                         Galois theory%
                         }

\usepackage{amsmath}
\usepackage{amssymb}
\usepackage{amsfonts}
\usepackage{amsthm}
\theoremstyle{plain}
\newtheorem{theorem}{Theorem}[section]
\newtheorem{lemma}[theorem]{Lemma}
\newtheorem{proposition}[theorem]{Proposition}
\newtheorem{corollary}[theorem]{Corollary}

\theoremstyle{definition}
\newtheorem{definition}[theorem]{Definition}

\newtheorem{remark}[theorem]{Remark}

\usepackage{enumerate}
\usepackage[mathscr]{eucal}
\usepackage{url}                     %

\usepackage{mathoperators}
\usepackage[uglyemptyset,hyperref]{usefulthings}
\usepackage{PolInv}

\newenvironment{prooflist}%
  {\begin{list}{}{%
    \setlength{\labelwidth}{0pt}
    \setlength{\leftmargin}{0.5\parindent}
    \setlength{\itemindent}{-\leftmargin}
    \setlength{\listparindent}{\parindent}
    }}%
  {\end{list}}
\newenvironment{easyproof}%
  {\begin{proof}}%
  {\end{proof}}%
\newenvironment{easykeepproof}%
  {\begin{proof}}%
  {\end{proof}}%
\newenvironment{straightforwardproof}%
  {\begin{proof}}%
  {\end{proof}}%

\DeclareMathOperator{\pr}{pr}%
\DeclareMathOperator{\OP}{OP}
\DeclareMathOperator{\LOOP}{\mathbf{L}\mathbf{O}}
\DeclareMathOperator{\PolpOp}{Polp}%
\DeclareMathOperator{\InvpOp}{Invp}%
\DeclareMathOperator{\SemiCloneLatOp}{\mathcal{S}}%
\newcommand{\Sclones}[1][\DefaultCarrier]{\ensuremath{\SemiCloneLatOp_{#1}}}
\newcommand{\genInvPair}[2][F]{\ensuremath{\Gamma_{#1}\apply{#2}}}%
\DefineRelCloneClosureLikeCmd{\genSclone}{\Op}%
\newcommand{\Relp}[1]{\RelpOp_{#1}}%
\DefineRelCloneClosureLikeCmd{\genRpclone}{\Relp}%
\newcommand{\LO}[2][s]{\ensuremath{\LOOP_{#1}\apply{#2}}}%
\newcommand{\genComp}[3]{\ensuremath{\bigwedge^{#1}_{#2}#3}}%
\DefinePolLikeCmd{\Polp}{\PolpOp}
\DefinePolLikeCmd{\Invp}{\InvpOp}
\makeatletter
\def\DefaultCarrier{\PolInv@defaultsetA}
\newcommand{\CarrierSet}[1][]{%
  \ifthenelse{\equal{#1}{}}{%
    \ensuremath{\DefaultCarrier}
  }{%
    \ensuremath{\DefaultCarrier^{#1}}
  }
}
\DeclareMathOperator{\ConstantsOp}{C}%
\newcommand{\Constants}[2][]{\@OpWithOptionalArityAndCombination{#1}{#2}{\ConstantsOp}{\@combinewithindex}}
\newcommand{\Ops}[1][]{\Op[#1]{\DefaultCarrier}}%
\newcommand{\TrivOps}[1][]{\J[#1]{\DefaultCarrier}}%
\let\Clones=\LA%
\newcommand*{\Rels}[1][]{\Rel[#1]{\DefaultCarrier}}%
\DeclareMathOperator{\RelpOp}{Rp}%
\newcommand*{\Relps}[1][]{%
  \ifthenelse{\equal{#1}{}}{%
    \ensuremath{\RelpOp_{\DefaultCarrier}}
  }{%
    \ensuremath{\RelpOp_{\DefaultCarrier}^{\apply{#1}}}
  }
}
\newcommand{\qleq}{\preceq}%
\newcommand{\enc}[1]{\ensuremath{%
  {}^{\scriptscriptstyle\rightarrow}#1{}^{\scriptscriptstyle\leftarrow}}}%
\newcommand*{\Palatalization}[1]{%
  \bgroup\fontencoding{T1}\selectfont\v{#1}\egroup}
\newcommand{\polymer}[2]{\ensuremath{\delta_{#1}\apply{#2}}}%
\makeatother

\hypersetup{
    pdftitle={\PublicationTitle},    %
    pdfauthor={\CorrespondingAuthor},     %
    pdfsubject={\PublicationTopic},   %
    pdfkeywords={\PublicationKeywords},         %
}
\ifforceblacklinks%
\hypersetup{%
    linkcolor=black,     %
    citecolor=black,     %
    filecolor=black,     %
    urlcolor=black       %
}%
\fi%
\numberwithin{equation}{section} %
\allowdisplaybreaks
\begin{document}

\thispagestyle{empty}
\selectlanguage{british}

\title{\PublicationTitle}
\thanks{Supported by the Austrian Science Fund (FWF) under grant I836-N23.}
\author[M. Behrisch]{\CorrespondingAuthor}%
\address{\InstitutCL\\
         \TUWname\\
         \PostleitzahlWien\\
         Austria}%
\email{behrisch@logic.at}%
\date{\today}
%% AMS subject classification; see http://www.ams.org/msc
%% Only one Primary. Possibly several Secondary.
\subjclass[2010]{Primary:
  08A40; % General algebraic systems, Algebraic structures,
         % Operations, polynomials, primal algebras
  Secondary:
  08A02, % General algebraic systems, Algebraic structures,
         % Relational systems, laws of composition
  08A99. % General algebraic systems, Algebraic structures,
         % None of the above, but in this section
  }%

\keywords{\PublicationKeywords}

\begin{abstract}
We present a \name{Galois} theory connecting finitary operations with
pairs of finitary relations one of which is contained in the other. The
\name{Galois} closed sets on both sides are characterised as locally
closed subuniverses of the full iterative function algebra (semiclones)
and relation pair clones, respectively.
Moreover, we describe the modified closure operators if only functions and relation
pairs of a certain bounded arity, respectively, are considered.
\end{abstract}
\maketitle

\section{Introduction}\label{sect:intro}
\emph{Clones of operations}, \ie\ composition closed sets of operations
containing all projections
(cf.~\cite{PoeKal,%
           SzendreiClonesInUniversalAlgebra,%
           LauFunctionAlgebrasOnFiniteSets,%
           GoldsternPinskerSurveyOfClonesOnInfiniteSets}),
play an important role in universal algebra as they encode structural
properties independently of the similarity type of the algebra. It is
well\dash{}known
(see~\cite{BodnarcukKaluzninKotovRomovGaloisTheoryForPostAlgebras,%
           GeigerClosedSystemsOfFunctionsAndPredicates},
 translations available
 in~\cite{BodnarcukKaluzninKotovRomovGaloisTheoryForPostAlgebrasI,%
          BodnarcukKaluzninKotovRomovGaloisTheoryForPostAlgebrasII})
that on finite carrier sets clones are in a one\dash{}to\dash{}one
correspondence with structures called \emph{relational clones}.
This is established via the \name{Galois} correspondence
\m{\PolOp\text{-}\InvOp}, which is induced
by the relation of ``functions \emph{preserving} relations''. In general,
\ie\ including in particular the case of infinite sets, so\dash{}called
\emph{local closure operators} come into play
(see~\cite{GeigerClosedSystemsOfFunctionsAndPredicates,%
PoeGeneralGaloisTheoryForOperationsAndRelations,%
PoeConcreteRepresentationOfAlgebraicStructuresAndGeneralGaloisTheory,%
KerkhoffGeneralGaloisTheoryFunRelInCats,%
BehClonesWithNullaryOperations}),
and also the notion of relational clone as known from finite domains needs
to be generalised (cf.~ibid.). In this way the \name{Galois} connection
singles out certain \emph{locally closed} clones from the lattice of all
clones on a given set. These clones can also be seen as those which are
topologically closed \wrt\ the topology that one gets by endowing each
set~\m{\CarrierSet[{\CarrierSet[n]}]}, \m{n\in\N}, with the product topology
arising from~\m{\CarrierSet} initially carrying the discrete topology
(see \eg~\cite{BodirskyPinskerTopologicalBirkhoff,
               BodirskyPinskerMinimalFunctionsRandomGraph}).
\par

By equipping the set of all finitary functions on a fixed
set~\m{\CarrierSet} with a finite
number of operations (including permutation of variables, identification of
variables, introduction of fictitious variables, a certain binary
composition operation and a projection as a constant; we present more
details later on), one obtains the \emph{full function algebra} of finitary
functions on~\m{\CarrierSet}. It is known
(cf.\ \eg~\cite{PoeKal,LauFunctionAlgebrasOnFiniteSets}) that the clones
on~\m{\CarrierSet} are exactly the carrier sets of subalgebras of this
structure. This relationship is a special case of the one between the
\emph{full iterative function algebra}, also known as \emph{iterative
\name{Post} algebra} (introduced by Ma\Palatalization{l}cev
in~\cite{MalcevIterativeAlgebrasAndPostVarieties1966}), and its
subuniverses (called \emph{\name{Post} algebras}
in~\cite{BodnarcukKaluzninKotovRomovGaloisTheoryForPostAlgebras}), which
have often only been referred to as \emph{closed classes} of functions in
the Russian literature
(\eg~\cite{JablonskijFunctionalConstructions1958,%
           JablonskijGavrilovKudrjavcevFunctionsAlgLogicAndPostClasses}).
These are similar in spirit to
clones, but they do not need to contain the projections
(\emph{selectors} in the terminology
of~\cite{MalcevIterativeAlgebrasAndPostVarieties1966}) as the iterative
\name{Post} algebra omits the projection constant in its signature
compared to the full function algebra.
\par

In analogy to the \m{\PolOp\text{-}\InvOp} \name{Galois} connection, a
\name{Galois} correspondence \m{\PolpOp\text{-}\InvpOp} has been developed
in~\cite{Harnau8Habil}
(see also~\cite{Harnau7AVerallgRelationenbegriffI,%
                Harnau7BVerallgRelationenbegriffII,%
                Harnau7CVerallgRelationenbegriffIII})
based on the notion of functions preserving pairs \m{\apply{\rho,\rho'}} of
relations \m{\rho'\subs\rho}.
For finite carrier sets the \name{Galois} closed sets have been
characterised to be precisely the subuniverses of the full iterative
\name{Post} algebra and the subuniverses of a suitably defined relation
pair algebra, respectively. To the best knowledge of the author, a
generalisation of this result to arbitrary base sets has not yet appeared
in the literature.
In particular the general (and thus infinite) case is also missing in
Table~1 of~\cite[p.~296]{%
CouceiroLehtonenGaloisTheoryPermCylindrificationComposition}
summarising related \name{Galois} connections and characterisations of
their closure operators.
\par

In this article it is our aim to fill in this gap. We first coin the
notion of a \emph{semiclone}, which relates to transformation semigroups in
the same way as clones relate to transformation monoids. It is not hard to
figure out that semiclones and subuniverses of the full iterative
\name{Post} algebra coincide. However, not only is the name shorter, but
also do we feel that the way how a semiclone is defined is much more
similar to the usual definition of a clone and easier to grasp than that
of a subalgebra of the full iterative function algebra; hence the
proposition of the new terminology of semiclones.
Unfortunately, our semiclones are different from those appearing
in~\cite{SchmidtClonesAndSemiclones}, which are closed \wrt\ a different
form of composition, do have to contain the identity operation, but are
not necessarily closed under variable substitutions.
\par

In a similar fashion as one needed to generalise the notion of relational
clone to accommodate the closed sets of \m{\Inv{\Pol{}}} for infinite carrier
sets~\m{\CarrierSet}, it will be necessary to modify the relation pair
algebra proposed by Harnau in~\cite{Harnau8Habil}. We shall refer to the
corresponding (new) subuniverses as \emph{relation pair clones}.
\par

Using the same local closure operator as introduced
in~\cite{GeigerClosedSystemsOfFunctionsAndPredicates,%
   PoeConcreteRepresentationOfAlgebraicStructuresAndGeneralGaloisTheory,%
   PoeGeneralGaloisTheoryForOperationsAndRelations} for sets of
functions (the topological closure), and appropriately modifying the local
closure on the side of relation pairs, we shall prove the following two
main results: the \name{Galois} closed sets of operations \wrt\
\m{\Polp{\Invp{}}} are exactly the locally closed semiclones.
Dually, the closed sets of \m{\Invp{\Polp{}}} are precisely the locally
closed relation pair clones.
\par

Because it fits nicely in this context, we shall
more specifically study and characterise what it means that a semiclone can
be described in the form \m{\Polp{Q}} for some set~\m{Q} of at most
\nbdd{s}ary relation pairs, and that a relation pair clone is given by
\m{\Invp{F}} using a set~\m{F} of at most \nbdd{s}ary operations. As
in~\cite{PoeGeneralGaloisTheoryForOperationsAndRelations} this involves
certain \emph{\nbdd{s}local} closure operators, and, in general, the reader
may find that quite a few results in our text are analogous to
those in~\cite{PoeGeneralGaloisTheoryForOperationsAndRelations}, where
similar questions have been studied \wrt\ \m{\PolOp\text{-}\InvOp}.
\par

We mention that a related, in some sense more general, \name{Galois}
connection has been studied in~\cite{PippengerGaloisTheoryForMinors}
(finite case) and~\cite{%
CouceiroFoldesClosedSetsRelConstrFuncsClosedWrtVarSubst,%
CouceiroGaloisConnectionsExternalOperationsRelationalConstraints}. There,
for fixed sets~\m{\CarrierSet} and~\m{B}, functions
\m{\functionhead{f}{\CarrierSet[n]}{B}} have been related to pairs of
relations \m{R\subs\CarrierSet[m]}, \m{S\subs B^m} for some \m{m\in\Np},
called \emph{relational constraints}. In this situation the \name{Galois}
closed sets on the functional side are also closed \wrt\ variable
substitutions (as our semiclones), but already for syntactic reasons
cannot be closed \wrt\ compositions. So even if one considers the special
case that \m{B=A}, the results
from~\cite{CouceiroFoldesClosedSetsRelConstrFuncsClosedWrtVarSubst}
and~\cite{CouceiroGaloisConnectionsExternalOperationsRelationalConstraints}
describe similar but differently closed sets of functions due to other
objects on the dual side (there is no containment condition for the
relations as in our setting since for general~\m{\CarrierSet} and~\m{B}
there cannot be one).
\par

We acknowledge that, perhaps, it could be possible to derive
our results by restricting the relational side of the \name{Galois}
correspondence studied
in~\cite{CouceiroFoldesClosedSetsRelConstrFuncsClosedWrtVarSubst,%
CouceiroGaloisConnectionsExternalOperationsRelationalConstraints}, but we
think that the way of describing the closed objects on the dual side used
there is (and has to be) more complicated (using so\dash{}called
conjunctive minors), in fact, too technical for our situation.
Besides, our strategy of proof
exhibits more similarities with the classical arguments known from clones
and relational clones. Also the local closures developed for relational
constraints
in~\cite{CouceiroFoldesClosedSetsRelConstrFuncsClosedWrtVarSubst,%
         CouceiroGaloisConnectionsExternalOperationsRelationalConstraints}
necessarily need to be modified (see Remark~\ref{rem:local-closures}) to
be used with our relation pairs due to the inclusion requirement in their
definition.
\par

Still a different weakening of the notion of clone and an associated
\name{Galois} theory for arbitrary domains has been considered
in~\cite{LehtonenClosedClassesOfFunctionsGeneralizedConstraintsClusters}:
there sets of functions that contain (as clones do) all projections, are
closed under substitution of one function into the first place of
another one, permutation of positions and addition of
fictitious variables but are not necessarily closed under variable
identification (as semiclones are) have been characterised in terms of
closed sets of so\dash{}called \emph{clusters}. For the classes of functions
characterised
in~\cite{LehtonenClosedClassesOfFunctionsGeneralizedConstraintsClusters}
contain all projections, these results explore a separate direction and
cannot be exploited either to obtain the missing general (infinite) case
for semiclones.
\par

\subsection*{Acknowledgements}
The author expresses his gratitude to Erhard \mbox{Aichinger} for an
invitation to the Institute for Algebra at Johannes Kepler University
Linz, which enabled fruitful discussions on some aspects of the topic with
members of the institute including Erhard Aichinger, Peter Mayr, Keith
Kearnes and \'{A}gnes Szendrei. The author wishes to thank them, too, for
their valuable comments and contributions.
\par

\section{Preliminaries}

\subsection{Notation, functions and relations}\label{subsect:notations}
In this article the symbol~\m{\N} will denote the set of all natural
numbers (including zero), and~\m{\Np} will be used for
\m{\N\setminus\set{0}}.
Moreover, we shall make use of the standard set theoretic representation of
natural numbers by \name{John von Neumann},
\ie\ \m{n = \lset{i\in \N}{i <n}} for \m{n \in \N}. The \emph{power set} of
a set~\m{S} will be denoted by~\m{\powerset{S}}.
\par

When discussing semiclones, relation pair clones and their
\name{Galois} theory we shall make no further assumptions on the carrier
set, which we usually represent by~\m{\CarrierSet}. Any finite (including 0)
or infinite cardinality is allowed for~\m{\CarrierSet}.
\par

For sets~\m{A} and~\m{B} we write~\m{A^B} for the set of all mappings
from~\m{B} to~\m{A}. The order of composition employed in this article is from
right to left, \ie\ \m{g\circ f \in C^A} for \m{f \in B^A} and
\m{g \in C^B}. That is, \m{g\circ f} maps elements \m{a \in A} to
\m{g(f(a))}. For any index set~\m{I}, sets~\m{A} and
\m{\apply{B_i}_{i\in I}} and maps
\m{\apply{\functionhead{f_i}{A}{B_i}}_{i\in I}}, their \emph{tupling} is
the unique map \m{\functionhead{h}{A}{\prod_{i\in I} B_i}} satisfying
\m{\pi_i \circ h = f_i} for each \m{i\in I}, where
\m{\functionhead{\pi_i}{\prod_{j\in I} B_j}{B_i}} is the \nbdd{i}th
projection map belonging to the \name{Cartes}ian product
\m{\prod_{j\in I} B_j}. As any ambiguity can usually be resolved from the
context, we denote the tupling~\m{h} by \m{\apply{f_i}_{i\in I}}, in the
same way as the tuple~\m{\apply{f_i}_{i\in I}}.
\par

The notion of tupling is, of course, meaningful (by definition) in any
category having suitable products, and hence the following simple lemma
about composition of tuplings can be proven in such a general context. We
recall here just its instance for the category of sets
(cf.~\cite[Lemma~2.5]{BehClonesWithNullaryOperations}):

\begin{lemma}\label{lem:superassociativity}
Let~\m{I} and~\m{J} be arbitrary index sets, \m{k,m,n \in \N} natural
numbers, and \m{A, B, D, X} and~\m{B_i} (\m{i\in I}), \m{C_j} (\m{j \in J})
be sets.
Furthermore, suppose that we are given
mappings \m{\functionhead{r}{A}{B}}, \m{\functionhead{r_i}{A}{B_i}} (\m{i
\in I}), \m{\functionhead{g_j}{B}{C_j}} (\m{j \in J}), and
\m{\functionhead{f}{\prod_{j \in J}C_j}{D}}.
\begin{enumerate}[(a)]
\item\label{item:pre-superassociativity}
      One has
      \m{\ds \apply{g_j}_{j \in J} \circ r = \apply{g_j \circ r}_{j \in J}}.
\item\label{item:general-superassociativity}
      If \m{B = \prod_{i\in I} B_i}, then
      \m{\ds \apply{g_j}_{j \in J} \circ \apply{r_i}_{i \in I}
           = \apply{g_j \circ \apply{r_i}_{i \in I}}_{j\in J}}, and thus
      \[ \apply{f \circ \apply{g_j}_{j \in J}} \circ \apply{r_i}_{i \in I}
           = f \circ \apply{g_j \circ \apply{r_i}_{i \in I}}_{j \in J}.\]
\item\label{item:superassociativity}
      If \m{B_i=C_j=D=X} for \m{i \in I} and \m{j \in J}, \m{A= X^k} and
      \m{I=m} and \m{J=n}, then we have
      \begin{multline*}
         \apply{f \circ \apply{g_0,\dotsc,g_{n-1}}} \circ
            \apply{r_0,\dotsc,r_{m-1}}\\
           = f \circ \apply{g_0 \circ \apply{r_0,\dotsc,r_{m-1}}, \dotsc,
                            g_{n-1} \circ \apply{r_0,\dotsc,r_{m-1}}},
      \end{multline*}
      the \emph{superassociativity law} for finitary operations on~\m{X}.
\end{enumerate}
\end{lemma}

As in our modelling natural numbers are sets, we consequently interpret
tuples as maps, too: if \m{B=n \in \N} is a natural number, then
\m{A^B=A^n} is the set of all \nbd{\m{n}}tuples \m{x = \apply{x(i)}_{i<n}}.
We shall often write~\m{x_i} for the entry~\m{x(i)} (\m{i \in n}), and,
whenever convenient, we shall also refer to the entries of tuples by
different indexing, \eg\ \m{x=\apply{\liste{x}{n}}}.
Note that the sole element of
\m{A^0=A^\emptyset} is the empty mapping (tuple), whose graph is the empty
relation. It will consistently be denoted by~\m{\emptyset}.
As tuples are functions we may compose them with other functions: for
instance, if \m{x \in A^n} and \m{\alpha\colon m \to n}, (\m{m,n \in \N}),
then \m{x \circ \alpha} is the tuple in~\m{A^m} whose entries are
\m{x_{\alpha(i)}} (\m{i\in m}). Similarly, if \m{g\colon A \to B}, then
\m{g \circ x = \apply{g(x_i)}_{i\in n}} is an element of~\m{B^n}.\par

Any mapping \m{f \in \CarrierSet[{\CarrierSet[n]}]} (\m{n \in \N}) is
called an \emph{\nbd{\m{n}}ary operation} on~\m{\CarrierSet}, and the
number~\m{n} is referred to as its \emph{arity}, denoted by
\m{\arity\apply{f}}. The set of
all \emph{finitary operations} on~\m{\CarrierSet} is
\m{\Ops\defeq \biguplus_{k \in \N} \CarrierSet[{\CarrierSet[k]}]}. Note
that we explicitly include nullary operations here, which is slightly
uncommon in standard clone theory. For a set of operations
\m{F \subs \Ops} we denote its \emph{\nbd{\m{n}}ary part} by
\m{\Fn{F}\defeq F \cap \CarrierSet[{\CarrierSet[n]}]}. We extend this
notation to operators yielding subsets of operations:
if \m{\OP\colon S \to \powerset{\Ops}} is
an operator on a set~\m{S}, %
then we define \m{\functionhead{\Fn{\OP}}{S}{\powerset{\Ops[n]}}} by
the restriction \m{\Fn{\OP}(s) \defeq \Fn{\apply{\OP(s)}}} for \m{s\in S}.
Based on this, we put moreover
\m{\Fn[{\liste{n}{k}}]{\OP}(s)\defeq
   \biguplus_{i=1}^{k} \Fn[n_{i}]{\OP}(s)}
for \m{s\in S} and a finite list of arities \m{\liste{n}{k}}, \m{k>0}.
We also abbreviate \m{\Fn[{0,\dotsc,n}]{\OP}} as
\m{\Fn[\leq n]{\OP}}, and for \m{s\in S} we let
\m{\Fn[{>0}]{\OP}(s)\defeq \OP(s)\setminus\Fn[0]{\OP}(s)}.
\par

The projection operations belonging to the finite \name{Cartes}ian powers
of the carrier set play a special role. For \m{n\in\N} and \m{i\in n}, we
denote by \m{\eni{i}\in\Ops[n]} the \emph{\nbdd{n}ary projection on the
\nbdd{i}th coordinate}. Evidently, there do not exist any nullary
projections. Therefore, the \emph{set of all projections}
on~\m{\CarrierSet}, denoted by \m{\TrivOps}, equals
\m{\bigcup_{n\in\Np}\lset{\eni{i}}{0\leq i<n}}. For the identity operation
\m{\eni[1]{0}} we occasionally also use the notation \m{\id_{\CarrierSet}}.
\par

Knowing about re\dash{}indexing tuples, we can recollect the notion of
\emph{polymer}. If \m{m,n\in \N} are arities,
\m{\functionhead{\alpha}{n}{m}} is any indexing map and \m{f\in\Ops[n]},
then \m{\polymer{\alpha}{f}} is the operation in \m{\Ops[m]} given by
\m{\polymer{\alpha}{f}\apply{x}\defeq f\apply{x\circ \alpha}}
for \m{x\in \CarrierSet[m]}. Any operation derived from \m{f\in\Ops[n]} by
some map \m{\functionhead{\alpha}{n}{m}}, \m{m\in\N}, is said to be a
\emph{polymer} of~\m{f}. Clearly any polymer of~\m{f} can be obtained by
composition with a suitable tupling of projections:
\m{\polymer{\alpha}{f} = f\circ\apply{\eni[m]{\alpha\apply{i}}}_{i\in n}}.
\par

Besides operations we shall also need \emph{relations}: for \m{m \in \N}
any subset \m{\rho \subs \CarrierSet[m]} of \nbd{\m{m}}tuples is an
\emph{\nbd{\m{m}}ary relation} on~\m{\CarrierSet}. Thus
\m{\powerset{\CarrierSet[m]}} is the set of all
\emph{\nbd{\m{m}}ary relations}, and, again allowing arity equal to null,
the set of \emph{all finitary relations} is defined by
\m{\Rels \defeq \bigcup_{\ell \in \N} \powerset{\CarrierSet[\ell]}}.
If \m{Q\subs\Rels}, we use
\m{\Fn[m]{Q}\defeq Q\cap \powerset{\CarrierSet[m]}} to
denote its \emph{\nbdd{m}ary part}.
Moreover, if \m{\OP\colon S \to \powerset{\Rels}} is an operator on a
set~\m{S}, we put
\m{\functionhead{\Fn[m]{\OP}}{S}{\powerset{\Rels[m]}}}, mapping
\m{s\in S} to  \m{\Fn[m]{\OP}(s) \defeq \Fn[m]{\apply{\OP(s)}}}.
Similarly as for operations, for \m{s\in S} we define
\m{\Fn[{\leq m}]{\OP}(s)\defeq \bigcup_{k=0}^{m} \Fn[k]{\OP}(s)},
\m{\Fn[{\geq m}]{\OP}(s)\defeq
\bigcup_{k\in\N,k\geq m}\Fn[k]{\OP}(s)}, and we let
\m{\Fn[{>m-1}]{\OP}\defeq \Fn[{\geq m}]{\OP}}.
\par

A \emph{relation pair} of arity \m{m\in\N} (\cite[p.~15]{Harnau8Habil}
or~\cite[p.~11]{Harnau7BVerallgRelationenbegriffII}) is any pair
\m{\apply{\rho,\rho'}}, where \m{\rho,\rho'\in\Rels[m]} and
\m{\rho'\subs\rho}. We collect all \nbdd{m}ary relation pairs in the set
\m{\Relps[m]}; the disjoint union (the importance of this
technical aspect is
discussed on page~\pageref{page:importance-of-disj-union})
\m{\Relps\defeq \biguplus_{\ell\in\N}\Relps[\ell]} denotes
the set of \emph{all finitary relation pairs}. As before, we abbreviate
\nbdd{m}ary parts as \m{\Fn[m]{Q}\defeq {\Relps[m]}\cap Q} for any
\m{Q\subs\Relps} and define operator restrictions
\m{\functionhead{\Fn[m]{\OP}}{S}{\powerset{\Relps[m]}}} by mapping
\m{s\in S} to  \m{\Fn[m]{\OP}(s) \defeq \Fn[m]{\apply{\OP(s)}}} for any
\m{\OP\colon S \to \powerset{\Relps}}.
Further we put
\m{\Fn[{\leq m}]{\OP}(s)\defeq \biguplus_{k=0}^{m}\Fn[k]{\OP}(s)}
for \m{s\in S}.
\par

There is a natural order relation on \m{\Relps[m]} for each \m{m\in\N},
which is given by set inclusion in both components. That is, we write
\m{\apply{\sigma,\sigma'}\leq \apply{\rho,\rho'}} for
\m{\apply{\rho,\rho'},\apply{\sigma,\sigma'}\in \Relps[m]} if and only if
\m{\sigma\subs \rho} and \m{\sigma'\subs\rho'}. Moreover, we shall need the
quasiorder on \m{\Relps[m]}, \m{m\in\N}, that is specified by just ordering
the first components: \m{\apply{\sigma,\sigma'}\qleq\apply{\rho,\rho'}}
holds by definition if and only if \m{\sigma\subs\rho}.

We say that a relation pair \m{\apply{\sigma,\sigma'}\in\Relps} is a
\emph{relaxation of} some other pair \m{\apply{\rho,\rho'}\in\Relps}
(cf.~\cite[p.~153]{%
CouceiroFoldesClosedSetsRelConstrFuncsClosedWrtVarSubst}) if
\m{\rho'\subs\sigma'} and \m{\sigma\subs\rho}. A collection
\m{Q\subs\Relps} is \emph{closed \wrt\ relaxations} if with each pair
\m{\apply{\rho,\rho'}\in Q} it also contains any of its relaxations,
\ie\ if
\m{\enc{Q}\defeq
   \lset{\apply{\sigma,\sigma'}\in\Relps}{%
         \exists\apply{\rho,\rho'}\in Q\colon
         \rho'\subs\sigma'\subs\sigma\subs\rho}}
is a subset of (equal to)~\m{Q}. Since set inclusion is transitive, the
collection \m{\enc{Q}} is the least subset of \m{\Relps} (\wrt\ \m{\subs})
that contains~\m{Q} and is closed \wrt\ relaxations. We call \m{\enc{Q}}
the \emph{closure} of~\m{Q} \emph{\wrt\ relaxation}.
In~\cite[Definition~1, p.~16]{Harnau7BVerallgRelationenbegriffII} the
closure \wrt\ relaxation has been handled by so\dash{}called
multioperations~\m{d_{\mathrm{v}}} and~\m{d_{\mathrm{h}}}.
\par

\subsection{The \name{Galois} correspondence
\texorpdfstring{\m{\PolpOp\text{-}\InvpOp}}{Polp-Invp}}%
\label{subsect:Polp-Invp}

Here we recall the \name{Galois} connection \m{\PolpOp\text{-}\InvpOp} as
defined in~\cite[p.~15]{Harnau8Habil}
and~\cite[p.~11]{Harnau7BVerallgRelationenbegriffII}. The formulation is
identical except for extending the scope by allowing nullary operations
and relations.
\par

\begin{definition}\label{def:preservation}
For an \nbd{\m{n}}ary operation \m{f \in \Ops[n]} (\m{n\in\N}) and an
\nbd{\m{m}}ary relation pair \m{\apply{\rho,\rho'} \in \Relps[m]}
(\m{m \in \N}) on a set~\m{\CarrierSet}, we say that \emph{\m{f} preserves
\m{\apply{\rho,\rho'}}} and write \m{f \preserves \apply{\rho,\rho'}} if
the following equivalent conditions hold:
\begin{enumerate}[(i)]
\item\label{cond:preservation-tupling}
      For every tuple \m{\mathbf{r} \in \rho^n}, the composition of~\m{f}
      with the tupling \m{\apply{\mathbf{r}}} of the tuples
      in~\m{\mathbf{r}} belongs to the smaller relation:
      \m{\composition{f}{\mathbf{r}}\in \rho'}.
\item\label{cond:preservation} For every \nbd{\m{\apply{m \times n}}}matrix
      \m{X \in \CarrierSet[{m \times n}]} the columns~\m{X_{-,j}}
      (\m{j \in n}) of which are tuples in~\m{\rho}, the tuple
      \m{\apply{f(X_{i,-})}_{i \in m}} obtained by row\dash{}wise
      application of~\m{f} to~\m{X} yields a tuple of \m{\rho'}.
\end{enumerate}
\end{definition}

Note in this respect that for any tuple
\m{\mathbf{r}=\apply{r_j}_{0\leq j< n}\in\apply{\CarrierSet[m]}^n} where
for \m{0\leq j< n} each tuple is given as
\m{r_j = \apply{r_{ij}}_{0\leq i< m}},
the definition of tupling precisely yields that
\m{f\circ \apply{\mathbf{r}} =
               \apply{f\apply{\apply{r_{ij}}_{0\leq j<n}}}_{0\leq i<m}},
\ie\ the result of applying~\m{f} row\dash{}wise to the matrix
\m{\apply{r_{ij}}_{\apply{i,j}\in m\times n}\in\CarrierSet[{m\times n}]}.
\par

Note furthermore, that for \m{\rho\in\Rels} and \m{f\in\Ops} the condition
\m{f\preserves \apply{\rho,\rho}} coincides with the usual preservation
condition for functions and relations
(cf.~\cite[Definition~2.3]{BehClonesWithNullaryOperations} for the
framework involving nullary operations).
\par

Based on the preservation condition we introduce a \name{Galois}
correspondence in the usual way: for a set \m{F\subs\Ops} we denote by
\begin{align*}
\Invp{F}&\defeq
      \lset{\apply{\rho,\rho'}\in\Relps}{\forall f\in F\colon
                                         f\preserves\apply{\rho,\rho'}}\\
\intertext{the set of its \emph{invariant relation pairs},
           and, dually, for \m{Q\subs\Rels}, the set}
\Polp{Q} &\defeq\lset{f\in\Ops}{\forall \apply{\rho,\rho'}\in Q\colon
                                  f\preserves\apply{\rho,\rho'}}
\end{align*}
contains all \emph{polymorphisms of relation pairs} in~\m{Q}. The pair
\m{\apply{\Polp{},\Invp{}}} forms the \name{Galois} correspondence
\m{\PolpOp\text{-}\InvpOp}.
\par
If we restrict the latter just to relation pairs \m{\apply{\rho,\rho'}} where
\m{\rho=\rho'}, then we get the standard \name{Galois} connection
\m{\PolOp\text{-}\InvOp{}}: for \m{F\subs\Ops} we have
\begin{equation*}
\lset{\apply{\rho,\rho}}{%
      \rho\in\Rels\land\,\forall f\in F\colon f\preserves\apply{\rho,\rho}}
=\lset{\apply{\rho,\rho}}{\rho\in\Inv{F}},
\end{equation*}
where
\m{\Inv{F} = \lset{\rho\in\Rels}{\forall f\in F\colon f\preserves \rho}};
and for \m{Q\subs\lset{\apply{\rho,\rho}}{\rho\in\Rels}}, letting
\m{Q'\defeq\lset{\rho\in\Rels}{\apply{\rho,\rho}\in Q}}, it is the case that
\begin{equation*}
\Polp{Q} = \lset{f\in\Ops}{\forall \apply{\rho,\rho}\in Q\colon
                            f\preserves\apply{\rho,\rho}}\\
=\Pol{Q'},
\end{equation*}
wherein
\m{\Pol{Q'}\defeq\lset{f\in\Ops}{\forall\rho\in Q'\colon f\preserves \rho}}.
\par

The name \emph{polymorphism} attributed to the functions in \m{\Pol{Q}} for
sets of relations \m{Q\subs\Rels} comes from the fact that an operation
\m{f\in\Ops} belongs to \m{\Pol{Q}} if and only if it is a homomorphism
from the power \m{\RelAlgebra{A}^{\arity\apply{f}}} into the
relational structure \m{\RelAlgebra[\apply{\rho}_{\rho\in Q}]{A}}.
This characterisation can be generalised in the following way.
\begin{lemma}\label{lem:char-Polp-hom}
For \m{Q\subs\Relps} and any arity \m{n\in\N} an operation \m{f\in\Ops[n]}
satisfies \m{f\in\Polp{Q}} if and only if
\m{\functionhead{f}{\algwops{A}{\apply{\rho}_{\apply{\rho,\rho'}\in Q}}^{n}%
                  }{\algwops{A}{\apply{\rho'}_{\apply{\rho,\rho'}\in Q}}}}
is a homomorphism of relational structures.
\end{lemma}
The proof is a straightforward rewriting of the definitions and is therefore
omitted.
\par

It is an evident consequence of the definition of preservation that sets of
the form \m{\Invp{F}}, \m{F\subs\Ops}, are closed \wrt\ relaxation
(cf.~\cite[{Lemma~8, p.~16}]{Harnau7BVerallgRelationenbegriffII}).
\begin{lemma}\label{lem:closure-encirclement}
For \m{F\subs\Ops} we have \m{\Invp{F}=\enc{\!\Invp{F}}}.
\end{lemma}
\begin{easyproof}
Although the statement easily follows from the definition, we present here
another argument, based on the local closure \m{\LOC{}} of sets of
relation pairs (see Definition~\ref{def:s-local-closure}). For any
\m{Q\subs\Relps} we have \m{Q\subs\enc{Q}}, and we shall prove
\m{\enc{Q}\subs\LOC{Q}} in Corollary~\ref{cor:relaxation-subs-LOC}. Hence,
we get the inclusions
\m{\Invp{F}\subs\enc{\!\Invp{F}}\subs\LOC{\Invp{F}}}, and we shall
see in Corollary~\ref{cor:Invp-LOC} that \m{\LOC{\Invp{F}}=\Invp{F}}.
\end{easyproof}
\par

The following result provides a simple reformulation of the previous lemma.
\begin{corollary}\label{cor:closure-encirclement}
For every \m{Q\subs\Relps} we have \m{\Polp{Q} = \Polp{\enc{Q}}}.
\end{corollary}
\begin{proof}
By Lemma~\ref{lem:closure-encirclement}, %
\m{Q\subs\enc{Q}\subs\enc{\!\Invp{\Polp{Q}}} =\Invp{\Polp{Q}}} for
\m{Q\subs\Relps}, so
\m{\Polp{Q}\sups\Polp{\enc{Q}}\sups\Polp{\Invp{\Polp{Q}}}=\Polp{Q}}.
\ovflhbx{1.3em}%
\end{proof}

Some relation pairs are preserved by no operation. If they are part of a
set \m{Q\subs\Relps} or, more generally, part of \m{\Invp{\Polp{Q}}}, then
\m{\Polp{Q}} is forced to be empty. The next lemma characterises when this
happens (cf.~\cite[p.~15]{Harnau8Habil}
and~\cite[p.~12]{Harnau7BVerallgRelationenbegriffII}).
\begin{lemma}\label{lem:char-empty-Polp}
For \m{Q\subs\Relps} we have \m{\Polp{Q}=\emptyset} if and only if
\m{\Invp{\Polp{Q}}} contains a relation pair of the form
\m{\apply{\rho,\emptyset}} with \m{\rho\neq\emptyset}, which happens
precisely if \m{\apply{\CarrierSet[0],\emptyset}\in\Invp{\Polp{Q}}}.
\end{lemma}
\begin{proof}
If \m{\rho\in\Relps} is non\dash{}empty, then also \m{\rho^n\neq\emptyset}
for any possible \m{n\in\N}. Therefore, the condition in
Definition~\ref{def:preservation}\eqref{cond:preservation-tupling}
with \m{\rho'=\emptyset} is not satisfiable for any function \m{f\in\Ops}.
Hence, \m{\Polp{Q}= \Polp{\Invp{\Polp{Q}}} =\emptyset}, whenever
\m{\apply{\rho,\emptyset}\in\Invp{\Polp{Q}}}.\par
Conversely, if \m{\Polp{Q} =\emptyset},
then \m{\Invp{\Polp{Q}} = \Invp{\emptyset}=\Relps}, which clearly contains
the relation pair \m{\apply{\CarrierSet[0],\emptyset}}. The nullary
relation~\m{\CarrierSet[0]} is never empty, even for
\m{\CarrierSet=\emptyset}, so the exhibited example is of the right form.
\end{proof}

\begin{remark}\label{rem:nullary-rels}
The previous lemma demonstrates the necessity to include nullary relations
in the framework, caused by our wish not to impose any restriction on the
carrier set~\m{\CarrierSet}. Namely, for \m{\CarrierSet=\emptyset}, we have
\m{\CarrierSet[m]=\emptyset} for all \m{m\in\Np}, and thus
\m{\Rels[m]=\powerset{\CarrierSet[m]}=\set{\emptyset}}. Hence,
\m{\Relps=\set{\apply{\CarrierSet[0],\emptyset}}\uplus\biguplus_{m\in\Np}
               \set{\apply{\emptyset,\emptyset}}},
which allows us to distinguish between
\m{\Polp{\Relps} = \Polp{\set{\apply{\CarrierSet[0],\emptyset}}}=\emptyset}
and \m{\Polp{\set{\apply{\emptyset,\emptyset}}} = \Polp{\emptyset} = \Ops}.
Both sets are evidently semiclones (subalgebras of the iterative
\name{Post} algebra), on any carrier set~\m{\CarrierSet}, so, in view of
our overall objective, it is more than desirable
to be able to model them with our \name{Galois} correspondence.
Restricting to relations of positive arity, this would clearly be
impossible for \m{\CarrierSet=\emptyset}.
\end{remark}

\subsection{Local closure operators for functions and relation pairs}%
\label{subsect:local-closures}
For the \name{Galois} connection \m{\PolOp\text{-}\InvOp} in the case of
infinite carrier sets, there exist examples \m{F\subs\Ops} where the
inclusion \m{\genClone{F}\subs\Pol{\Inv{F}}} is proper. Hence, in order to
characterise the \name{Galois} closure, an additional local closure
operator is needed. A similar situation arises with
\m{\PolpOp\text{-}\InvpOp}:
for operations we can indeed reuse the same local closure operators as
known from \m{\PolOp\text{-}\InvOp}. For the side of relation pairs, we
have to introduce a new variant of local closure.
\par

In fact, in order to characterise \name{Galois} closures of sets of at most
\nbdd{s}ary operations / relations (\m{s\in\N}) we define more specific
variants of \nbdd{s}local closure operators. Note that apart from
extending the scope of the definition to \m{s=0} and nullary operations,
the operators \m{\sLoc{}} and \m{\Loc{}} we define coincide with those
from~\cite[1.9, p.~15]{PoeGeneralGaloisTheoryForOperationsAndRelations}
(see also~\cite[1.5, p.~255 et seq.]{%
PoeConcreteRepresentationOfAlgebraicStructuresAndGeneralGaloisTheory}).
\par
\begin{definition}\label{def:s-local-closure}
For \m{s\in\N}, \m{F\subs\Ops} and \m{Q\subs\Relps} we set
\begin{align*}
\sLoc{F}&
\defeq \biguplus_{n\in\N}\lset{g\in\Ops[n]}{%
            \forall B\subs\CarrierSet[n], \abs{B}\leq s\,
            \exists\,f\in \Fn{F}\colon
                    g\Restriction_B = f\Restriction_B},\\
\Loc{F}&
\defeq \bigcap_{s\in\N}\sLoc{F},\\
\sLOC{Q}&
\defeq \biguplus_{m\in\N}\lset{\apply{\sigma,\sigma'}\in\Relps[m]}{%
            \begin{aligned}[c]
            \forall B\subs\sigma,\abs{B}\leq s\,
            &\exists\apply{\rho,\rho'}\in \Fn[m]{Q}\colon\\
                   &B\subs \rho\land \rho'\subs\sigma'
            \end{aligned}},\\
\LOC{Q}&
\defeq\bigcap_{s\in\N}\sLOC{Q},
\end{align*}
and call these \emph{\nbdd{s}local} and \emph{local closure operators},
respectively.
\end{definition}
\par

It is easy to check that \m{\sLoc{}}, \m{\Loc{}}, \m{\sLOC{}} and
\m{\LOC{}} are indeed closure operators on the sets of finitary operations
and relation pairs, respectively.
Likewise, it is not hard to see that for every \m{s,n\in\N} we have
\m{\sLoc[s][n]{F} = \sLoc{\apply{\Fn{F}}}} and
\m{\Loc[n]{F}=\Loc{\apply{\Fn{F}}}}
for \m{F\subs\Ops}, and similarly we have
\m{\sLOC[s][n]{Q} = \sLOC{\apply{\Fn{Q}}}} and
\m{\LOC[n]{Q}=\LOC{\apply{\Fn{Q}}}}
for any set \m{Q\subs\Relps}.
To make a technical remark:\label{page:importance-of-disj-union}
if we had not insisted on using the disjoint
union for the definition of \m{\Relps}, then for any \m{n\in\N} we would
have \m{\sLOC[0]{\apply{\Fn[n]{Q}}}=\Relps} whenever
\m{\apply{\emptyset,\emptyset}\in Q} (as in this case
\m{\apply{\emptyset,\emptyset}\in \Fn[m]{{\Fn{Q}}}} were true for all
\m{m\in\N}), and this would obviously violate the equality mentioned
above:  \m{\Relps = \sLOC[0]{\apply{\Fn{Q}}} \not\subseteq
           \sLOC[0][n]{Q} \subs \Relps[n]}.
\par

Moreover, it follows directly from the
definition that \m{\sLoc[t]{F}\subs\sLoc[s]{F}} and
\m{\sLOC[t]{Q}\subs\sLOC{Q}} hold for all \m{F\subs\Ops} and
\m{Q\subs\Relps} whenever \m{s\leq t}, \m{s,t\in\N}.
Therefore, for \m{F\subs\Ops}, \m{Q\subs\Relps} and \m{s\in\N} we have the
inclusions
\begin{alignat*}{12}
\sLoc[0]{F}&\sups\dotsm&&\sups \sLoc{F}&&\sups \sLoc[{\apply{s+1}}]{F}
&&\sups\dotsm&&\sups \Loc{F}&&\sups F,\\
\sLOC[0]{Q}&\sups\dotsm&&\sups \sLOC{Q}&&\sups \sLOC[{\apply{s+1}}]{Q}
&&\sups\dotsm&&\sups \LOC{Q}&&\sups Q.
\end{alignat*}
\par

It follows from these relations that
\begin{align*}
\sLoc{\sLoc[t]{F}} &= \sLoc[{\apply{\min\set{s,t}}}]{F}
\intertext{holds for all \m{F\subs\Ops} and}
\sLOC{\sLOC[t]{Q}} &= \sLOC[{\apply{\min\set{s,t}}}]{Q}
\end{align*}
for all \m{Q\subs\Relps} and any
\m{s,t\in\N\cup\set{\infty}}
(cp.~\cite[Proposition~1.10, p.~16]{%
PoeGeneralGaloisTheoryForOperationsAndRelations}),
where we have temporarily put
\m{\sLoc[{\infty}]{}\defeq\Loc{}} and \m{\sLOC[{\infty}]{}\defeq\LOC{}}.
\par

Note that our definition of \nbdd{s}local closure of relation pairs for
\m{s\in\Np} entails the corresponding one for relations given
in~\cite[1.9, p.~16]{PoeGeneralGaloisTheoryForOperationsAndRelations} in
the following way: for \m{Q'\subs\Rels\setminus\Rels[0]} put
\m{Q\defeq\biguplus_{m\in\Np}
                          \lset{\apply{\rho,\rho}}{\rho\in \Fn[m]{Q'}}}.
Then given \m{s>0}, one can check that
\m{\sLOC{Q}=\biguplus_{m\in\Np}
                \lset{\apply{\sigma,\sigma}}{\sigma\in\sLOC[s][m]{Q'}}}
(see Lemma~\ref{lem:sLOC-vs-sLOC} for further details),
in which \m{\sLOC{Q'}} denotes the set
\m{\lset{\sigma\in\Rels}{\forall B\subs\sigma,\abs{B}\leq s\,
\exists\rho\in Q'\colon B\subs\rho\subs\sigma}}\label{page:sLOC-rel}.
Hence one may reconstruct \m{\sLOC[s]{Q'}} as
\m{\lset{\sigma\in\Rels}{\apply{\sigma,\sigma}\in\sLOC{Q}}}. The local
closure \m{\LOC{}} of sets of non\dash{}nullary relations can be
handled in a similar way.
\par

Furthermore, the following characterisation is also simple to verify.
\begin{lemma}\label{lem:char-loc-by-finite-interpolation}
For any set~\m{A}, collections \m{F\subs\Ops} and \m{Q\subs\Relps} we have
\begin{align*}
\Loc{F}&
= \biguplus_{n\in\N}\lset{g\in\Ops[n]}{%
            \forall B\subs\CarrierSet[n], \abs{B}<\aleph_{0}\,
            \exists\,f\in \Fn[n]{F}\colon
                    g\Restriction_B = f\Restriction_B},\\
\LOC{Q}&
= \biguplus_{m\in\N}\lset{\apply{\sigma,\sigma'}\in\Relps[m]}{%
            \begin{aligned}[c]
            \forall B\subs\sigma,\abs{B}<\aleph_{0}\,
            &\exists\apply{\rho,\rho'}\in \Fn[m]{Q}\colon\\
                   &B\subs \rho\land \rho'\subs\sigma'
            \end{aligned}}.
\end{align*}
\end{lemma}
\par
From this result it easily follows that closure \wrt\ relaxation is just a
special case of the local closure of relation pairs.
\begin{corollary}\label{cor:relaxation-subs-LOC}
For \m{Q\subs\Relps} we have \m{\enc{Q}\subs\LOC{Q}}.
\end{corollary}
\begin{straightforwardproof}
Let \m{\apply{\sigma,\sigma'}} be a relaxation of some pair
\m{\apply{\rho,\rho'}\in\Fn[m]{Q}} for some \m{m\in\N}, \ie\
\m{\rho'\subs\sigma'\subs\sigma\subs\rho}. For any finite subset
\m{B\subs\sigma} we evidently have \m{B\subs\sigma\subs\rho} and
\m{\rho'\subs\sigma'} for the relation pair \m{\apply{\rho,\rho'}\in
\Fn[m]{Q}}. Thus, according to
Lemma~\ref{lem:char-loc-by-finite-interpolation}, it is the case that
\m{\apply{\sigma,\sigma'}\in\LOC{Q}}.
\end{straightforwardproof}
\par

The following consequence is now evident.
\begin{corollary}\label{cor:locally-closed-implies-relaxation-closed}
Let \m{Q\subs\Relps} be locally closed (or even \nbdd{s}locally closed for
some \m{s\in\N}), then it is closed \wrt\ relaxation.
\end{corollary}
\par

For relations of fixed arity and finite base sets, there is even a much
stronger connection between relaxation and \nbdd{s}local closure:
\begin{lemma}\label{lem:enc-Qm=km-LOCQ}
For all finite carrier sets~\m{\CarrierSet} of cardinality
\m{k\defeq\abs{A}<\aleph_{0}} and any \m{m\in\N}, we have
\m{\enc{Q}=\LOC{Q} = \sLOC[{k^m}]{Q}} for all \m{Q\subs\Relps[m]}.
\end{lemma}
\begin{proof}
The inclusions \m{\enc{Q}\subs\LOC{Q}\subs\sLOC[{k^m}]{Q}} hold in
general (cf.~Corollary~\ref{cor:relaxation-subs-LOC}). Conversely, for
a set \m{Q\subs\Relps[m]} of \nbdd{m}ary pairs, let us consider any
\m{\apply{\sigma,\sigma'}\in\sLOC[{k^m}]{Q}
    =\sLOC[{k^m}]{\apply{\Fn[m]{Q}}} = \sLOC[{k^m}][m]{Q}}.
As \m{\sigma\in\Rels[m]}, we have
\m{\abs{\sigma}\leq\abs{\CarrierSet[m]}=k^m}. Hence, taking
\m{B\defeq\sigma} as a subset of~\m{\sigma} having at most~\m{k^m}
elements, by definition of \m{\sLOC[{k^m}]{}}, we get a pair
\m{\apply{\rho,\rho'}\in\Fn[m]{Q}} such that \m{\sigma = B \subs\rho} and
\m{\rho'\subs\sigma'}. Therefore, \m{\apply{\sigma,\sigma'}\in\enc{Q}}.
\end{proof}
\par

In particular, in case of finite carrier sets, the inclusion in
Corollary~\ref{cor:relaxation-subs-LOC} is always an equality.
\begin{corollary}\label{cor:finite=>relaxation=LOC}
For finite~\m{\CarrierSet} we have \m{\enc{Q}=\LOC{Q}} for all
\m{Q\subs\Relps}; in particular, a subset \m{Q\subs\Relps} is locally
closed if and only if it is closed \wrt\ relaxation.
\end{corollary}
\begin{proof}
The set
\m{\LOC{Q}=\biguplus_{m\in\N}\LOC[m]{Q} =
\biguplus_{m\in\N}\LOC{\apply{\Fn[m]{Q}}}}
is equal to
\m{\biguplus_{m\in\N}\enc{\Fn[m]{Q}}=\enc{Q}}
upon application of Lemma~\ref{lem:enc-Qm=km-LOCQ}.
\end{proof}
\par

The following closure property will become important regarding the
characterisation of the closure operator \m{\Invp{\Polp[\leq s]{}}} in
Section~\ref{sect:char-closures}. For \m{s\in \N} a
collection \m{\mathcal{T}\subs \powerset{S}} of subsets of a
set~\m{S} is called \emph{\nbdd{s}directed} if and only if
for all \m{t\leq s}, all \m{\apply{X_{i}}_{i\in t}\in \mathcal{T}^{t}}
and every \m{\mathbf{r} = \apply{r_{i}}_{i\in t}\in \prod_{i\in t}X_i}
there is a set \m{Z\in\mathcal{T}}
such that \m{\im{\mathbf{r}}=\lset{r_{i}}{i\in t}\subs Z}.
Clearly, this condition is equivalent to~\m{\mathcal{T}} being
non\dash{}empty and that for all
\m{\apply{X_{i}}_{i\in s}\in\mathcal{T}^{s}} and
\m{\mathbf{r}\in\prod_{i\in s}X_i} there exists \m{Z\in\mathcal{T}}
fulfilling \m{\im{\mathbf{r}}\subs Z}.
We say that a set
\m{Q\subs\Relps[m]} of \nbdd{m}ary relation pairs is
\emph{\nbdd{s}directed} if and only if
\m{\lset{\rho}{\apply{\rho,\rho'}\in Q}\subs\powerset{\CarrierSet[m]}}
is \nbdd{s}directed in the sense above.
We prove now that sets of the form \m{\sLOC{Q}}, where \m{Q\subs\Relps},
are closed \wrt\ unions of \nbdd{s}directed systems of relation pairs of
the same arity.
\begin{lemma}\label{lem:s-directed-unions}
If \m{s,m\in\N}, \m{Q\subs\Relps}, and
\m{\mathcal{T}\subs \sLOC[s][m]{Q}}
is \nbdd{s}directed, then we have
\m{\bigcup \mathcal{T}\defeq
\apply{\bigcup_{\apply{\mu,\mu'}\in \mathcal{T}}\mu,
       \bigcup_{\apply{\mu,\mu'}\in \mathcal{T}}\mu'}
\in\sLOC[s][m]{Q}}.
\end{lemma}
\begin{straightforwardproof}
Clearly, we have
\m{\sigma'\defeq\bigcup_{\apply{\mu,\mu'}\in\mathcal{T}}\mu'
          \subs\bigcup_{\apply{\mu,\mu'}\in\mathcal{T}}\mu
          \eqdef\sigma},
so the union \m{\apply{\sigma,\sigma'}} is a well\dash{}defined relation
pair in \m{\Relps[m]}. In order to prove that it belongs to \m{\sLOC{Q}},
we consider any subset \m{B= \lset{b_{i}}{i\in t}\subs\sigma} such that
\m{t\defeq \abs{B}\leq s}.
By definition of~\m{\sigma}, for each \m{i\in t} there exists a pair
\m{\apply{\mu_i,\mu'_i}\in\mathcal{T}} such that \m{b_i\in\mu_i}.
By \nbdd{s}directedness of~\m{\mathcal{T}} there exists some
\m{\apply{\mu,\mu'}\in \mathcal{T}} such that
\m{B=\lset{b_{i}}{i\in t}\subs\mu}.
For \m{\apply{\mu,\mu'}\in \mathcal{T}\subs\sLOC[s][m]{Q}} and
\m{B\subs\mu} has at most~\m{s} elements, there must exist some
pair \m{\apply{\rho,\rho'}\in\Fn[m]{Q}} such that
\m{B\subs\rho} and \m{\rho'\subs\mu'\subs\sigma'}. This shows
that \m{\apply{\sigma,\sigma'}\in\sLOC[s][m]{Q}}.
\end{straightforwardproof}
\par

For \m{m\in\N} and we say that a set \m{\mathcal{T}\subs\Relps[m]} is
\emph{\nbdd{\aleph_{0}}directed} if it is \nbdd{s}directed for all \m{s\in\N}.
This means we require the condition presented before
Lemma~\ref{lem:s-directed-unions} to hold for any finite sequence of
relations and tuples.
\par
We call a set \m{\mathcal{T}\subs\Relps[m]}
\emph{directed} if \m{\mathcal{T}\neq\emptyset} and for all
\m{\apply{\rho_1,\rho'_1},\apply{\rho_2,\rho'_2}\in\mathcal{T}} there
exists some \m{\apply{\rho,\rho'}\in\mathcal{T}} such that \m{\rho_1\cup\rho_2\subs\rho}. This is
equivalent to saying that for any finite subset
\m{\mathcal{F}\subs \mathcal{T}} there is an
pair \m{\apply{\rho,\rho'}\in \mathcal{T}} such that
\m{\bigcup_{\apply{\mu,\mu'}\in \mathcal{F}}\mu\subs\rho},
wherefore directedness clearly implies
\nbdd{\aleph_{0}}directedness.\label{page:infty-directedness}
\par

As a consequence of this implication we get that locally closed sets of relation
pairs are closed under directed unions of sets of pairs of identical
arity.
\begin{corollary}\label{cor:inf-directed-union}
For all \m{m\in\N}, \m{Q\subs\Relps} and every
\nbdd{\aleph_{0}}directed collection \m{\mathcal{T}\subs\LOC[m]{Q}},
we have
\m{\bigcup \mathcal{T}\defeq
\apply{\bigcup_{\apply{\mu,\mu'}\in \mathcal{T}}\mu,
       \bigcup_{\apply{\mu,\mu'}\in \mathcal{T}}\mu'}
\in\LOC[m]{Q}}.
In particular this is true whenever \m{\mathcal{T}\subs\LOC[m]{Q}} is
directed.
\end{corollary}
\par

Under additional assumptions on the set of relation pairs~\m{Q} we shall
extend Lemma~\ref{lem:s-directed-unions} and
Corollary~\ref{cor:inf-directed-union} to characterisations of local and
\nbdd{s}local closedness. We conclude this subsection with remarks on the
relationship of our local closure operators to others defined in the
more general setting treated
in~\cite{CouceiroGaloisConnectionsExternalOperationsRelationalConstraints}.
\par

\begin{remark}\label{rem:local-closures}
The local closure operators (and \nbdd{s}local closure operators for
\m{s\in\Np}) defined here cannot directly be derived as special cases of
the corresponding closure operators from~\cite{%
CouceiroGaloisConnectionsExternalOperationsRelationalConstraints}. As the
case of local closures is similar, we shall only argue for \nbdd{s}local
closures. Specialising the framework in the mentioned article for a pair
of carrier sets \m{\apply{\CarrierSet, B}} where \m{B=A}, we may apply
the \nbdd{s}local closure \m{\LOOP_{s}} described there to any set
\m{Q\subs\Relps\setminus\Relps[0]}, then yielding the collection
\begin{multline*}
\LO{Q} = Q\cup{}\\
\bigcup_{m\in\Np}\!\!
      \lset{\apply{R,S}\in\apply{\powerset{\CarrierSet[m]}}^{2}}{%
                 \forall C\subs R, \abs{C}\leq s\,
                 \forall \CarrierSet[m]\sups T\sups S\colon
                 \apply{C,T}\in Q}.
\end{multline*}
As this set contains pairs \m{\apply{R,S}} that are not relation pairs,
\ie\ failing the condition \m{R\sups S}, the
canonical modification would be to simply intersect the result with
\m{\Relps}, leading to
\begin{multline*}
\LO{Q}\cap\Relps =Q \cup{}\\
\bigcup_{m\in\Np}\!\!
      \lset{\apply{R,S}\in\Relps[m]}{%
                 \forall C\subs R, \abs{C}\leq s\,
                 \forall \CarrierSet[m]\sups T\sups S\colon
                 \apply{C,T}\in Q}.
\end{multline*}
This set equals~\m{Q} on any set~\m{\CarrierSet} (in fact, the second part of
the union is empty, whenever \m{\CarrierSet\neq\emptyset}, as for
\m{C=\emptyset} and \m{T=\CarrierSet[m]\neq\emptyset} the condition
\m{\apply{C,T}\in Q} is never satisfied).
So the original definition of \m{\LOOP_{s}} (or its canonical
modification) is not helpful at all in our setting.
\par

Suppose, in the union over \m{m\in\Np}, we change the condition
describing when a relation pair \m{\apply{R,S}} is added to the
\nbdd{s}local closure of~\m{Q} as follows: among all relational
constraints \m{\apply{C,T}} relaxing \m{\apply{R,S}} and verifying
\m{\abs{C}\leq s} only those are required to be in~\m{Q} that
\emph{are indeed relation pairs}.
Then we get
\begin{multline*}
Q \cup
{\bigcup_{m\in\Np}\!\!
      \lset{\apply{R,S}\in\Relps[m]}{%
                 \forall C\subs R, \abs{C}\leq s\,
                 \forall C \sups T\sups S\colon
                 \apply{C,T}\in Q}}.
\end{multline*}
This set still differs from \m{\sLOC{Q}} as defined above. For instance
for any \m{s\in\Np} and \m{Q=\emptyset} we have \m{\sLOC{Q} = \emptyset},
while the previously displayed collection contains all relation pairs
\m{\apply{R,S}\in\Relps} where \m{\abs{S}>s}.
\par

We do not see an obvious way how to translate \m{\LOOP_{s}} into
\m{\sLOC{}} or vice versa.
\end{remark}

\section{Semiclones and the full iterative \name{Post} algebra}%
\label{sect:semiclones}
The following definition is very similar to that of a clone of operations.
The only difference is that a clone \m{F\subs\Ops} is additionally required
to contain the set \m{\TrivOps} of projections as a subset.

\begin{definition}\label{def:semiclone}
A \emph{(concrete) semiclone (of operations)} on a set~\m{\CarrierSet} is a
subset \m{F \subseteq \Ops} of all finitary operations such that
for all \m{m,n\in\N}
we have \m{\composition{f}{\listen{g}{n}}\in F} for each
\m{f\in\Fn{F}} and
\m{\apply{\listen{g}{n}}\in\apply{\Fn[m]{\apply{F\cup\TrivOps}}}^n}.
\end{definition}

The closure property stated in Definition~\ref{def:semiclone} is formulated
in terms of partial composition operations on \m{\Ops} as the functions
making up the tupling all have to be of identical arity. However, it is
possible to extend these operations in a conservative way to totally
defined operations on \m{\Ops} such that semiclones are exactly the
subuniverses of a certain universal algebra on the carrier set~\m{\Ops}:
for each \m{n,m\in\N} and each subset \m{I\subs n} and any tuple
\m{\apply{g_i}_{i\in I}\in\TrivOps[m]} of \nbdd{m}ary projections we define
an \nbdd{\apply{\abs{n\setminus I}+1}}ary operation on \m{\Ops}, which maps
\m{\apply{f,\apply{g_i}_{i\in n\setminus I}}} to
\m{f\circ \apply{g_i}_{i\in n}} provided that \m{f\in\Ops[n]} and
\m{g_i\in\Ops[m]} for all \m{i\in n\setminus I}, and to~\m{f} otherwise.
If we collect all the finitary operations obtained in this way in a set
\m{\Phi\subs\Op{\Ops}}, then it becomes clear that \m{F\subs \Ops} is a
semiclone if and only if it is a subuniverse of the algebra
\m{\algwops{\Ops}{\Phi}}.
\par

Hence, the set \m{\Sclones\defeq \Sub\apply{\algwops{\Ops}{\Phi}}} of all
semiclones on~\m{\CarrierSet} bears the structure of a complete algebraic
lattice \wrt\ set\dash{}inclusion, and is, in particular, a closure system.
The corresponding closure operator will be denoted by \m{\genSclone{\,}}.
\par

Evident, trivial examples of semiclones are the empty set of operations and
any clone \m{F\subs \Ops}. Moreover, we have the following class of
examples:
\begin{lemma}\label{lem:semiclone-gen-by-transf}
For a set \m{G\subs\Ops[1]} of unary transformations, abbreviate its
generated transformation semigroup by
\m{S\defeq\gapply{G}_{\algwops{\Ops[1]}{\circ}}}.
Then we have
\[\genSclone{G} = \lset{f\circ\eni{i}}{i\in n\land n\in \Np\land f\in S}.\]
\end{lemma}
\begin{straightforwardproof}
\m{S} is obtained from~\m{G} by closure \wrt\ composition of
unary operations, which is part of the requirement in
Definition~\ref{def:semiclone}. Thus, we have \m{S\subs\genSclone{G}}, but
now the closure property clearly yields that \m{\genSclone{G}} must contain
the whole set on the right\dash{}hand side as a subset.
\par
Conversely, it is easy to check that the latter collection, first of
all, contains~\m{S} and therefore~\m{G}, and second, actually forms a
semiclone. Thus, it must be a superset of the least semiclone
containing~\m{G}, which is \m{\genSclone{G}}.
\end{straightforwardproof}

\begin{corollary}\label{cor:unary-parts=transf-sg}
The unary parts of semiclones \m{\lset{\Fn[1]{F}}{F\in \Sclones}} are
precisely all (carrier sets of) transformation semigroups
on~\m{\CarrierSet}.
\end{corollary}
\begin{easyproof}
By definition, the restriction \m{\Fn[1]{F}} of any semiclone
\m{F\in\Sclones} forms a transformation semigroup. The converse inclusion
follows from Lemma~\ref{lem:semiclone-gen-by-transf} as we have
\m{\genSclone[1]{S} = \lset{f\circ \eni[1]{0}}{f\in S} = S} for any
transformation semigroup \m{S\subs\Ops}.
\end{easyproof}

As mentioned in the introduction, semiclones are not a new invention. They
are just the subuniverses (``closed classes of functions'') of the full
iterative \name{Post} algebra. In order to see this we need a few
definitions.
\par

For \m{n\in\Np} define \m{\functionhead{\alpha^{\zeta}_n}{n}{n}} by
\m{\alpha^{\zeta}_n\apply{i}\defeq i+1 \pmod n} and
\m{\alpha^{\zeta}_0 \defeq \id_0}. Moreover, let
\m{\functionhead{\alpha^{\tau}_n}{n}{n}} be the transposition
\m{\apply{0,1}} for \m{n\in\N}, \m{n\geq 2}, and put
\m{\alpha^{\tau}_n\defeq \id_n} for \m{n\in\set{0,1}}.
We continue by defining \m{\functionhead{\alpha^{\Delta}_n}{n}{n-1}} via
\m{\alpha^{\Delta}_n\apply{i}\defeq \max\apply{0,i-1}} for \m{n\in\N},
\m{n\geq 2}, letting \m{\alpha^{\Delta}_n\defeq \id_n} for
\m{n\in\set{0,1}} and declaring the map
\m{\functionhead{\alpha^{\nabla}_n}{n}{n+1}} by
\m{\alpha^{\nabla}_n\apply{i}\defeq i+1} for any \m{n\in\N}.
\par
On this basis we define for \m{\omega\in\set{\zeta,\tau,\Delta,\nabla}} a
unary map \m{\functionhead{\omega}{\Ops}{\Ops}} by
\m{\omega\apply{f}\defeq \polymer{\alpha^{\omega}_{\arity\apply{f}}}{f}}
for \m{f\in\Ops}.
Moreover, for \m{f,g\in\Ops}, \m{n\defeq\arity(f)},
\m{m\defeq\arity(g)} we construct \m{f * g\in\Ops[k]} where
\m{k\defeq\max\apply{0,n+m-1}} as follows: if \m{n\geq 2} we put
\m{f*g\defeq\composition{f}{g\circ\apply{\eni[k]{i}}_{i\in m},
                                  \apply{\eni[k]{m+j}}_{j\in n-1}}};
for \m{n=1}, we define the product
\m{f*g\defeq f\circ \composition{g}{\eni[k]{i}}_{i\in m}} whenever
\m{m>0}, and \m{f*g\defeq f\circ g} if \m{m=0};
for \m{n=0}, we define
\m{f*g\defeq f} in case that \m{k=0}, and
\m{f*g\defeq f\circ\apply{\eni[k]{i}}_{i\in 0}} otherwise (where
\m{\apply{\eni[k]{i}}_{i\in 0}} by definition is the unique map
from~\m{\CarrierSet[k]} to~\m{\CarrierSet[0]}).
\par
In this way, we obtain an algebra
\m{\alg{\Ops}\defeq \algwops{\Ops}{\zeta,\tau,\Delta,\nabla,*}} of arity
type \m{\apply{1,1,1,1,2}} that we call \emph{full iterative \name{Post}
algebra}. It is easy to see that \m{\Ops\setminus\Ops[0]} is a subuniverse,
and the corresponding subalgebra is the one that has been introduced under
precisely the same name
in~\cite{MalcevIterativeAlgebrasAndPostVarieties1966}. The difference in
terminology is just of technical nature and shows up because we wish to
accommodate all nullary constants in our framework.
\par

The algebra \m{\alg{\Ops}} obviously is less prodigal of its fundamental
operations than \m{\algwops{\Ops}{\Phi}} introduced above. The
following lemma proves that both actually do the same job.
\begin{lemma}\label{char:semiclones=subs-of-iterative-alg}
The semiclones on~\m{\CarrierSet} are exactly the subuniverses of the full
iterative \name{Post} algebra: \m{\Sclones = \Sub\apply{\alg{\Ops}}}.
\end{lemma}
\begin{proof}
We saw earlier that any polymer can be expressed as a composition with a
tupling of projections under which semiclones are closed by definition.
Thus any semiclone is closed \wrt\ the unary operations~\m{\zeta},
\m{\tau}, \m{\Delta} and~\m{\nabla}. By construction of~\m{*},
it is also closed \wrt~\m{*}.
\par
It is a tedious, but well\dash{}known exercise (known from the proof that
clones are exactly the subuniverses of function algebras, which differ from
iterative algebras by just adding an additional constant representing a
projection) that the converse also holds:
for arities \m{m,n\in\N} and operations \m{f\in\Fn{F}} and
\m{\listen{g}{n}\in \apply{\Fn[m]{\apply{F\cup\TrivOps}}}^{n}} any
composition \m{\composition{f}{\listen{g}{n}}} can be expressed as the
result of a term operation of \m{\alg{\Ops}} applied to
\m{\apply{f,\listen{g}{n}}}. For \m{F\in\Sub\apply{\alg{\Ops}}} this means
that any such composition also has to belong to~\m{F}.
\end{proof}

The following facts on the relationship of semiclones and clones are
well\dash{}known
(see~\cite[p.~5 et seq.]{Harnau8Habil}
  or~\cite[p.~8 et seq.]{Harnau7AVerallgRelationenbegriffI},
  Lemmata~3 and~4, and Satz~1).
In this context, we recollect that
\m{\genClone{F}} denotes the least clone containing some set
\m{F\subs\Ops}, \ie\ the \emph{clone generated by~\m{F}}.
The symbol \m{\Clones} stands for the set of all clones on~\m{\CarrierSet}.

\begin{lemma}\label{lem:semiclones+proj=clones}
For any set \m{F\subs\Ops} and any \m{0\leq i<n}, \m{n\in\N},
the following assertions are true:
\begin{enumerate}[(a)]
\item\label{item:semiclone-gen-by-proj}
      \m{\genSclone{\set{\eni{i}}} = \TrivOps}.
\item\label{item:char-clone-cl=sclone-cl+proj}
      \m{\genSclone{F \cup\set{\eni{i}}} = \genSclone{F}\cup \TrivOps =
         \genClone{F}}.
\item\label{item:semiclones-all-or-nothing}
      If \m{F\in\Sclones}, then
      \m{F\cap\TrivOps\in\set{\emptyset,\TrivOps}}.
\item\label{item:char-clones=semiclones-with-proj}
      \m{\Clones = \lset{G\in\Sclones}{G\cap \TrivOps\neq\emptyset}}.
\end{enumerate}
\end{lemma}
\begin{easyproof}
Fix any set \m{F\subs\Ops} and any projection \m{\eni{i}},
where \m{0<i\leq n}, \m{n\in\Np}.
\begin{enumerate}[(a)]
\item The set of projections is a clone, and hence a semiclone. We only
      need to check that \m{\eni{i}} generates any other projection. First,
      we note that \m{\id_A = \composition{\eni{i}}{\id_A,\dotsc,\id_A}},
      so \m{\id_A\in \genSclone{\set{\eni{i}}}}. Besides, for any
      \m{m\in\Np} and \m{0\leq j<m} we have
      \m{\eni[m]{j} = \id_A\circ\eni[m]{j}}, whence we obtain the inclusion
      \m{\TrivOps\subs\genSclone{\set{\id_A}}
                 \subs\genSclone{\set{\eni{i}}}}.
\item The relation
      \m{\genSclone{F}\cup\TrivOps\subs\genSclone{F\cup\set{\eni{i}}}}
      follows from~\eqref{item:semiclone-gen-by-proj}, and the inclusion
      \m{\genSclone{F\cup\set{\eni{i}}}\subs\genClone{F}} holds as each
      clone is a semiclone. Finally, \m{\genSclone{F}\cup\TrivOps}
      contains~\m{F}, and it is easy to check that it is indeed a clone.
      Therefore, it has to contain \m{\genClone{F}} as a subset.
\item If~\m{F} is a semiclone and \m{F\cap\TrivOps\neq\emptyset}, then for
      some arity \m{n\in\N} and some \m{0\leq i<n} we have
      \m{F=\genSclone{F} = \genSclone{F\cup\set{\eni{i}}}\sups\TrivOps}
      by~\eqref{item:char-clone-cl=sclone-cl+proj}. Therefore,
      \m{F\cap\TrivOps=\TrivOps}.
\item The inclusion ``\m{\subs}'' is trivial. Conversely, any semiclone
      \m{G\in\Sclones} such that \m{G\cap\TrivOps\neq\emptyset} fulfils
      \m{\TrivOps\subs G} by~\eqref{item:semiclones-all-or-nothing}. Hence,
      by~\eqref{item:char-clone-cl=sclone-cl+proj}, one obtains that
      \m{\genClone{G} = \genSclone{G}\cup\TrivOps = G\cup\TrivOps = G},
      \ie\ that~\m{G} is a clone.\qedhere
\end{enumerate}
\end{easyproof}
\par

As a consequence of the previous lemma, we can describe those semiclones
whose unary parts yield proper transformation semigroups, \ie\ those which
are no monoids.
\begin{corollary}\label{cor:char-proper-sg}
On any set~\m{\CarrierSet} we have
\begin{equation*}
\lset{\Fn[1]{F}}{F\in\Sclones\setminus\Clones}
=\lset{S\subs\Ops[1]\setminus\set{\id_{\CarrierSet}}}{\apply{S,\circ}
\text{is a semigroup}}.
\end{equation*}
\end{corollary}
\begin{easyproof}
If \m{F\in\Sclones\setminus\Clones}, then \m{\Fn[1]{F}} is a carrier set
of a transformation semigroup by Corollary~\ref{cor:unary-parts=transf-sg}.
Since \m{F\in\Sclones\setminus\Clones}, we have
\m{F\cap\TrivOps=\emptyset} by Lemma~\ref{lem:semiclones+proj=clones}%
                        \eqref{item:char-clones=semiclones-with-proj},
and so \m{\Fn[1]{F}\subs\Ops[1]\setminus\set{\id_{\CarrierSet}}}.
Conversely, if \m{S\subs\Ops[1]\setminus\set{\id_{\CarrierSet}}} is a
carrier set of a proper transformation semigroup, then, by
Corollary~\ref{cor:unary-parts=transf-sg}, there exists some
\m{F\in\Sclones} such that \m{S=\Fn[1]{F}}. If \m{F\in\Clones}, then we
would have \m{\id_{\CarrierSet}\in\Fn[1]{F}=S}, violating our assumption.
Hence, \m{F\in\Sclones\setminus\Clones}.
\end{easyproof}
\par

The \name{Galois} correspondence \m{\PolpOp\text{-}\InvpOp} gives us plenty
of examples of semiclones
(cf.~\cite[{Lemma~2, p.~12}]{Harnau7BVerallgRelationenbegriffII} for the
situation without nullary operations).
\begin{lemma}\label{lem:Polp=semiclone}
Any polymorphism set \m{\Polp{Q}} with \m{Q\subs\Relps}
is a semiclone.
\end{lemma}
\begin{easykeepproof}%
Consider any \m{Q\subs\Relps} and put
\m{Q_1\defeq \lset{\rho\in\Rels}{\apply{\rho,\rho'}\in Q}}. Let
\m{m,n\in\N}, \m{f\in\Polp[n]{Q}} and
\m{\apply{\listen{g}{n}}\in\apply{\Pol[m]{Q_1}}^{n}}. We prove that the
composition \m{h\defeq \composition{f}{\listen{g}{n}}} belongs to
\m{\Polp{Q}}, which demonstrates our lemma as obviously
\m{\TrivOps\cup\Polp{Q}\subs\Pol{Q_1}}. Indeed, \m{h} preserves any
relation pair \m{\apply{\rho,\rho'}\in Q}: whenever
\m{\mathbf{r}\in\rho^m}, then superassociativity yields
\begin{equation*}
\composition{h}{\mathbf{r}} =
\composition{\apply{\composition{f}{\listen{g}{n}}}}{\mathbf{r}}
\stackrel{\text{\ref{lem:superassociativity}%
                \eqref{item:superassociativity}}}{=}
\composition{f}{\composition{g_{0}}{\mathbf{r}},\dotsc,
                \composition{g_{n-1}}{\mathbf{r}}}.
\end{equation*}
The latter tuple is a member of \m{\rho'} as
\m{f\preserves\apply{\rho,\rho'}} and
\m{\apply{\composition{g_{0}}{\mathbf{r}},\dotsc,
          \composition{g_{n-1}}{\mathbf{r}}}}
belongs to \m{\rho^{n}} due to
\m{\listen{g}{n}\in\Pol{Q_1}\subs\Pol{\set{\rho}}}.
\end{easykeepproof}
\par

The following facts can be routinely proven using
Lemma~\ref{lem:Polp=semiclone}.
\begin{corollary}\label{cor:genSclone-subs-PolpInvp}
For any set \m{F\subs\Ops} we have
\begin{align*}
\genSclone{F}\subs\Polp{\Invp{F}}&&
\text{and}&&
\Invp{F} = \Invp{\genSclone{F}}.
\end{align*}
\end{corollary}
\par

The next lemma
(cf.~\cite[{Lemma~3, p.~13}]{Harnau7BVerallgRelationenbegriffII}) clarifies
which sets of relation pairs yield proper clones.
\begin{lemma}\label{lem:char-Polp-clone}
For \m{Q\subs\Relps} a semiclone \m{\Polp{Q}} is a clone if and only if
\m{\rho=\rho'} holds for all \m{\apply{\rho,\rho'}\in Q}.
\end{lemma}
\begin{easykeepproof}
If \m{Q\subs\Relps} only consists of identical pairs, then we saw already
in Subsection~\ref{subsect:Polp-Invp} that \m{\Polp{Q}=\Pol{Q'}} where
\m{Q'=\lset{\rho\in\Rels}{\apply{\rho,\rho}\in Q}}. This set always is a
clone. On the other hand, if \m{\Polp{Q}} is a clone, then we have
\m{\id_{\CarrierSet}\in\Polp{Q}}, which implies \m{\rho\subs\rho'} and
thus \m{\rho=\rho'} for every \m{\apply{\rho,\rho'}\in Q}.
\end{easykeepproof}
\par

Note that (along with an appropriate generalisation of preservation)
the three previous statements remain true if one considers
relation pairs of arbitrary, possibly infinite arity. That is to say,
pairs \m{\apply{R,S}}, where \m{S\subs R\subs \CarrierSet[K]} for some
fixed set~\m{K}.
\par

The following two results are in close analogy to Proposition~1.11(a),(b)
from~\cite[p.~17]{PoeGeneralGaloisTheoryForOperationsAndRelations}.
\begin{lemma}\label{lem:Polp-sLoc}
For \m{s\in\N} and any set
\m{Q\subs\Relps[{\leq s}]\defeq\biguplus_{0\leq m\leq s}\Relps[m]} of at
most \nbdd{s}ary relation pairs, we have
\m{\sLoc{\Polp{Q}}=\Polp{Q}}.
\end{lemma}
\begin{easykeepproof}
Consider \m{n\in\N} and \m{g\in\sLoc[s][n]{\Polp{Q}}}. To prove that
\m{g\in\Polp{Q}} take \m{\apply{\rho,\rho'}\in\Fn[m]{Q}} for some
\m{0\leq m\leq s}; we have to check that
\m{g\preserves\apply{\rho,\rho'}}. For this consider any \nbdd{n}tuple
\m{\mathbf{r}=\apply{r_j}_{0\leq j<n}\in\rho^n} of tuples from~\m{\rho}
and define \m{B\defeq \lset{\apply{r_j(i)}_{0\leq j<n}}{0\leq i<m}}.
Evidently, \m{B} is a subset of the domain~\m{\CarrierSet[n]} of~\m{g}
and satisfies \m{\abs{B}\leq m\leq s}. Hence, as
\m{g\in\sLoc{\Polp{Q}}}, there exists some \m{f\in\Polp[n]{Q}} such that
\m{g\Restriction_B = f\Restriction_B}. This implies that
\m{\composition{g}{\mathbf{r}} = \composition{f}{\mathbf{r}}}, and the
latter tuple belongs to \m{\rho'} as
\m{f\preserves\apply{\rho,\rho'}\in Q}.
\end{easykeepproof}

\begin{corollary}\label{cor:Polp-Loc}
The equality \m{\Loc{\Polp{Q}}=\Polp{Q}} is satisfied for any
\m{Q\subs\Relps}.
\end{corollary}
\begin{easykeepproof}
Consider \m{g\in\Loc{\Polp{Q}}} and \m{\apply{\rho,\rho'}\in\Fn[m]{Q}},
\m{m\in\N}. By definition of \m{\Loc{}}, we have
\m{g\in\sLoc[m]{\Polp{Q}}\subs\sLoc[m]{\Polp{\set{\apply{\rho,\rho'}}}}},
which equals \m{\Polp{\set{\apply{\rho,\rho'}}}} by
Lemma~\ref{lem:Polp-sLoc}. As the pair \m{\apply{\rho,\rho'}\in Q} was
arbitrarily chosen, we obtain \m{g\in\Polp{Q}}.
\end{easykeepproof}

\section{Relation pair clones}\label{subsect:relationpair-clones}
In this section we first recollect the so\dash{}called \emph{general
superposition} of relations
(\cite[Definition~3.4(R4), p.~27]{%
PoeGeneralGaloisTheoryForOperationsAndRelations}, see
also~\cite[Definition~2.2(ii), p.~258]{%
PoeConcreteRepresentationOfAlgebraicStructuresAndGeneralGaloisTheory}
and~\cite{BehClonesWithNullaryOperations}),
which comes into play when generalising the
notion of relational clone from finite carrier sets to arbitrary ones. It
is not surprising that it will be important for the generalisation of
relation pair algebras as introduced in~\cite[p.~21]{Harnau8Habil} (see
also~\cite[p.~16]{Harnau7BVerallgRelationenbegriffII}) to carrier sets of
arbitrary cardinality, as well.
\par
\begin{definition}\label{def:general-superposition}
Let~\m{\CarrierSet} be any carrier set, moreover let index sets~\m{I}
and~\m{\mu} (one could in principle restrict to ordinal numbers, but this only
makes working with the definition more technical), natural numbers
\m{m, m_i\in \N} (\m{i\in I}),
mappings \m{\apply{\functionhead{\alpha_i}{m_i}{\mu}}_{i \in I}} and
\m{\functionhead{\beta}{m}{\mu}}, and relations \m{\rho_i
\in\Rels[{m_i}]}, \m{i\in I}, be given. The \emph{general superposition}
of these relations \wrt\ the given data is defined to be the \nbdd{m}ary
relation
\begin{equation*}
\begin{aligned}[t]
\genComp{\beta}{\apply{\alpha_{i}}_{i\in I}}{\apply{\rho_i}_{i \in I}}
&\defeq\rset{y\in \CarrierSet[m]}{\exists\,a\in \CarrierSet[\mu]\colon
                       y = a\circ\beta \land
                       \forall\,i\in I\colon a \circ \alpha_i \in \rho_i}\\
&\mathrel{\mathopen{\hphantom{:}}=}\rset{a \circ \beta}{%
                       a \in \CarrierSet[\mu] \land
                       \forall\,i\in I\colon a \circ \alpha_i \in \rho_i}.
\end{aligned}
\end{equation*}
\end{definition}
\par
We mention in passing that, in general, a \emph{relational clone} can be
defined as any set \m{Q\subs\Rels} that is closed \wrt\ general
superposition.
That is, whenever data as in Definition~\ref{def:general-superposition} is
given and all relations~\m{\rho_i}, \m{i\in I}, belong to~\m{Q}, then also
\m{\genComp{\beta}{\apply{\alpha_{i}}_{i\in I}}{\apply{\rho_{i}}_{i\in I}}}
has to be an element of~\m{Q} (if nullary relations are disregarded, then
one restricts the integers~\m{m} and \m{\apply{m_i}_{i\in I}} to positive
ones only). Depending on the carrier set~\m{\CarrierSet},
one can work out cardinality bounds on the sets~\m{I} and~\m{\mu}
involved in this closure property, but this is not our concern
here.
\par

Different specialisations of the general superposition yield operations
known from the closure property corresponding to relational clones on
finite carrier sets: variable permutation, projection onto arbitrary
subsets of coordinates, variable identification, addition of fictitious
coordinates, all diagonal relations as nullary constants, and (even
arbitrary) intersection of relations of the same arity.
\par

We now straightforwardly extend the general superposition from relations
to relation pairs.
\begin{definition}\label{def:gen-comp}
Let~\m{\CarrierSet} be any carrier set, \m{I}, \m{\mu}, \m{m}, \m{m_i},
\m{\functionhead{\alpha_i}{m_i}{\mu}}, \m{i\in I}, and
\m{\functionhead{\beta}{m}{\mu}} as in
Definition~\ref{def:general-superposition}. For relation pairs
\m{\apply{\rho_i,\rho'_i}\in\Relps[m_i]}, \m{i\in I}, we define their
\emph{general superposition} to be
\begin{equation*}
\genComp{\beta}{\apply{\alpha_{i}}_{i\in I}}{%
                                   \apply{\rho_i,\rho'_i}_{i \in I}}
\defeq \apply{%
\genComp{\beta}{\apply{\alpha_{i}}_{i\in I}}{\apply{\rho_i}_{i \in I}},
\genComp{\beta}{\apply{\alpha_{i}}_{i\in I}}{\apply{\rho'_i}_{i \in I}}}.
\end{equation*}
\end{definition}
\par

It is easy to see that this definition is well\dash{}defined, \ie\ that
we really have
\m{\genComp{\beta}{\apply{\alpha_{i}}_{i\in I}}{%
         \apply{\rho_i,\rho'_i}_{i \in I}
         \in\Relps[m]}}
in the situation described in Definition~\ref{def:gen-comp}. This allows
us to define relation pair clones as such sets of relation pairs that are
closed under general superposition.%
\par

\begin{definition}\label{def:rel-pair-clone}
We say that for some carrier~\m{\CarrierSet} a set \m{Q\subs\Relps} is a
\emph{relation pair clone} if and only if %
the following condition is satisfied: whenever~\m{I},
\m{\mu}, \m{m}, \m{m_i}, \m{\functionhead{\alpha_i}{m_i}{\mu}}, \m{i\in I}
and \m{\functionhead{\beta}{m}{\mu}} are as in
Definition~\ref{def:general-superposition}, and
\m{\apply{\rho_i,\rho'_i}\in\Fn[m_i]{Q}} are given for \m{i\in I}, then
also
\m{\genComp{\beta}{\apply{\alpha_{i}}_{i\in I}}{%
                           \apply{\rho_i,\rho'_i}_{i \in I}}\in\Fn[m]{Q}}.
\end{definition}
\par

One can routinely check that for a given carrier set~\m{\CarrierSet} the
collection of all relation pair clones on~\m{\CarrierSet} is a closure
system. We denote the corresponding closure operator by
\m{Q\mapsto\genRpclone{Q}} for \m{Q\subs\Relps} and refer to
\m{\genRpclone{Q}} as the \emph{relation pair clone generated by~\m{Q}}.
\par

Note that for finite carrier sets \m{\CarrierSet\neq\emptyset}, and
provided that
\m{\apply{\emptyset,\emptyset}\in\Fn[m]{Q}} for all \m{m\in\N}, our
concept of locally closed relation pair clone, by taking
\m{Q\setminus\Relps[0]}, subsumes that of subuniverses of the full
relation pair algebra defined in~\cite[p.~21]{Harnau8Habil} (see
also~\cite[p.~16]{Harnau7BVerallgRelationenbegriffII}).
\par

There are two issues here: the necessity to add local closure and the
requirement that pairs of empty relations have to belong to relation pair
algebras in Harnau's sense. We noted above in
Corollary~\ref{cor:finite=>relaxation=LOC} that for finite carrier sets
closure under relaxation coincides with our local closure of relation
pairs. Moreover, we shall prove in Corollary~\ref{cor:InvpPolp=LOC[]} that
the closed sets \wrt\ \m{\Invp{\Polp{}}} are precisely the locally closed
relation pair clones, which implies for finite carrier sets that they
are exactly those relation pair clones that are closed
\wrt\ relaxations. In~\cite{Harnau8Habil}
and~\cite{Harnau7BVerallgRelationenbegriffII} this additional closure
property (with the goal of characterising the \name{Galois} closures) has
been incorporated into the definition of the full relation pair algebra
via multioperations~\m{d_{\mathrm{v}}} and~\m{d_{\mathrm{h}}}; however, it
has been noted that these operators are of a different nature than the
other fundamental operations of relation pair algebras. Comparing to the
situation known from clones and relational clones on arbitrary domains
(see~\cite{PoeGeneralGaloisTheoryForOperationsAndRelations,%
    PoeConcreteRepresentationOfAlgebraicStructuresAndGeneralGaloisTheory,%
    SzaboConcreteRepresentationOfRelatedStructuresOfUniversalAlgebras})
and looking from the perspective of infinite carrier sets, which requires
local closures anyway, it is justified to modify Harnau's definition by
separating closure properties related to concrete constructions
involving relations from local interpolation properties.
We mention that for finite~\m{\CarrierSet} the constructive part can be
expressed via interpretations of primitive positive formul\ae{} in both
components.
In fact, it was noted by \'{A}gnes Szendrei that given a set
\m{Q\subs\Relps}, one may consider the relational structures
\m{\RelAlgebra[\apply{\rho}_{\apply{\rho,\rho'}\in Q}]{\CarrierSet}}
and
\m{\RelAlgebra{A'}=\algwops{\CarrierSet}{%
                            \apply{\rho'}_{\apply{\rho,\rho'}\in Q}}}
and primitive positively definable relations on the product
\m{\RelAlgebra{A}\times\RelAlgebra{A'}}: if~\m{\phi} is a primitive
positive formula in the language of~\m{Q} (including equality) with at
most~\m{m} free variables, then it defines the following \nbdd{m}ary
relation on the product
\begin{align*}
\hat{\sigma}&\defeq
\lset{\apply{\apply{x_i,y_i}}_{i\in m}\in \apply{\CarrierSet[2]}^m}{%
\apply{\RelAlgebra{A}\times\RelAlgebra{A'},
          \apply{\apply{x_i,y_i}}_{i\in m}}\models \phi}\\
&\,=\lset{\apply{\apply{x_i,y_i}}_{i\in m}\in \apply{\CarrierSet[2]}^m}{%
\apply{\RelAlgebra{A},\apply{x_i}_{i\in m}}\models \phi
\land
\apply{\RelAlgebra{A'},\apply{y_i}_{i\in m}}\models \phi
}\\
&\,=\lset{\apply{\apply{x_i,y_i}}_{i\in m}\in \apply{\CarrierSet[2]}^m}{%
\apply{\apply{x_i}_{i\in m},\apply{y_i}_{i\in m}}\in\sigma\times\sigma'},
\end{align*}
where
\m{\sigma\defeq\lset{\mathbf{x}\in\CarrierSet[m]}{%
                          \apply{\RelAlgebra{A},\mathbf{x}}\models \phi}}
and
\m{\sigma'\defeq\lset{\mathbf{x}\in\CarrierSet[m]}{%
                          \apply{\RelAlgebra{A'},\mathbf{x}}\models \phi}}.
If~\m{\sigma} and~\m{\sigma'} are both non\dash{}empty, then one may
obtain the relation pair \m{\apply{\sigma,\sigma'}} defined by~\m{\phi} as
projections of~\m{\hat{\sigma}}. If one of them is the empty set, then
\m{\hat{\sigma}=\emptyset} and therefore both projections will be empty.
Thus only taking projections of~\m{\hat{\sigma}} (\ie\ of
pp\dash{}definable relations in the product
\m{\RelAlgebra{A}\times\RelAlgebra{A'}}) will never produce
relation pairs \m{\apply{\sigma,\sigma'}} where
\m{\sigma'=\emptyset\subsetneq \sigma}, which is certainly needed, \eg,
to model intersection in both components.
However, collecting all pairs \m{\apply{\sigma,\sigma'}} arising from
primitive positive formul\ae~\m{\phi} correctly describes the closure
\m{\genRpclone{Q}} in the case of finite carrier sets.
\par

The second issue pointed out above is related to nullary operations. In
the literature these are often neglected, which makes it necessary for
relation pair algebras (\cite{Harnau7BVerallgRelationenbegriffII}) and for
relational clones (relation algebras,
\cite{PoeGeneralGaloisTheoryForOperationsAndRelations}) to contain the
empty pair \m{\apply{\emptyset,\emptyset}} and the empty relation,
respectively, in order to be in accordance with the corresponding
\name{Galois} theory.
\par

If nullary operations are given their proper place, this absurdity
vanishes (see~\cite{BehClonesWithNullaryOperations} for clones and
relational clones); then empty relations (pairs) get a true function,
indicating by their presence the absence of nullary operations on the dual
side (see Lemma~\ref{lem:rel-pair-clones-nullary-ops} below).
This is also the reason why we cannot and do not add the empty pairs
of all arities as nullary constants to the closure condition of relation
pair clones.
\par

Relational clones (as given
in~\cite[Definition~2.2, p.~8]{BehClonesWithNullaryOperations})
relate to relation pair clones in the following way:
\begin{lemma}\label{lem:rel-clone-rel-pair-clone}
For any carrier set~\m{\CarrierSet} a subset \m{Q\subs\Rels} is a
relational clone if and only if
\m{P\defeq\biguplus_{m\in\N}\lset{\apply{\rho,\rho}}{\rho\in\Fn[m]{Q}}} is
a relation pair clone.
\end{lemma}
\begin{proof}
Evidently, \m{P} is closed under arbitrary general superpositions if and
only if~\m{Q} is.
\end{proof}
\par

The following result is comparably easy.
\begin{lemma}\label{lem:rel-pair-clone-rel-clone}
Whenever \m{Q\subs\Relps} is a relation pair clone on some
set~\m{\CarrierSet}, then
\begin{align*}
Q'&\defeq
   \lset{\rho}{\apply{\rho,\rho'}\in\Fn[m]{Q} \text{ for some } m\in\N},\\
Q''&\defeq
   \lset{\rho'}{\apply{\rho,\rho'}\in\Fn[m]{Q} \text{ for some } m\in\N}
   \text{ and}\\
Q'''&\defeq
   \lset{\rho}{\apply{\rho,\rho}\in\Fn[m]{Q} \text{ for some } m\in\N}
\end{align*}
are relational clones on~\m{\CarrierSet}.
\end{lemma}
\begin{proof}
Closedness of~\m{Q} \wrt\ general superpositions carries over to~\m{Q'},
\m{Q''} and~\m{Q'''}.
\end{proof}
\par

\begin{sloppypar}
Similarly as for semiclones, the \name{Galois} correspondence
\m{\PolpOp\text{-}\InvpOp} provides many examples of relation pair clones
(see~\cite[Lemma~9, p.~21]{Harnau8Habil}
or~\cite[Lemma~9, p.~16]{Harnau7BVerallgRelationenbegriffII} for the case
of finite carrier sets; cf.~\cite[Lemma~3.1, p.~154]{%
CouceiroFoldesClosedSetsRelConstrFuncsClosedWrtVarSubst} for the
general framework of relational constraints and conjunctive minors).
\end{sloppypar}
\begin{lemma}\label{lem:Invp=rel-pair-clone}
For each \m{F\subs\Ops} the set \m{\Invp{F}} is a relation pair clone.
\end{lemma}
\begin{proof}
To check that \m{\Invp{F}} for \m{F\subs\Ops}
is closed \wrt\ to general superpositions
let~\m{I} and~\m{\mu} be sets, \m{m,m_i\in\N} for \m{i\in I}, and
\m{\functionhead{\beta}{m}{\mu}} and
\m{\functionhead{\alpha_i}{m_i}{\mu}} for \m{i\in I} be mappings.
For given relation pairs \m{\apply{\rho_i,\rho'_i}\in\Invp[m_i]{F}} we are
going to show that
\m{\apply{\rho,\rho'}\defeq
\genComp{\beta}{\apply{\alpha_i}_{i\in I}}{\apply{\rho_i,\rho'_i}_{i\in I}}
\in\Invp[m]{F}}.
For this let \m{f\in F} and put \m{n\defeq \arity\apply{f}}. To verify
that \m{f\preserves\apply{\rho,\rho'}}, let us take any
\m{\mathbf{r}\in\rho^n}. By definition of
\m{\rho=
   \genComp{\beta}{\apply{\alpha_i}_{i\in I}}{\apply{\rho_i}_{i\in I}}},
for each \m{0\leq j<n} there exists \m{a_j\in\CarrierSet[\mu]}
such that \m{\mathbf{r}\apply{j} = a_j\circ \beta} and
\m{a_j\circ\alpha_i\in\rho_i} for all \m{i\in I}. By putting
\m{\mathbf{a}\defeq\apply{a_j}_{j\in n}\in \apply{\CarrierSet[\mu]}^n}, we
hence obtain
\m{\mathbf{r}=\apply{a_j \circ\beta}_{j\in n}=\mathbf{a}\circ\beta}
(cp.~Lemma~\ref{lem:superassociativity}%
           \eqref{item:pre-superassociativity}).
Therefore, by associativity, we get
\m{\composition{f}{\mathbf{r}} = \composition{f}{\mathbf{a}\circ \beta}
                      = \apply{\composition{f}{\mathbf{a}}}\circ \beta},
which belongs to~\m{\rho'} as
\m{\apply{\composition{f}{\mathbf{a}}}\circ \alpha_i =
\composition{f}{\mathbf{a}\circ \alpha_i}\in \rho'_i} for all \m{i\in I}
(due to \m{\mathbf{a}\circ\alpha_i = \apply{a_j\circ\alpha_i}_{j\in n}\in
\rho_i^n} and~\m{f} preserving \m{\apply{\rho_i,\rho'_i}\in\Invp[m_i]{F}}).
\end{proof}
\par

As a direct consequence we get the following compulsory corollary.
\begin{corollary}\label{cor:genRpclone-subs-InvpPolp}
For any set \m{Q\subs\Relps} we have
\begin{align*}
\genRpclone{Q}\subs\Invp{\Polp{Q}}&&
\text{and}&&
\Polp{Q} = \Polp{\genRpclone{Q}}.
\end{align*}
\end{corollary}
\par

Next, we quickly address how nullary operations affect the
associated relation pair algebras.
\begin{lemma}\label{lem:rel-pair-clones-nullary-ops}
For \m{F\ovflhbx{0.91pt}\subs\ovflhbx{0.91pt}\Ops} we have \m{\apply{\emptyset,\emptyset}\ovflhbx{0.91pt}\in\ovflhbx{0.91pt}\Invp{F}} if and
only if \m{F\ovflhbx{0.91pt}\subs\ovflhbx{0.91pt}\Ops\ovflhbx{1.88pt}\setminus\Ops[0]}.
\end{lemma}
\begin{proof}
It is clear that every operation of positive arity preserves
\m{\apply{\emptyset,\emptyset}}, \ie\ that
\m{\apply{\emptyset,\emptyset}\in\Invp{\apply{\Ops\setminus\Ops[0]}}}.
Conversely, assume that \m{\Fn[0]{F}\neq \emptyset}, say~\m{F} contains a
nullary constant operation~\m{c} with value \m{a\in\CarrierSet}. If
\m{c\preserves \apply{\emptyset,\rho}} then it follows that
\m{\apply{a,\dotsc,a}\in\rho}, \ie\ \m{\rho\neq\emptyset}. Thus,
\m{\apply{\emptyset,\emptyset}\notin\Invp{\set{c}}}, and so
\m{\apply{\emptyset,\emptyset}\notin\Invp{F}}.
\end{proof}
\par

The following two results have their analogues in Proposition~1.11(a'),(b')
from~\cite[p.~17]{PoeGeneralGaloisTheoryForOperationsAndRelations}.
\begin{lemma}\label{lem:Invp-sLOC}
For \m{s\in\N} and any set
\m{F\subs\Ops[{\leq s}]\defeq\biguplus_{0\leq n\leq s}\Ops[n]} of at
most \nbdd{s}ary operations, we have
\m{\sLOC{\Invp{F}}=\Invp{F}}.
\end{lemma}
\begin{proof}
Let \m{m\in\N} and \m{\apply{\sigma,\sigma'}\in\sLOC[s][m]{\Invp{F}}}.
Consider any \m{f\in F}, then \m{n\defeq\arity\apply{f}} necessarily
fulfils \m{n\leq s}. Therefore, if we consider any
\m{\mathbf{r}=\apply{r_j}_{0\leq j<n}} in \m{\sigma^{n}} and put
\m{B\defeq\lset{r_j}{0\leq j<n}\subs\sigma\subs\CarrierSet[m]}, we clearly
have a finite subset \m{B\subs\sigma} of cardinality at most \m{n\leq s}.
As \m{\apply{\sigma,\sigma'}\in\sLOC{\Invp{F}}}, there is some
\m{\apply{\rho,\rho'}\in\Invp[m]{F}} such that \m{B\subs\rho} and
\m{\rho'\subs\sigma'}. We know that \m{f\preserves\apply{\rho,\rho'}}, so
since \m{B\subs\rho}, we get \m{\mathbf{r}\in\rho^{n}} and thus
\m{\composition{f}{\mathbf{r}}\in \rho'\subs\sigma'}. Consequently, we
have shown \m{f\preserves\apply{\sigma,\sigma'}}, and as \m{f\in F} was
arbitrary, we obtain \m{\apply{\sigma,\sigma'}\in \Invp{F}} as desired.
\end{proof}

\begin{corollary}\label{cor:Invp-LOC}
The equality \m{\LOC{\Invp{F}}=\Invp{F}} is satisfied for any
\m{F\subs\Ops}.
\end{corollary}
\begin{easykeepproof}
Consider \m{\apply{\sigma,\sigma'}\in\LOC{\Invp{F}}} and \m{f\in\Fn{F}},
\m{n\in\N}. By definition of \m{\LOC{}}, we have
\m{\apply{\sigma,\sigma'}\in  \sLOC[n]{\Invp{F}}
                         \subs\sLOC[n]{\Invp{\set{f}}}},
which by Lemma~\ref{lem:Invp-sLOC} is equal to \m{\Invp{\set{f}}}.
As the function \m{f\in F} was
arbitrarily chosen, we obtain \m{\apply{\sigma,\sigma'}\in\Invp{F}}.
\end{easykeepproof}

\section{Characterisation of closures related to
         \texorpdfstring{\m{\PolpOp\text{-}\InvpOp{}}}{Polp-Invp}}%
\label{sect:char-closures}
In this section we characterise, for any parameter \m{s\in\N}, the
operators \m{\Polp{\Invp[\leq s]{}}} and \m{\Invp{\Polp[\leq s]{}}} as
\nbdd{s}local closures of the generated semiclone and relation pair clone,
respectively. Subsequently, we present a few consequences of these
theorems.
\par

\subsection{The operational side}\label{subsect:PolpInvp}
For our task it is helpful to gather some knowledge about the least
(\wrt~\m{\mathord{\leq}} and thus a least among several equivalent ones
\wrt~\m{\mathord{\qleq}}) pair \m{\apply{\rho,\rho'}\in\Relps[m]} being
invariant for some set \m{F\subs\Ops} and satisfying \m{B\subs \rho} for a
given finite set \m{B\subs\Rels[m]}, \m{m\in\N}. Addressing this issue,
the following lemma generalises Proposition~2.4
of~\cite[p.~21]{PoeGeneralGaloisTheoryForOperationsAndRelations} from
relations to relation pairs.
\par

\begin{lemma}\label{lem:Gamma-F}
Let \m{F\subs\Ops} be a set of operations
and \m{\mathbf{b}\in\apply{\CarrierSet[m]}^{n}} for some \m{m,n\in\N};
set \m{B\defeq\lset{\mathbf{b}(j)}{0\leq j<n}\subs\CarrierSet[m]}.
Then the pair \m{\genInvPair{B}\defeq\apply{\rho,\rho'}}, where
\m{\rho\defeq\lset{\composition{f}{\mathbf{b}}}{f\in\genClone[n]{F}}} and
\m{\rho'\defeq\lset{\composition{f}{\mathbf{b}}}{f\in\genSclone[n]{F}}},
is the least pair (\wrt~\m{\mathord{\leq}}) in \m{\Invp[m]{F}} satisfying
\m{B\subs\rho}.
\end{lemma}
Note that the lemma also shows that the relations
\m{\rho,\rho'\in\Rels[m]} do not depend on the order of the entries of the
tuple~\m{\mathbf{b}}.
Furthermore, instead of the finite cardinal~\m{m}, any cardinal or, in
fact, any indexing set~\m{K} can be used, provided the notion of
preservation is straightforwardly extended to relation pairs of arbitrary
arity, \ie\ pairs \m{\apply{R,S}} such that
\m{S\subs R\subs\CarrierSet[K]}.
\begin{proof}
First of all, it is clear that \m{B\subs\rho} as
\m{\TrivOps[n]\subs\genClone[n]{F}}. Next, we prove that
\m{\apply{\rho,\rho'}\in\Invp{F}}. For this let \m{\ell\in\N},
\m{g\in\Fn[\ell]{F}} and
\m{\mathbf{r}=\apply{r_{j}}_{0\leq j<\ell}\in\rho^{\ell}}.
By construction of~\m{\rho}, for each \m{0\leq j<\ell} there exists some
\m{f_j\in\genClone[n]{F}=\genSclone[n]{F}\cup\TrivOps[n]} (see
Lemma~\ref{lem:semiclones+proj=clones}%
      \eqref{item:char-clone-cl=sclone-cl+proj})
such that \m{r_j = \composition{f_j}{\mathbf{b}}}. Using
Lemma~\ref{lem:superassociativity}\eqref{item:superassociativity}, we have
\begin{equation*}
  \composition{g}{\mathbf{r}}
  =\composition{g}{\composition{f_0}{\mathbf{b}},\dotsc,
                   \composition{f_{\ell-1}}{\mathbf{b}}}
  =\composition{\apply{\composition{g}{\listen{f}{\ell}}}}{\mathbf{b}}
  \in\rho',
\end{equation*}
since \m{\composition{g}{\listen{f}{\ell}}\in\genSclone[n]{F}} by the
closure property of semiclones.
\par
Finally, we prove that any pair \m{\apply{\sigma,\sigma'}\in\Invp[m]{F}}
satisfying \m{B\subs\sigma} fulfils
\m{\apply{\rho,\rho'}\leq\apply{\sigma,\sigma'}}.
By Corollary~\ref{cor:genSclone-subs-PolpInvp} we know
\m{\apply{\sigma,\sigma'}\in\Invp{F} = \Invp{\genSclone{F}}}, so since
\m{B\subs\sigma} we have \m{\composition{f}{\mathbf{b}}\in\sigma'} for any
\m{f\in\genSclone[n]{F}}. Therefore, \m{\rho'\subs\sigma'\subs\sigma}. As,
by Lemma~\ref{lem:semiclones+proj=clones}%
      \eqref{item:char-clone-cl=sclone-cl+proj},
\m{\genClone[n]{F}=\genSclone[n]{F}\cup\TrivOps[n]}, it follows that
\m{\rho=\rho'\cup B}. We have \m{B\subs\sigma} by assumption and
\m{\rho'\subs\sigma} as demonstrated before. Thus \m{\rho\subs\sigma},
whence \m{\apply{\rho,\rho'}\leq\apply{\sigma,\sigma'}}.
\end{proof}
\par

\begin{corollary}\label{cor:rho-X-n}
Let \m{F\subs\Ops}, \m{n\in\N}, and \m{X\subs\CarrierSet[n]} be any subset
of finite cardinality \m{\abs{X}\eqdef k < \aleph_{0}}; moreover, consider
an arbitrary bijection \m{\functionhead{\beta}{k}{X}} as fixed. Defining
\m{B\defeq\lset{\eni{i}\Restriction_{X}\circ\beta}{0\leq i<n}
    \subs\CarrierSet[k]},
as well as \nbdd{k}ary relations
\m{\rho_{X,n}\defeq
                \lset{f\Restriction_{X}\circ \beta}{f\in\genClone[n]{F}}}
and \m{\rho'_{X,n}\defeq
                \lset{f\Restriction_{X}\circ \beta}{f\in\genSclone[n]{F}}},
we have \m{\apply{\rho_{X,n},\rho'_{X,n}}=\genInvPair{B}\in\Invp[k]{F}}.
\end{corollary}
\begin{easyproof}
The claim follows from Lemma~\ref{lem:Gamma-F} by observing that the
equality
\m{\composition{f}{\eni{i}\Restriction_{X}}_{0\leq i<n}
 = \composition{f}{\eni{i}}_{0\leq i<n}\Restriction_{X}
 = f\Restriction_{X}}
holds for all \m{f\in\Ops[n]}.
\end{easyproof}
\par

We are now prepared to prove our first theorem, characterising the closure
\m{\Polp{\Invp[\leq s]{}}} for \m{s\in\N}.
\begin{theorem}\label{thm:char-Polp-Invp-leq-s}
For \m{s\in\N} and any set of operations \m{F\subs\Ops} we have the
equality \m{\Polp{\Invp[\leq s]{F}}=\sLoc{\genSclone{F}}}.
\end{theorem}
\begin{proof}
We have
\m{\genSclone{F}\subs\Polp{\Invp[\leq s]{\genSclone{F}}}
                    =\Polp{\Invp[\leq s]{F}}}
by Corollary~\ref{cor:genSclone-subs-PolpInvp},
whence
\m{\sLoc{\genSclone{F}}\subs
   \sLoc{\Polp{\Invp[\leq s]{F}}}=\Polp{\Invp[\leq s]{F}}},
using Lemma~\ref{lem:Polp-sLoc}.
\par
For the converse inclusion take \m{g\in\Polp[n]{\Inv[\leq s]{F}}} for any
\m{n\in\N}; we want to prove that \m{g\in\sLoc[s][n]{\genSclone{F}}}. To
do so, we consider any finite \m{X\subs\CarrierSet[n]} where
\m{k\defeq\abs{X}\leq s} and an arbitrary bijection
\m{\functionhead{\beta}{k}{X}}. Now Corollary~\ref{cor:rho-X-n} yields
that \m{\apply{\rho_{X,n},\rho'_{X,n}}\in\Invp[k]{F}\subs\Invp[\leq s]{F}},
wherefore \m{g\preserves \apply{\rho_{X,n},\rho'_{X,n}}}. Moreover, we
have
\m{B=\lset{\eni{i}\Restriction_{X}\circ \beta}{0\leq i<n}\subs\rho_{X,n}},
whence we obtain
\m{g\Restriction_{X}\circ\beta
=\composition{g}{\eni{i}}_{0\leq i<n}\Restriction_{X}\circ\beta
=\composition{g}{\eni{i}\Restriction_{X}\circ\beta}_{0\leq i<n}
\in\rho'_{X,n}}.
Thus by definition of~\m{\rho'_{X,n}} there has to exist some
\m{f\in\genSclone[n]{F}} such that
\m{g\Restriction_{X}\circ\beta=f\Restriction_{X}\circ\beta}, which implies
\m{g\Restriction_{X} = f\Restriction_{X}} by bijectivity of~\m{\beta}.
Yet, this finally proves that \m{g\in\sLoc[s][n]{\genSclone{F}}}.
\end{proof}
\par

The following simple observation is not unexpected.
\begin{lemma}\label{lem:rel-pair-clone-sub-arity-closure}
Any relation pair clone \m{Q\subs\Relps} on a non\dash{}empty carrier
set~\m{\CarrierSet} satisfies \m{\Fn[m]{Q}\subs\genRpclone{\Fn[s]{Q}}} for all
\m{m,s\in\N} where \m{m\leq s}.
\par
For \m{\CarrierSet=\emptyset}, we have in fact
\m{\Fn[s]{Q}\subs\genRpclone{\Fn[0]{Q}}} for all \m{s\in\N} and any
relation pair clone \m{Q\subs\Relps}.
\end{lemma}
\begin{easykeepproof}
For \m{\CarrierSet\neq\emptyset}, it is clear for \nbdd{m}ary relations
\m{\rho\subs\CarrierSet[m]} that one can write
\m{\rho=\pr_{0,\dotsc,m-1}\apply{\rho\times\CarrierSet[{s-m}]}}.
Designating by \m{\functionhead{\iota}{m}{s}} the identical embedding, one
may rewrite this relationship as
\m{\rho=\genComp{\iota}{\id_{s}}{\genComp{\id_{s}}{\iota}{\rho}}}. Since
the definition of the operators only depends on the arity of the
relation~\m{\rho}, the same works for \nbdd{m}ary relation pairs.
So, if \m{\apply{\rho,\rho'}\in\Fn[m]{Q}}, then
\m{\genComp{\id_{s}}{\iota}{\apply{\rho,\rho'}}\in\Fn[s]{Q}}, and thus
\m{\apply{\rho,\rho'}
=\genComp{\iota}{\id_{s}}{\genComp{\id_{s}}{\iota}{\apply{\rho,\rho'}}}
\in\genRpclone{\Fn[s]{Q}}}.
\par
For the empty carrier set, in point of fact, the opposite holds: for
\m{s=0} the claim is trivial. For \m{s\in\Np} and any
\m{\rho\in\Rels[s]}, we have \m{\rho=\CarrierSet[s] =\emptyset}, and
\m{\rho=\emptyset\times\emptyset =
\apply{\pr_{\emptyset}\!\rho}\times \CarrierSet[s] =
\genComp{\id_{s}}{\iota}{\!\!\genComp{\iota}{\id_{s}}{\rho}}} where \m{\iota}
is the map from above for \m{m=0}. Since
\m{\genComp{\iota}{\id_{s}}{\apply{\rho,\rho'}}\in\Fn[0]{Q}} for
\m{\apply{\rho,\rho'}\in\Fn[s]{Q}}, we have
\m{\Fn[s]{Q}\subs\genRpclone{\Fn[0]{Q}}}.
\end{easykeepproof}
\par

Hence, we can prove the first corollary to our theorem.
\begin{corollary}\label{cor:char-Polp-Invp-s}
For \m{s\in\N} and any set of operations \m{F\subs\Ops}
on~\m{\CarrierSet\neq\emptyset} we have the
equality \m{\Polp{\Invp[s]{F}}=\sLoc{\genSclone{F}}}.
\par
If \m{\CarrierSet=\emptyset}, we have
\m{\sLoc{\genSclone{F}} = \Polp{\Invp[0]{F}} = \sLoc[0]{\genSclone{F}}}
for any \m{F\subs\Ops} and \m{s\in\N}.
\end{corollary}
\begin{easykeepproof}
By Lemma~\ref{lem:Invp=rel-pair-clone}, the set \m{\Invp{F}} is a relation
pair clone, so Lemma~\ref{lem:rel-pair-clone-sub-arity-closure} yields
\m{\Invp[m]{F}\subs\genRpclone{\Invp[s]{F}}} for all \m{m\leq s}
and \m{\CarrierSet\neq\emptyset}.
From Corollary~\ref{cor:genRpclone-subs-InvpPolp} one obtains
\m{\Polp{\Invp[m]{F}}\sups\Polp{\genRpclone{\Invp[s]{F}}}
                         =\Polp{\Invp[s]{F}}}
for all \m{m\leq s}, and hence, we have
\begin{align*}
\Polp{\Invp[s]{F}}=\bigcap_{0\leq m\leq s}\Polp{\Invp[m]{F}}
                 &=\Polp{\biguplus_{0\leq m\leq s}\Invp[m]{F}}\\
                 &=\Polp{\Invp[\leq s]{F}}
                  =\sLoc{\genSclone{F}},
\end{align*}
where the last equality holds by Theorem~\ref{thm:char-Polp-Invp-leq-s}.
\par
The claim for \m{\CarrierSet=\emptyset} follows by similar transformations.
\end{easykeepproof}
\par

The second corollary characterises the closure \m{\Polp{\Invp{}}}.
\begin{corollary}\label{cor:PolpInvp=Loc[]}
We have \m{\Polp{\Invp{F}} = \Loc{\genSclone{F}}} for all \m{F\subs\Ops}.
\end{corollary}
\begin{straightforwardproof}
Using the definition of the operator \m{\Loc{}} and
Theorem~\ref{thm:char-Polp-Invp-leq-s}, we can write
\begin{align*}
\Loc{\genSclone{F}} = \bigcap_{s\in\N}\sLoc{\genSclone{F}}
&=\bigcap_{s\in\N}\Polp{\Invp[{\leq s}]{F}}\\
&=\Polp{\bigcup_{s\in\N}\Invp[{\leq s}]{F}}
=\Polp{\Invp{F}}.\qedhere
\end{align*}
\end{straightforwardproof}
\par

The third corollary proves that a set \m{F\subs\Ops} is closed \wrt\
\m{\genSclone{\,}} and \m{\sLoc{}} if (and clearly only if) it is closed
\wrt\ to the operator \m{\sLoc{\genSclone{\,}}}. An analogous result holds,
of course, for the operators \m{\genSclone{\,}}, \m{\Loc{}} and
\m{\Loc{\genSclone{\,}}}. These two facts can be seen as generalisations
of Lemma~2.5(ii),(iii)
in~\cite[p.~22]{PoeGeneralGaloisTheoryForOperationsAndRelations}, where
similar results have been proven for clones.
\par
\begin{corollary}\label{cor:(s)-locally-closed-semiclones}
For \m{s\in\N} a set \m{F\subs\Ops} of operations is an \nbdd{s}locally
(locally) closed semiclone if and only if \m{\sLoc{\genSclone{F}}=F}
(\m{\Loc{\genSclone{F}}=F}).
\end{corollary}
\begin{easyproof}
If \m{\genSclone{F}=F} and \m{\sLoc{F}=F} (\m{\Loc{F}=F}), then it
clearly follows that \m{\sLoc{\genSclone{F}}=F}
(\m{\Loc{\genSclone{F}}=F}). Conversely, if the latter equality holds,
then we obviously have \m{\sLoc{F}=F} (\m{\Loc{F}=F}) by idempotence of
the \nbdd{s}local (local) closure. Besides, combining the condition
\m{F=\sLoc{\genSclone{F}}} (\m{F=\Loc{\genSclone{F}}}) with
Theorem~\ref{thm:char-Polp-Invp-leq-s}
(Corollary~\ref{cor:PolpInvp=Loc[]}), one obtains the equality
\m{F=\Polp{\Invp[{\leq s}]{F}}}
(\m{F=\Polp{\Invp{F}}}), and the latter set is a semiclone by
Lemma~\ref{lem:Polp=semiclone}. Therefore, \m{\genSclone{F}=F}.
\end{easyproof}
\par

The next corollary is in analogy to Lemma~2.6
in~\cite[p.~22]{PoeGeneralGaloisTheoryForOperationsAndRelations}.
\begin{corollary}\label{cor:Invp-sLoc}
For every \m{F\subs\Ops} we have
\begin{enumerate}[(a)]
\item\label{item:Invp-sLoc}
      \begin{sloppypar}
      \m{\Invp[m]{F}=\Invp[m]{\genSclone{F}}=\Invp[m]{\Loc{\genSclone{F}}}
        =\Invp[m]{\sLoc{\genSclone{F}}}} for all \m{m\leq s\in\N}
        whenever \m{\CarrierSet\neq\emptyset}.
      \end{sloppypar}
\item\label{item:Invp-Loc}
      \m{\Invp{F}=\Invp{\genSclone{F}}=\Invp{\Loc{\genSclone{F}}}}.
\end{enumerate}
\end{corollary}
\begin{proof}
\begin{enumerate}[(a)]
\item By Corollaries~\ref{cor:char-Polp-Invp-s}
      and~\ref{cor:(s)-locally-closed-semiclones}, and since \m{m\leq s},
      we have
      \begin{multline*}
      \Polp{\Invp[m]{\sLoc{\genSclone{F}}}}
        =\sLoc[m]{\genSclone{\sLoc{\genSclone{F}}}}\\
        =\sLoc[m]{\sLoc{\genSclone{F}}}
        =\sLoc[m]{\genSclone{F}}
        =\Polp{\Invp[m]{F}},
      \end{multline*}
      which implies \m{\Invp[m]{\sLoc{\genSclone{F}}} = \Invp[m]{F}} by
      applying \m{\Invp[m]{}} once more on both sides.
\item Similarly, we have
      \m{\Invp{\Loc{\genSclone{F}}}
        =\Invp{\Polp{\Invp{F}}}
        =\Invp{F}},
      using Corollary~\ref{cor:PolpInvp=Loc[]}.
      \qedhere
\end{enumerate}
\end{proof}
\par

Next, we turn to the characterisation of the other part of the
\name{Galois} connection.
\par

\subsection{The side of relation pairs}\label{subsect:InvpPolp}
We start by preparing the proof of our theorem with a lemma.
\begin{lemma}\label{lem:GammaF-in-genRpclone}
Let \m{Q\subs\Relps} be any set of relation pairs, \m{m\in\N} an arity and
\m{B\subs\CarrierSet[m]} be a finite subset of cardinality
\m{n\defeq\abs{B}}. Consider any enumeration
\m{\mathbf{b}=\apply{\listen{b}{s}}\in B^{s}} of
\m{B=\set{\listen{b}{s}}} (\m{s\geq n}) and define
\begin{align*}
\mu'_{B}&\defeq\lset{\composition{f}{\mathbf{b}}}{f\in\Polp[s]{Q}},&
\mu_{B}&\defeq\lset{\composition{f}{\mathbf{b}}}{f\in\Pol[s]{Q_1}},
\end{align*}
where \m{Q_1\defeq \lset{\rho\in\Rels}{\apply{\rho,\rho'}\in Q}}.
\begin{enumerate}[(a)]
\item\label{item:rel-pair-clone-contains-muB}
      The pair \m{\apply{\mu_{B},\mu'_{B}}} can be obtained from~\m{Q} by
      general superpositions, \ie\
      \m{\apply{\mu_{B},\mu'_{B}}\in\genRpclone{Q}}.
\item\label{item:rel-pair-clone-contains-GammaB}
      For \m{F\defeq\Polp{Q}} one may obtain \m{\genInvPair{B}} as a
      relaxation of \m{\apply{\mu_{B},\mu'_{B}}}, that is,
      \m{\genInvPair{B}\in\enc{\genRpclone{Q}}}.
\end{enumerate}
\end{lemma}
\begin{proof}
\begin{enumerate}[(a)]
\item In order to prove that \m{\apply{\mu_{B},\mu'_{B}}\in\genRpclone{Q}},
      we shall exhibit a general composition producing this relation pair
      from the ones in~\m{Q}. Using the notation from
      Definition~\ref{def:gen-comp}, we choose
      \m{\mu\defeq \CarrierSet[s]} and define
      \m{\functionhead{\beta}{m}{\CarrierSet[s]}} by
      \m{\beta\apply{i}\defeq
                       \apply{b_{0}\apply{i},\dotsc,b_{s-1}\apply{i}}}
      for \m{0\leq i<m}. Moreover, for \m{n\in\N} and
      \m{\apply{\rho,\rho'}\in \Fn{Q}} we put
      \m{I_{n,\apply{\rho,\rho'}}\defeq
          \lset{\apply{n,\rho,\rho',\mathbf{r}}}{\mathbf{r}\in\rho^{s}}};
      further, we define
      \m{I\defeq\biguplus_{n\in\N}\bigcup_{\apply{\rho,\rho'}\in\Fn{Q}}
                I_{n,\apply{\rho,\rho'}}}.
      Finally, for \m{\apply{n,\rho,\rho',\mathbf{r}}\in I} let the
      function
      \m{\functionhead{\alpha_{n,\rho,\rho',\mathbf{r}}}{%
                                                      n}{\CarrierSet[s]}}
      be given by
      \m{\alpha_{n,\rho,\rho',\mathbf{r}}\apply{j}
                   \defeq\apply{r_{0}\apply{j},\dotsc,r_{s-1}\apply{j}}}
      for \m{0\leq j<n}, where
      \m{\mathbf{r}=\apply{\listen{r}{s}}\in\rho^{s}}.
      \par
      \begin{sloppypar}
      We claim now that
      \m{\apply{\mu_{B},\mu'_{B}}=
         \genComp{\beta}{%
                  \apply{\alpha_{n,\rho,\rho',\mathbf{r}}}_{%
                                \apply{n,\rho,\rho',\mathbf{r}}\in I}}{%
              \apply{\rho,\rho'}_{\apply{n,\rho,\rho',\mathbf{r}}\in I}}},
      which can be checked by the following straightforward calculation.
      Denoting for each \m{\apply{n,\rho,\rho',\mathbf{r}}\in I} by
      \m{\sigma_{n,\rho,\rho',\mathbf{r}}} some relation in
      \m{\Rels[n]}, we have
      \end{sloppypar}
      \begin{align*}
      &\genComp{\beta}{%
         \apply{\alpha_{n,\rho,\rho',\mathbf{r}}}_{%
           \apply{n,\rho,\rho',\mathbf{r}}\in I}}{%
         \apply{\sigma_{n,\rho,\rho',\mathbf{r}}}_{%
            \apply{n,\rho,\rho',\mathbf{r}}\in I}}=\\
      &=\lset{\apply{f\apply{\beta{(i)}}}_{0\leq i<m}}{\!\!%
      \begin{aligned}[c]
      &f\in\CarrierSet[{\CarrierSet[s]}]\land
      \forall \apply{n,\rho,\rho',\mathbf{r}}\in I\colon\\
      &\apply{f\apply{\alpha_{n,\rho,\rho',\mathbf{r}}{\scriptstyle(0)}},
             \dotsc,
             f\apply{\alpha_{n,\rho,\rho',\mathbf{r}}{\scriptstyle(n-1)}}}
             \in\sigma_{n,\rho,\rho',\mathbf{r}}
      \end{aligned}
      }\\
      &=\lset{\composition{f}{\listen{b}{s}}}{\!%
      \begin{aligned}[c]
      f\in\CarrierSet[{\CarrierSet[s]}]\land{}
      &\forall n\in\N\,\forall\apply{\rho,\rho'}\in\Fn{Q}\\%
      &\forall \mathbf{r}=\apply{\listen{r}{s}}\in\rho^{s}\colon\\%
      &\qquad\composition{f}{\listen{r}{s}}
             \in\sigma_{n,\rho,\rho',\mathbf{r}}
      \end{aligned}
      }\\
      &=\lset{\composition{f}{\listen{b}{s}}}{\!%
      f\in\CarrierSet[{\CarrierSet[s]}]\land{}
      \forall n\in\N\,\forall\apply{\rho,\rho'}\in\Fn{Q}\colon
      f\preserves\apply{\rho,\sigma_{n,\rho,\rho',\mathbf{r}}}
      }.
      \end{align*}
      Specialising this to
      \m{\sigma_{n,\rho,\rho',\mathbf{r}}\defeq\rho'}, we get
      \begin{align*}
      &\genComp{\beta}{%
         \apply{\alpha_{n,\rho,\rho',\mathbf{r}}}_{%
           \apply{n,\rho,\rho',\mathbf{r}}\in I}}{%
         \apply{\rho'}_{%
            \apply{n,\rho,\rho',\mathbf{r}}\in I}}=\\
      &=\lset{\composition{f}{\listen{b}{s}}}{\!%
      f\in\CarrierSet[{\CarrierSet[s]}]\land{}
      \forall n\in\N\,\forall\apply{\rho,\rho'}\in\Fn{Q}\colon
      f\preserves\apply{\rho,\rho'}
      }\\
      &=\lset{\composition{f}{\listen{b}{s}}}{f\in\Polp[s]{Q}}
       =\mu'_{B}.
      \end{align*}
      Specialising once more to
      \m{\sigma_{n,\rho,\rho',\mathbf{r}}\defeq\rho}, we obtain
      \begin{align*}
      &\genComp{\beta}{%
         \apply{\alpha_{n,\rho,\rho',\mathbf{r}}}_{%
           \apply{n,\rho,\rho',\mathbf{r}}\in I}}{%
         \apply{\rho}_{%
            \apply{n,\rho,\rho',\mathbf{r}}\in I}}=\\
      &=\lset{\composition{f}{\listen{b}{s}}}{\!%
      f\in\CarrierSet[{\CarrierSet[s]}]\land{}
      \forall n\in\N\,\forall\apply{\rho,\rho'}\in\Fn{Q}\colon
      f\preserves\apply{\rho,\rho}
      }\\
      &=\lset{\composition{f}{\listen{b}{s}}}{f\in\Pol[s]{Q_{1}}}
       =\mu_{B}.
      \end{align*}
\item By Lemma~\ref{lem:Polp=semiclone} we have \m{\genSclone{F}=F}, and
      therefore Lemma~\ref{lem:semiclones+proj=clones}%
                    \eqref{item:char-clone-cl=sclone-cl+proj} yields
      \m{\genClone{F}=\genSclone{F}\cup\TrivOps=F\cup\TrivOps}.
      Hence, according to Lemma~\ref{lem:Gamma-F}, we obtain that
      \m{\genInvPair{B}=\apply{\rho,\rho'}}, wherein
      \m{\rho=\lset{\composition{f}{\mathbf{b}}}{%
                    f\in\Fn[s]{F}\cup\TrivOps[s]}} and
      \m{\rho'=\lset{\composition{f}{\mathbf{b}}}{f\in\Fn[s]{F}}
              =\lset{\composition{f}{\mathbf{b}}}{f\in\Polp[s]{Q}}
              =\mu'_{B}}.
      Moreover, as obviously \m{F =\Polp{Q}\subs\Pol{Q_{1}}}, we have
      \m{\mu'_{B}=\rho'\subs\rho\subs \mu_{B}}. Since
      \m{\apply{\mu_{B},\mu'_{B}}\in \genRpclone{Q}}
      by~\eqref{item:rel-pair-clone-contains-muB},
      we finally see
      \m{\genInvPair{B}=\apply{\rho,\rho'}\in\enc{\genRpclone{Q}}}.
      \qedhere
\end{enumerate}
\end{proof}
\par

Now, we are ready to address the dual side of the \name{Galois}
correspondence, \ie\ the closure \m{\Invp{\Polp[{\leq s}]{}}}.
\begin{theorem}\label{thm:char-Invp-Polp-leq-s}
For \m{s\in\N} and any set \m{Q\subs\Relps} of relation pairs we have
\m{\Invp{\Polp[\leq s]{Q}}=\sLOC{\genRpclone{Q}}}.
\end{theorem}
\begin{proof}
We have \m{\genRpclone{Q}\subs\Invp{\Polp[\leq s]{\genRpclone{Q}}}
                           =\Invp{\Polp[\leq s]{Q}}}
by Corollary~\ref{cor:genRpclone-subs-InvpPolp},
so \m{\sLOC{\genRpclone{Q}}\subs\sLOC{\Invp{\Polp[\leq s]{Q}}}
                               =\Invp{\Polp[\leq s]{Q}}}
by Lemma~\ref{lem:Invp-sLOC}.\par
For the converse inclusion let us consider \m{m\in\N} and an arbitrary
\nbdd{m}ary pair \m{\apply{\sigma,\sigma'}\in\Invp[m]{\Polp[\leq s]{Q}}}.
In order to prove that \m{\apply{\sigma,\sigma'}\in\sLOC{\genRpclone{Q}}},
we take any subset \m{B\subs\sigma} such that \m{\abs{B}\leq s}.
From Lemma~\ref{lem:GammaF-in-genRpclone}%
           \eqref{item:rel-pair-clone-contains-muB} we get
\m{\apply{\mu_{B},\mu'_{B}}\in\genRpclone{Q}}, and obviously we have
\m{B\subs\mu_{B}}. Moreover, since
\m{\apply{\sigma,\sigma'}} belongs to \m{\Inv{\Polp[\leq s]{Q}}}, we have
\m{f\preserves\apply{\sigma,\sigma'}} for all \m{f\in\Polp[\leq s]{Q}}.
So, as \m{\abs{B}\leq s} and \m{B\subs\sigma}, we get
\m{\mu'_{B}\subs\sigma'}. This proves
\m{\apply{\sigma,\sigma'}\in\sLOC{\genRpclone{Q}}}.
\end{proof}
\par

The following result is the analogue of
Lemma~\ref{lem:rel-pair-clone-sub-arity-closure} for semiclones.
\begin{lemma}\label{lem:semiclone-sub-arity-closure}
Any semiclone \m{F\subs\Ops} satisfies
\m{\Fn{F}\subs\genSclone{\Fn[s]{F}}} for all arities \m{n,s\in\N} where
\m{0<n\leq s}.
\end{lemma}
\begin{easyproof}
For \m{f\in\Fn{F}} we have \m{g\defeq
\composition{f}{\eni[s]{0},\dotsc,\eni[s]{n-1}}\in
\genSclone[s]{F}=\Fn[s]{F}}. It follows that
\m{f = \composition{g}{\eni{0},\dotsc,\eni{n-1},\eni{n-1},\dotsc,\eni{n-1}}
     \in \genSclone{g}\subs\genSclone{\Fn[s]{F}}}
due to superassociativity. Hence, we obtain
\m{\Fn{F}\subs\genSclone{\Fn[s]{F}}}.
\end{easyproof}
\par

For \m{n=0<s} the previous lemma (and its proof) fail. This is why in the
following corollary to Theorem~\ref{thm:char-Invp-Polp-leq-s} arities~\m{s}
and~\m{0} are required.
\begin{corollary}\label{cor:char-Invp-Polp-s}
For \m{s\in\N} and any set \m{Q\subs\Rels} of relation pairs we have the
equality
\m{\Invp{\Polp[{0,s}]{Q}}=\sLOC{\genRpclone{Q}}}.
\end{corollary}
\begin{easykeepproof}
As, by Lemma~\ref{lem:Polp=semiclone}, the set \m{\Polp{Q}} is a
semiclone, Lemma~\ref{lem:semiclone-sub-arity-closure} is applicable and
yields \m{\Polp[n]{Q}\subs\genSclone{\Polp[s]{Q}}} for all \m{0<n\leq s}.
Via Corollary~\ref{cor:genSclone-subs-PolpInvp} this implies
\m{\Invp{\Polp[n]{Q}}\sups\Invp{\genSclone{\Polp[s]{Q}}}
                         =\Invp{\Polp[s]{Q}}}
for all \m{0<n\leq s}, whence we obtain
\begin{align*}
\Invp{\apply{\Polp[0]{Q}\uplus\Polp[s]{Q}}}
&=\Invp{\Polp[0]{Q}}\cap \Invp{\Polp[s]{Q}}\\
&=\bigcap_{0\leq n\leq s} \Invp{\Polp[n]{Q}}\\
&=\Invp{\biguplus_{0\leq n\leq s} \Polp[n]{Q}}
=\Invp{\Polp[\leq s]{Q}}\\
&=\sLOC{\genRpclone{Q}},
\end{align*}
where the last equality is true by
Theorem~\ref{thm:char-Invp-Polp-leq-s}.
\end{easykeepproof}
\par

In case that the relation pairs contain an empty pair, the nullary
polymorphisms in Corollary~\ref{cor:char-Invp-Polp-s} vanish.
\begin{corollary}\label{cor:char-Invp-Polp-s-wo-null}
For \m{s,m\in\N} and any set \m{Q\subs\Relps} such that
\m{\apply{\emptyset,\emptyset}\in\Fn[m]{Q}}, we have
\m{\Invp{\Polp[s]{Q}} = \sLOC{\genRpclone{Q}}}.
\end{corollary}
\begin{easykeepproof}
Since a pair of empty relations belongs to~\m{Q}, and thus to
\m{\Invp{\Polp{Q}}}, Lemma~\ref{lem:rel-pair-clones-nullary-ops}
instantiated for \m{F=\Polp{Q}} implies that
\m{\Polp{Q}\subs\Ops\setminus\Ops[0]}, \ie\ \m{\Polp[0]{Q}=\emptyset}.
Therefore, the claim follows from Corollary~\ref{cor:char-Invp-Polp-s}.
\end{easykeepproof}
\par

Next, we characterise the closure \m{\Invp{\Polp{}}}.
\begin{corollary}\label{cor:InvpPolp=LOC[]}
We have \m{\Invp{\Polp{Q}} = \LOC{\genRpclone{Q}}} for all
\m{Q\subs\Relps}.
\end{corollary}
\begin{straightforwardproof}
From the definition of the operator \m{\LOC{}} and
Theorem~\ref{thm:char-Invp-Polp-leq-s}, we obtain
\begin{align*}
\LOC{\genRpclone{Q}}
&= \bigcap_{s\in\N}\sLOC{\genRpclone{Q}}
=\bigcap_{s\in\N}\Invp{\Polp[\leq s]{Q}}\\
&=\Invp{\biguplus_{s\in\N} \Polp[\leq s]{Q}}
= \Invp{\Polp{Q}}.
\qedhere
\end{align*}
\end{straightforwardproof}
\par

Moreover, we can infer that a set \m{Q\subs\Relps} is closed \wrt\
\m{\genRpclone{\,}} and \m{\sLOC{}} if (and clearly only if) it is closed
\wrt\ to the operator \m{\sLOC{\genRpclone{\,}}}. An analogous result
holds, of course, for the operators \m{\genRpclone{\,}}, \m{\LOC{}} and
\m{\LOC{\genRpclone{\,}}}. These two facts can be seen to generalise
Proposition~3.8(ii),(iii)
in~\cite[p.~30]{PoeGeneralGaloisTheoryForOperationsAndRelations}, where
similar statements have been proven for relational clones.
\par
\begin{corollary}\label{cor:(s)-locally-closed-rel-pair-clones}
For \m{s\in\N} a set \m{Q\subs\Relps} of relation pairs is an
\nbdd{s}locally (locally) closed relation pair clone if and only if
\m{\sLOC{\genRpclone{Q}}=Q}
(\m{\LOC{\genRpclone{Q}}=Q}).
\end{corollary}
\begin{easyproof}
Suppose that \m{\genRpclone{Q}=Q} and \m{\sLOC{Q}=Q} (\m{\LOC{Q}=Q}), then
it follows evidently that \m{\sLOC{\genRpclone{Q}}=Q}
(\m{\LOC{\genRpclone{Q}=Q}}). Conversely, assume the latter equality to
hold. By idempotence of \m{\sLOC{}} (\m{\LOC{}}), we clearly get
\m{\sLOC{Q}=Q} (\m{\LOC{Q}=Q}). Moreover, if we combine our assumption
\m{Q=\sLOC{\genRpclone{Q}}}
(\m{Q=\LOC{\genRpclone{Q}}}) with Theorem~\ref{thm:char-Invp-Polp-leq-s}
(Corollary~\ref{cor:InvpPolp=LOC[]}) we obtain
\m{Q=\Invp{\Polp[\leq s]{Q}}} (\m{Q=\Invp{\Polp{Q}}}), and the latter set
always is a relation pair clone by Lemma~\ref{lem:Invp=rel-pair-clone}.
Hence, \m{\genRpclone{Q}=Q}.
\end{easyproof}
\par

The next corollary is the analogue of Proposition~3.9
in~\cite[{p.~30 et seq.}]{PoeGeneralGaloisTheoryForOperationsAndRelations}.
\begin{corollary}\label{cor:Polp-sLOC}
For every \m{Q\subs\Relps} we have
\begin{enumerate}[(a)]
\item\label{item:Polp-sLOC}
      that the set
      \m{\Polp[{0,n}]{Q}=\Polp[{0,n}]{\genRpclone{Q}}
        =\Polp[{0,n}]{\LOC{\genRpclone{Q}}}}
      equals \m{\Polp[{0,n}]{\sLOC{\genRpclone{Q}}}}
      for all \m{0\leq n\leq s\in\N}.
\item\label{item:Polp-sLOC-wo-null}
      that
      \m{\Polp[n]{Q}=\Polp[n]{\sLOC{\genRpclone{Q}}}}
      for all \m{0\leq n\leq s\in\N},
      whenever \m{\apply{\emptyset,\emptyset}\in \Fn[m]{Q}} for some arity
      \m{m\in\N}.
\item\label{item:Polp-LOC}
      \m{\Polp{Q}=\Polp{\genRpclone{Q}}=\Polp{\LOC{\genRpclone{Q}}}}.
\end{enumerate}
\end{corollary}
\begin{proof}
\begin{enumerate}[(a)]
\item It suffices to prove that
      \m{\Polp[{0,n}]{Q}\subs\Polp[{0,n}]{\sLOC{\genRpclone{Q}}}} is true
      for all \m{n\leq s}. Upon application of
      Corollary~\ref{cor:char-Invp-Polp-s}, we can infer that
      \m{\Invp{\Polp[{0,n}]{\sLOC{\genRpclone{Q}}}}
      =\sLOC[n]{\genRpclone{\sLOC{\genRpclone{Q}}}}\!\!}, which by
      Corollary~\ref{cor:(s)-locally-closed-rel-pair-clones} equals
      \m{\sLOC[n]{\sLOC{\genRpclone{Q}}}}. Due to \m{n\leq s} the latter
      set coincides with \m{\sLOC[n]{\genRpclone{Q}}}, which is
      \m{\Invp{\Polp[{0,n}]{Q}}}, using again
      Corollary~\ref{cor:char-Invp-Polp-s}. Now the claim follows by once
      more applying \m{\Polp[{0,n}]{}} to both sides of the equation.
\item If \m{\apply{\emptyset,\emptyset}\in \Fn[m]{Q}} for some \m{m\in\N},
      then we may substitute in the proof of~\eqref{item:Polp-sLOC} the
      operator \m{\Polp[{0,n}]{}} by \m{\Polp[n]{}} and the use of
      Corollary~\ref{cor:char-Invp-Polp-s} by
      Corollary~\ref{cor:char-Invp-Polp-s-wo-null}. Everything else works
      as just seen.
\item We have
      \m{\Polp{\LOC{\genRpclone{Q}}}=\Polp{\Invp{\Polp{Q}}} = \Polp{Q}}
      by Corollary~\ref{cor:InvpPolp=LOC[]};
      the other inclusions are trivial.\qedhere
\end{enumerate}
\end{proof}
\par

\subsection{Characterisation of local closures for relation
            pairs}\label{subsect:char-LOC}
Finally, we shall consider another characterisation of the \nbdd{s}local
closure operators, involving \nbdd{s}directed unions. The statement can be
improved for sets of relation pairs fulfilling an additional closure
property, which is in particular satisfied by relation pair clones.
Hence, our characterisation is especially useful in connection with the
operator \m{\LOC{\genRpclone{\,}}}.
\par

Our first result is a generalisation of Proposition~1.13(ii)
in~\cite[p.~18]{PoeGeneralGaloisTheoryForOperationsAndRelations} to
relation pairs (see also~\cite[Proposition~1.6(ii), p.~256]{%
PoeConcreteRepresentationOfAlgebraicStructuresAndGeneralGaloisTheory}).
It works provided one accepts the axiom of choice.
\begin{proposition}\label{prop:char-sLOC-s-directed-unions}
For any set \m{Q\subs\Relps} of relation pairs and all \m{s\in\Np},
the following holds:
\begin{equation*}
\sLOC{Q} = \biguplus_{m\in\N}\lset{\ovflhbx{3.346225pt}
           \apply{\sigma,\sigma'}\in\Relps[m]%
           \ovflhbx{3.346225pt}}{\ovflhbx{3.346225pt}%
\begin{aligned}
\exists\,\sigma''\in\Rels[m]\colon
 &\sigma''\subs \sigma'\land{}\\
 &\exists\,\mathcal{T}\subs\Relps[m]
         \text{ \nbdd{s}directed}\colon\\
 &\qquad\quad\apply{\sigma,\sigma''}=\bigcup \mathcal{T}\land{}\\
 &\forall\apply{\rho,\rho'}\in\mathcal{T}\,
  \exists\tilde{\rho}\in\Rels[m]\colon\\
 &\rho\subs\tilde{\rho} \land \apply{\tilde{\rho},\rho'}\in\Fn[m]{Q}\\
\end{aligned}\ovflhbx{3.346225pt}}.
\end{equation*}
\end{proposition}
\begin{proof}
To prove the inclusion ``\m{\sups}'', let us consider any \m{m\in\N} and a
pair \m{\apply{\sigma,\sigma'}\in\Relps[m]} satisfying the lengthy
condition in the proposition. Its first part says that there is an
\nbdd{s}directed system \m{\mathcal{T}\subs\Relps[m]} whose union equals
\m{\apply{\sigma,\sigma''}} for some \nbdd{m}ary relation
\m{\sigma''\subs\sigma'}. The remaining part states that for every pair
\m{\apply{\rho,\rho'}\in\mathcal{T}} there is an \nbdd{m}ary relation
\m{\tilde{\rho}\sups\rho} such that
\m{\apply{\tilde{\rho},\rho'}\in\Fn[m]{Q}}. This implies that
\m{\apply{\rho,\rho'}\in  \enc{\set{\apply{\tilde{\rho},\rho'}}}
                     \subs\enc{Q}\subs\enc{\sLOC{Q}}}.
Since the set \m{\sLOC{Q}} is \nbdd{s}locally closed,
Corollary~\ref{cor:locally-closed-implies-relaxation-closed} implies that
it is also closed under relaxation. Hence, we have
\m{\mathcal{T}\subs\enc{\sLOC{Q}}=\sLOC{Q}}.
Now as~\m{\mathcal{T}} is an \nbdd{s}directed system,
Lemma~\ref{lem:s-directed-unions} yields that
\m{\apply{\sigma,\sigma''}\in\sLOC{Q}}. Thus, from
\m{\sigma''\subs\sigma'\subs\sigma=\sigma} we get that
\m{\apply{\sigma,\sigma'}\in  \enc{\set{\apply{\sigma,\sigma''}}}
                         \subs\enc{\sLOC{Q}}},
which we already showed to coincide with \m{\sLOC{Q}}.
\par
For the converse inclusion, take any
\m{\apply{\sigma,\sigma'}\in\sLOC[s][m]{Q}}, \m{m\in\N}. Then for any
\m{B\subs\sigma} such that \m{\abs{B}\leq s}, the set
\m{\Sigma_{B}\defeq\lset{\apply{\rho,\rho'}\in\Fn[m]{Q}}{%
                       B\subs\rho\land \rho'\subs\sigma'}}
is non\dash{}empty; using the axiom of choice, one can
fix some pair \m{\apply{\tilde{\rho}_{B},\rho'_{B}}\in \Sigma_{B}}.
It satisfies
\m{\rho'_{B}\subs\sigma'\subs\sigma}, thus,
\m{\rho'_{B}\subs\tilde{\rho}_{B}\cap \sigma \eqdef \rho_{B}}.
By construction, we have \m{B\subs\rho_{B}\subs\tilde{\rho}_{B}}, thus putting
\m{\mathcal{T}\defeq \lset{\apply{\rho_{B},\rho'_{B}}}{%
                           B\subs\sigma\land\abs{B}\leq s}}, the
collection~\m{\mathcal{T}} satisfies the second part of the condition we
need to verify. We shall check that~\m{\mathcal{T}} is \nbdd{s}directed
farther below; first we deal with the union
\m{\apply{\mu,\mu'}\defeq\bigcup\mathcal{T}} (meaning union in both
components). Since for every subset \m{B\subs\sigma}, \m{\abs{B}\leq s}, we
have \m{\rho'_{B}\subs\sigma'} and \m{\rho_{B}\subs\sigma}, it follows
that also \m{\mu'\subs\sigma'} and \m{\mu\subs\sigma}. Due to \m{s>0}, we
have that \m{\sigma=    \bigcup_{B\subs\sigma,\abs{B}\leq s} B
                   \subs\bigcup_{B\subs\sigma,\abs{B}\leq s}\rho_{B}
                   = \mu},
wherefore \m{\mu=\sigma}. This shows that \m{\apply{\sigma,\sigma'}} has
the right form to fit into the set on the right\dash{}hand side, provided
we establish that the non\dash{}empty set~\m{\mathcal{T}} is
\nbdd{s}directed.
\par
For this goal, we consider \m{t\leq s} subsets
\m{\listen{B}{t}\subs \sigma} subject
to the condition \m{\abs{B_i}\leq s} for each \m{0\leq i<t} and
tuples \m{r_i\in \rho_{B_i}\subs\sigma}. Let us define
\m{C\defeq \lset{r_{i}}{0\leq i<t}\subs\sigma}.
As \m{\abs{C}\leq s}, the
pair \m{\apply{\rho_{C},\rho'_{C}}} belongs to~\m{\mathcal{T}} by
definition. Thus \m{C\subs\rho_C} demonstrates \nbdd{s}directedness,
concluding the proof.
\end{proof}

\begin{remark}\label{rem:char-false-for-s=0}
The inclusion ``\m{\subs}'' in
Proposition~\ref{prop:char-sLOC-s-directed-unions} fails to hold for
\m{s=0}. Consider, for example, any pair of relations
\m{\rho'\subs\rho\subsetneq \CarrierSet[m]} for some fixed \m{m\in\N}.
Define \m{Q\defeq \enc{\set{\apply{\rho,\rho'}}} =
          \lset{\apply{\sigma,\sigma'}\in\Relps[m]}{%
                \rho'\subs\sigma'\subs\sigma\subs\rho}},
then~\m{Q} is certainly closed \wrt\ relaxation, and, moreover, it is not
hard to see that it is also closed under arbitrary non\dash{}empty
unions, \ie\ \nbdd{0}directed unions. Therefore, the set appearing on the
right\dash{}hand side in Proposition~\ref{prop:char-sLOC-s-directed-unions}
is contained in~\m{Q}. Now, the set
\m{\sLOC[0]{Q}
=\lset{\apply{\sigma,\sigma'}\in\Relps[m]}{\ovflhbx{1.01495pt}\exists
\apply{\mu,\mu'}\in\Fn[m]{Q}\colon \sigma'\sups \mu'}}
clearly contains \m{\apply{\CarrierSet[m],\rho'}}, but this pair does not
belong to the set on the right for it fails to belong to~\m{Q} due
to~\m{\CarrierSet[m]\not\subs\rho}.
\end{remark}
\par

In generalisation of Proposition~1.13(i)
of~\cite[p.~18]{PoeGeneralGaloisTheoryForOperationsAndRelations} (see
also~\cite[Proposition~1.6(i), p.~256]{%
PoeConcreteRepresentationOfAlgebraicStructuresAndGeneralGaloisTheory}), a
similar characterisation as above can be achieved for the local closure
operator.
\begin{corollary}\label{cor:char-LOC}
For any set \m{Q\subs\Relps} of relation pairs, we have
\begin{equation*}
\LOC{Q} = \biguplus_{m\in\N}\lset{\ovflhbx{2.66915pt}%
           \apply{\sigma,\sigma'}\in\Relps[m]%
           \ovflhbx{2.66915pt}}{\ovflhbx{2.66915pt}%
\begin{aligned}
\exists\,\sigma''\in\Rels[m]\colon
 &\sigma''\subs \sigma'\land{}\\
 &\exists\,\mathcal{T}\subs\Relps[m]
         \text{ \nbdd{\aleph_{0}}directed}\colon\\
 &\quad\quad\apply{\sigma,\sigma''}=\bigcup \mathcal{T}\land{}\\
 &\forall\apply{\rho,\rho'}\in\mathcal{T}\,
  \exists\tilde{\rho}\in\Rels[m]\colon\\
 &\rho\subs\tilde{\rho} \land \apply{\tilde{\rho},\rho'}\in\Fn[m]{Q}\\
\end{aligned}\ovflhbx{2.66915pt}}.
\end{equation*}
\end{corollary}
\begin{proof}
The proof of
Proposition~\ref{prop:char-sLOC-s-directed-unions} can be literally copied
employing the following modifications: the use of
Lemma~\ref{lem:s-directed-unions} has to be substituted by
Corollary~\ref{cor:inf-directed-union}; every occurrence of
``\nbdd{s}directed'', ``\nbdd{s}locally closed'' and the operator
\m{\sLOC{}} has to be replaced by ``\nbdd{\aleph_{0}}directed'', ``locally
closed'' and the operator \m{\LOC{}}, respectively; every restriction of
the form \m{\abs{B}\leq s}, \m{\abs{B_{i}}\leq s} and \m{\abs{C}\leq s}
should be changed to \m{\abs{B}<\aleph_{0}},
\m{\abs{B_{i}}\leq \aleph_{0}} and \m{\abs{C}<\aleph_{0}},
respectively; and, finally, the phrase ``Due to \m{s>0},'' is to be
removed completely.
\end{proof}
\par

Placing an additional closure requirement on the sets \m{Q\subs\Relps} in
Corollary~\ref{cor:char-LOC}, we can sharpen the statement by replacing
\nbdd{\aleph_{0}}directed unions by directed unions.
\par
\begin{corollary}\label{cor:char-LOC-directed-unions}
For any set \m{Q\subs\Relps} of relation pairs that is closed under
arbitrary intersections of pairs of identical arity,
the following equality holds:
\begin{equation*}
\LOC{Q} = \biguplus_{m\in\N}\lset{%
           \apply{\sigma,\sigma'}\in\Relps[m]%
           \ovflhbx{1.15654pt}}{\ovflhbx{1.15654pt}%
\begin{aligned}
\exists\,\sigma''\in\Rels[m]\colon
 &\sigma''\subs \sigma'\land{}\\
 &\exists\,\mathcal{T}\subs\Relps[m]
         \text{ directed}\colon\\
 &\quad\quad\apply{\sigma,\sigma''}=\bigcup \mathcal{T}\land{}\\
 &\forall\apply{\rho,\rho'}\in\mathcal{T}\,
  \exists\tilde{\rho}\in\Rels[m]\colon\\
 &\rho\subs\tilde{\rho} \land \apply{\tilde{\rho},\rho'}\in\Fn[m]{Q}\\
\end{aligned}}.
\end{equation*}
\end{corollary}
\begin{easyproof}
The proof of the inclusion ``\m{\sups}'' stays the same as in
Corollary~\ref{cor:char-LOC}, one just reads ``directed'' in place of
``\nbdd{\aleph_{0}}directed''.
\par
The dual inclusion requires a few more changes. Consider any relation pair
\m{\apply{\sigma,\sigma'}\in\LOC[m]{Q}}, \m{m\in\N}. Then for any
finite subset \m{B\subs\sigma}, the set
\m{\Sigma_{B}\defeq\lset{\apply{\rho,\rho'}\in\Fn[m]{Q}}{%
                       B\subs\rho\land \rho'\subs\sigma'}}
is non\dash{}empty. Let us define the pair
\m{\apply{\tilde{\rho}_{B},\rho'_{B}}\defeq\bigcap\Sigma_{B}},
that is, we have
\m{\tilde{\rho}_{B}=\bigcap_{\apply{\rho,\rho'}\in\Sigma_{B}}\rho}
and
\m{\rho'_{B}=\bigcap_{\apply{\rho,\rho'}\in\Sigma_{B}}\rho'}.
Since \m{\Sigma_{B}\neq\emptyset}, we have \m{\rho'_{B}\subs\rho'} for
some \m{\apply{\rho,\rho'}\in\Fn[m]{Q}}, \ie\
\m{\rho'_{B}\subs\rho'\subs\sigma'\subs\sigma}. Thus,
\m{\rho'_{B}\subs\tilde{\rho}_{B}\cap \sigma \eqdef \rho_{B}}.
Moreover, we know that \m{B\subs\rho_{B}\subs\tilde{\rho}_{B}} since
\m{B\subs\rho} for any \m{\apply{\rho,\rho'}\in\Sigma_{B}}. By closure
\wrt\ intersection, we obtain
\m{\apply{\tilde{\rho}_{B},\rho'_{B}}\in\Fn[m]{Q}}, and so
\m{\apply{\tilde{\rho}_{B},\rho'_{B}}\in\Sigma_{B}\subs\Fn[m]{Q}}.
Defining
\m{\mathcal{T}\defeq \lset{\apply{\rho_{B},\rho'_{B}}}{%
                           B\subs\sigma\land\abs{B}<\aleph_{0}}}, we can
continue as in the proof of Corollary~\ref{cor:char-LOC}; only the final
paragraph needs further modifications to demonstrate that~\m{\mathcal{T}}
is directed and not only \nbdd{\aleph_{0}}directed.
\par

For this we consider any finite subset
\m{\mathcal{B}\subs\lset{B\subs\sigma}{\abs{B}<\aleph_{0}}}. Clearly, 
the union \m{C\defeq \bigcup\mathcal{B}} is again a finite subset
of~\m{\sigma}. Moreover, for all \m{B\in\mathcal{B}}, we have
\m{B\subs C} and hence \m{\Sigma_{C}\subs\Sigma_{B}}, which implies
\m{\tilde{\rho}_{B}\subs\tilde{\rho}_{C}} and, consequently,
\m{\rho_{B}=\sigma\cap \tilde{\rho}_{B}
   \subs \sigma\cap\tilde{\rho}_{C} = \rho_{C}}.
Since this holds for all \m{B\in\mathcal{B}}, we get
\m{\bigcup_{B\in\mathcal{B}} \rho_{B}\subs\rho_{C}},
proving directedness of~\m{\mathcal{T}}.
\end{easyproof}

\section{Special cases}%
\label{sect:variants}

\subsection{Proper semiclones}%
\label{subsect:closed-proper-semiclones}

Based on the results of the previous section, we may also characterise all
\nbdd{s}locally closed semiclones that fail to be clones.
\begin{proposition}\label{prop:char-sLoc-proper-semiclones}
For any parameter \m{s\in\N} and any carrier set~\m{\CarrierSet}, the
collection
\m{\lset{F\in\Sclones\setminus\Clones}{F \text{ \nbdd{s}locally closed}}}
can be relationally described in the form
\m{\lset{\Polp{Q}}{Q\subs{\Relps[{\leq s}]}\land\exists
\apply{\rho,\rho'}\in Q\colon \rho'\subsetneq \rho}}.
\end{proposition}
\begin{proof}
Consider a set \m{Q\subs\Relps[{\leq s}]} such that
\m{\rho'\subsetneq\rho} holds for at least one relation pair
\m{\apply{\rho,\rho'}\in Q}.
By Lemma~\ref{lem:Polp=semiclone} we have \m{\Polp{Q}\in\Sclones}, and
moreover, this set is \nbdd{s}locally closed by Lemma~\ref{lem:Polp-sLoc}.
Further, Lemma~\ref{lem:char-Polp-clone} ensures that
\m{\Polp{Q}\notin\Clones}.
\par
Conversely, if \m{F\in\Sclones\setminus\Clones} is \nbdd{s}locally closed,
then by Corollary~\ref{cor:(s)-locally-closed-semiclones}, we have
\m{F=\sLoc{\genSclone{F}}=\Polp{Q}} %
where \m{Q=\Invp[{\leq s}]{F}}
(cf.\ Theorem~\ref{thm:char-Polp-Invp-leq-s}). If we had \m{\rho=\rho'}
for all \m{\apply{\rho,\rho'}\in Q}, then it would follow
\m{\id_{\CarrierSet}\in \Polp{Q}=F}, implying \m{F\in\Clones} according to
Lemma~\ref{lem:semiclones+proj=clones}%
\eqref{item:char-clones=semiclones-with-proj}. Hence, there is at least one
\m{\apply{\rho,\rho'}\in Q} where \m{\rho'\subsetneq\rho}.
\end{proof}

In a very analogous fashion we may prove the following result.
\begin{proposition}\label{prop:char-Loc-closed-proper-semiclones}
For all carriers~\m{\CarrierSet} we have the equality
\begin{multline*}
\lset{F\in\Sclones\setminus\Clones}{F \text{ locally closed}}\\
=\lset{\Polp{Q}}{Q\subs{\Relps}\land\exists
\apply{\rho,\rho'}\in Q\colon \rho'\subsetneq \rho}.
\end{multline*}
\end{proposition}
\begin{easyproof}
In the proof of Proposition~\ref{prop:char-sLoc-proper-semiclones} replace
the use of Lemma~\ref{lem:Polp-sLoc} by Corollary~\ref{cor:Polp-Loc},
Theorem~\ref{thm:char-Polp-Invp-leq-s} by
Corollary~\ref{cor:PolpInvp=Loc[]}, the operators \m{\sLoc{}} by
\m{\Loc{}}, \m{\Invp[{\leq s}]{}} by \m{\Invp{}} and \m{\Relps[{\leq s}]}
by \m{\Relps}, respectively, and ``\nbdd{s}locally'' by ``locally''.
\end{easyproof}

Using the theory of the previous sections, we can also prove a decidability
result regarding the question if a clone with projections removed yields a
semiclone, or if the non\dash{}trivial functions generate the projections.
\par

For this we need a more detailed analysis of the process generating
\m{\genInvPair{B}} occurring in Lemma~\ref{lem:Gamma-F}.
\begin{lemma}\label{lem:Gamma-F-algorithmically}
Let~\m{K} be any set, \m{B\subs\CarrierSet[K]} and \m{F\subs\Ops}.
Define \m{R_0\defeq B} and
\begin{align*}
S_j&\defeq \biguplus_{n\in\N}\lset{\composition{f}{\listen{g}{n}}}{%
 f\in\Fn{F}\land \apply{\listen{g}{n}}\in {R_j}^n}\\
R_{j+1}&\defeq R_j\cup S_j
\end{align*}
for \m{j\in\N}; set \m{R\defeq \bigcup_{j\in\N}R_j} and
\m{S\defeq\bigcup_{j\in\N}S_j}.
Employing the straightforward generalisation of the preservation concept to
infinite arities, the pair \m{\apply{R,S}} is the least
(\wrt~\m{\mathord{\leq}}) member of the set
\begin{equation*}
Q_{B}\defeq
\lset{\apply{\rho,\rho'}}{B\cup \rho'\subs\rho\subs \CarrierSet[K],
\apply{\rho,\rho'}\text{ is preserved by all } f\in F}.
\end{equation*}
\par
If \m{R_n=R_{n+1}}, \ie\ \m{S_n\subs R_n}, holds for some \m{n\in\N}, then
it is \m{S_m=S_n} and \m{R_m=R_n} for all \m{m\geq n}. Therefore, for
finite~\m{\CarrierSet} and finite~\m{K}, the condition \m{R_n=R_{n+1}} is
satisfied for some \m{n\leq \abs{\CarrierSet[K]}}.
\end{lemma}
\begin{proof}
First, we note that, by definition, \m{R_j\subs R_{j+1}}, which implies
\m{S_j\subs S_{j+1}}, holds for all \m{j\in \N}. Hence, the unions
defining~\m{R} and~\m{S} are directed.
\par
It is not difficult to see that \m{\apply{R,S}} belongs to~\m{Q_B}. Namely,
we have \m{B=R_0\subs R} and \m{S_j\subs R_{j+1}\subs R} for every
\m{j\in \N}, whence \m{S\subs R}. To prove that \m{\apply{R,S}} is
preserved by every \nbdd{n}ary \m{f\in F}, one considers an \nbdd{n}tuple
of tuples \m{\apply{\listen{g}{n}}\in R^n}. Due to directedness of the
union producing~\m{R} and finiteness of~\m{n}, there exists one \m{j\in\N}
such that \m{\apply{\listen{g}{n}}\in {R_j}^{n}}, wherefore
\m{\composition{f}{\listen{g}{n}}\in S_j\subs S}. Consequently,
\m{\apply{R,S}} is preserved by every member of~\m{F} and thus belongs
to~\m{Q_B}.
\par
Second, take any pair \m{\apply{\rho,\rho'}\in Q_B}. By definition, we have
\m{R_0=B\subs\rho}. Moreover, supposing that \m{R_j\subs\rho}, by the
preservation condition, we get that \m{S_j\subs\rho'\subs\rho} and hence
\m{R_{j+1} = R_j \cup S_j \subs \rho}, as well as \m{S_j\subs\rho'}. Thus,
by induction we have shown \m{R=\bigcup_{j\in\N}R_j\subs\rho} and
\m{S=\bigcup_{j\in\N} S_j\subs \rho'}.
This proves that \m{\apply{R,S}\leq \apply{\rho,\rho'}}, whence
\m{\apply{R,S}} is the \nbdd{\mathord{\leq}}least member of~\m{Q_{B}}.
\par
It is easy to check by induction that \m{R_n=R_{n+1}} entails \m{S_m=S_n}
and \m{R_m=R_n} for all \m{m\geq n}. Moreover, if \m{R_0,\liste{R}{n}} are
pairwise distinct (\ie\ form a strictly increasing chain \m{R_0\subsetneq
R_1\subsetneq \dotsm\subsetneq R_n}), then
\m{n\leq \abs{R_n}\leq \abs{\CarrierSet[K]}}. Therefore, for
finite~\m{\CarrierSet} and~\m{K}, the condition \m{R_n=R_{n+1}} must be
satisfied for some \m{n\leq \abs{\CarrierSet[K]}}.
\end{proof}
\par

The following lemma goes back to an idea by Peter Mayr.
\begin{lemma}\label{lem:gen-F-minus-projs}
For \m{F\subs\Ops} we have
\m{\genSclone{F\setminus\TrivOps}
  =\genSclone{\genClone{F}\setminus\TrivOps}}.
\end{lemma}
\begin{proof}
Using Lemma~\ref{lem:semiclones+proj=clones}%
            \eqref{item:char-clone-cl=sclone-cl+proj} we have
\begin{align*}
F\setminus\TrivOps\subs \genClone{F}\setminus\TrivOps
&=\genClone{F\setminus\TrivOps}\setminus\TrivOps
= \apply{\genSclone{F\setminus\TrivOps}\cup\TrivOps}\setminus\TrivOps\\
&= \genSclone{F\setminus\TrivOps}\setminus\TrivOps
\subs\genSclone{F\setminus\TrivOps},
\end{align*}
whence we obtain
\m{\genSclone{F\setminus\TrivOps} =
\genSclone{\genClone{F}\setminus\TrivOps}} by another application of
the operator \m{\genSclone{\,}}.
\end{proof}
\par

With this result we can now prove the problem of whether a finitely
generated clone on a finite set generates projections from its
non\dash{}trivial members to be decidable.
\begin{proposition}\label{prop:decidability-gen-of-projs}
For both~\m{\CarrierSet} and \m{F\subs\Ops} finite, it is decidable whether
\m{\genClone{F}\setminus\TrivOps\in\Sclones}.
\end{proposition}
\begin{proof}
Since \m{\genClone{F}} is a clone, and thus, in particular, a semiclone,
containing \m{\genClone{F}\setminus\TrivOps}, we have
\m{\genClone{F}\setminus\TrivOps
         \subs\genSclone{\genClone{F}\setminus\TrivOps}\subs\genClone{F}}.
Therefore, by Lemma~\ref{lem:semiclones+proj=clones}%
                    \eqref{item:semiclones-all-or-nothing}, the conditions
\m{\genClone{F}\!\setminus\TrivOps\notin\Sclones},
\m{\genClone{F}\!\setminus\TrivOps\subsetneq
                              \genSclone{\genClone{F}\!\setminus\TrivOps}},
\m{\genSclone{\genClone{F}\!\setminus\TrivOps}\cap\TrivOps\neq\emptyset},
\m{\TrivOps\subs\genSclone{\genClone{F}\!\setminus\TrivOps}}
and \m{\id_{\CarrierSet}\in\genSclone[1]{\genClone{F}\!\setminus\TrivOps}}
are all equivalent. By Lemma~\ref{lem:gen-F-minus-projs} we get
\m{\genSclone[1]{\genClone{F}\setminus\TrivOps}
                                       =\genSclone[1]{F\setminus\TrivOps}},
and using Corollary~\ref{cor:rho-X-n} for \m{n=1}, \m{X=A} and some
bijection~\m{\beta} between~\m{A} and its cardinality, we have a
description of the invariant pair
\m{\apply{\rho_{A,1},\rho'_{A,1}}
  \ovflhbx{0.906pt}=\ovflhbx{0.906pt}
  \genInvPair[{F\setminus\TrivOps\!\!}]{\set{\id_{\CarrierSet}\circ\beta}}}.
Via Lemma~\ref{lem:Gamma-F-algorithmically} this invariant relation pair
generated by \m{\id_{\CarrierSet}\circ\beta} can be expressed as
\m{\apply{\bigcup_{j\in\N} R_j, \bigcup_{j\in\N} S_j}} and finiteness
of~\m{\CarrierSet} guarantees that \m{R_n=R_{n+1}} happens for some
\m{n\leq\abs{\CarrierSet[A]}}. This implies that
\m{\lset{f\circ\beta}{\!\!f\in\genSclone[1]{F\setminus\TrivOps}}
   =\rho'_{A,1}}
can be written as the finite union
\m{\bigcup_{0\leq j\leq \abs{\CarrierSet[A]}} S_j}, which due to
finiteness of~\m{F} can be straightforwardly calculated using the
definitions of Lemma~\ref{lem:Gamma-F-algorithmically}. Hence, one may
check if \m{\id_{\CarrierSet}} belongs to
\m{\genSclone[1]{F\setminus\TrivOps}} by checking if
\m{\beta = \id_{\CarrierSet}\circ \beta} belongs to this union.
\end{proof}
\par

\subsection{Closed transformation semigroups}\label{subsect:loc-transf-sg}
By considering just unary parts we can obtain characterisations of locally
and \nbdd{s}locally closed (proper) transformation semigroups,
respectively.
\par

\begin{proposition}\label{prop:char-s-loc-transf-sg}
For \m{s\in\N} and a set \m{H\subs\Ops[1]} of transformations the
following facts are equivalent
\begin{enumerate}[(a)]
\item\label{item:sLoc-sg}
      \m{H} is an \nbdd{s}locally closed transformation semigroup
      [and \m{\id_{\CarrierSet}\notin H}].
\item\label{item:Endp-Invp-s-closed}
      \m{H=\Polp[1]{\Invp[{\leq s}]{H}}}
      [and there is \m{\apply{\rho,\rho'}\in\Invp[{\leq s}]{H}} such that
      \m{\rho'\neq\rho}].
\item\label{item:Endp-s-ary-pairs}
      \m{H=\Polp[1]{Q}} for some set \m{Q\subs\Relps[{\leq s}]}
      [where \m{\rho\neq\rho'} for some \m{\apply{\rho,\rho'}\in Q}].
\end{enumerate}
For non\dash{}empty carrier sets~\m{\CarrierSet}, the arity restrictions
``\m{\leq s}'' can be replaced by simply ``\m{s}''.
\end{proposition}
\begin{proof} We shall prove the implications ``\eqref{item:sLoc-sg}
$\Rightarrow$ \eqref{item:Endp-Invp-s-closed}
$\Rightarrow$ \eqref{item:Endp-s-ary-pairs}
$\Rightarrow$ \eqref{item:sLoc-sg}''.
\begin{prooflist}
\item[``\eqref{item:sLoc-sg} $\Rightarrow$
        \eqref{item:Endp-Invp-s-closed}'']
     Let \m{S\defeq \genSclone{H}} and \m{T\defeq\sLoc{S}}. Using
     Lemma~\ref{lem:semiclone-gen-by-transf} we obtain \m{H=\Fn[1]{S}}.
     We also have
     \m{T=\sLoc{S}=\sLoc{\genSclone{H}}=\Polp{\Invp[{\leq s}]{H}}}
     by Theorem~\ref{thm:char-Polp-Invp-leq-s}, whence
     \m{\Polp[1]{\Invp[{\leq s}]{H}} = \sLoc[s][1]{S}
                                     = \sLoc{\apply{\Fn[1]{S}}}}
     is equal to \m{\sLoc{H} = H} due to~\m{H} being \nbdd{s}locally
     closed.
     For the alternative reading note that \m{\id_{\CarrierSet}} would
     belong to \m{\Polp[1]{\Invp[{\leq s}]{H}}=H} if all pairs in
     \m{\Invp[{\leq s}]{H}} had equal components. Furthermore, if
     \m{\CarrierSet\neq \emptyset}, then one may use
     Corollary~\ref{cor:char-Polp-Invp-s} instead of
     Theorem~\ref{thm:char-Polp-Invp-leq-s} to change ``\m{\leq s}'' into
     ``\m{s}''.
\item[``\eqref{item:Endp-Invp-s-closed} $\Rightarrow$
        \eqref{item:Endp-s-ary-pairs}'']
     Simply choose \m{Q\defeq\Invp[{\leq s}]{H}}.
\item[``\eqref{item:Endp-s-ary-pairs} $\Rightarrow$
        \eqref{item:sLoc-sg}'']
     Clearly, \m{\Polp{Q}\in\Sclones} by Lemma~\ref{lem:Polp=semiclone},
     so Corollary~\ref{cor:unary-parts=transf-sg} yields that
     \m{H=\Polp[1]{Q}} is a transformation semigroup. Since
     \m{Q\subs\Relps[{\leq s}]}, we obtain
     \m{\Polp{Q}=\sLoc{\Polp{Q}}} by
     Lemma~\ref{lem:Polp-sLoc}, wherefore we may express~\m{H} as
     \m{H=\Polp[1]{Q} = \sLoc[s][1]{\Polp{Q}} = \sLoc{\Polp[1]{Q}} =
     \sLoc{H}}.
     Note for the second reading that
     \m{\id_{\CarrierSet}} does not preserve relation pairs
     \m{\rho'\subsetneq\rho}.\qedhere
\end{prooflist}
\end{proof}
\par

In an analogous way we may characterise the locally closed (proper)
transformation semigroups.
\par
\begin{proposition}\label{prop:char-loc-transf-sg}
For any set \m{H\subs\Ops[1]} of transformations the
following facts are equivalent:
\begin{enumerate}[(a)]
\item\label{item:Loc-sg}
      \m{H} is a locally closed transformation semigroup
      [and \m{\id_{\CarrierSet}\notin H}].
\item\label{item:Endp-Invp-closed}
      \m{H=\Polp[1]{\Invp{H}}}
      [and there is \m{\apply{\rho,\rho'}\in\Invp{H}} such that
      \m{\rho'\neq\rho}].
\item\label{item:Endp}
      \m{H=\Polp[1]{Q}} for some set \m{Q\subs\Relps}
      [where \m{\rho\neq\rho'} for some \m{\apply{\rho,\rho'}\in Q}].
\end{enumerate}
\end{proposition}
\begin{easyproof}
In the proof of Proposition~\ref{prop:char-s-loc-transf-sg} substitute
``\nbdd{s}locally'' by ``locally'', the operators
\m{\Invp[{\leq s}]{}} by \m{\Invp{}},
\m{\sLoc{}} by \m{\Loc{}}, and \m{\Relps[{\leq s}]} by \m{\Relps},
respectively, and the applications of
Theorem~\ref{thm:char-Polp-Invp-leq-s} by
Corollary~\ref{cor:PolpInvp=Loc[]}
and of Lemma~\ref{lem:Polp-sLoc} by Corollary~\ref{cor:Polp-Loc}.
\end{easyproof}
\par

By intersecting (in a similar way as outlined in this subsection) with
other classes of functions, for example, the set of all permutations
instead of all unary operations, one can obtain further characterisations
of locally closed classes of functions in terms of relation pairs.
Continuing the example of permutations, one may get a characterisation of
all locally [\nbdd{s}locally] closed (proper) transformation semigroups
that consist of permutations only. As on finite carrier sets every
permutation has a finite order, such a result is necessarily more
appealing on infinite domains.

\subsection{Classical \texorpdfstring{\m{\PolOp\text{-}\InvOp}}{Pol-Inv}
           \name{Galois} correspondence}%
\label{subsect:classical-Pol-Inv}

Here we demonstrate that it is not difficult to derive the
characterisations of the closure operators of the \name{Galois} connection given by
polymorphisms and invariant relations from our theorems above.
In this respect, we first consider the framework including nullary
operations and relations as discussed
in~\cite{BehClonesWithNullaryOperations}; from there it will be a small
step to obtain the variants known
from~\cite{PoeGeneralGaloisTheoryForOperationsAndRelations,%
PoeConcreteRepresentationOfAlgebraicStructuresAndGeneralGaloisTheory}.
\par

First we recollect information concerning the relationship of the
operators \m{\Pol{}} and \m{\Inv{}} \wrt\ \m{\Polp{}} and \m{\Invp{}},
which we already briefly discussed before Lemma~\ref{lem:char-Polp-hom}.
\begin{lemma}\label{lem:rel-PolInv-PolpInvp}
For \m{Q\subs\Rels}, \m{F\subs\Ops} and any \m{n\in\N} we have
\begin{align*}
\Pol[n]{Q} &=
\Polp[n]{\biguplus_{m\in\N}\lset{\apply{\rho,\rho}}{\rho\in \Fn[m]{Q}}},\\
\Inv[n]{F} &=
\lset{\rho}{\apply{\rho,\rho}\in \Invp[n]{F}}.
\end{align*}
\end{lemma}
\begin{easyproof}
The claim follows since a function \m{f\in\Ops} preserves a relation
\m{\rho\in\Rels} (\wrt~\m{\PolOp\text{-}\InvOp}) if and only if
\m{f\preserves\apply{\rho,\rho}} (in the sense of
\m{\PolpOp\text{-}\InvpOp}).
\end{easyproof}

We shall also need to express the set \m{\Invp{F}} in terms of \m{\Inv{F}}
for \m{F\subs\Ops} containing at least one projection.
\begin{lemma}\label{lem:Invp-via-Inv-for-F-with-proj}
Suppose, a set \m{F\subs\Ops} of operations satisfies
\m{F\cap\TrivOps\neq\emptyset}, then
\m{\Invp{F}=\biguplus_{m\in\N}\lset{\apply{\rho,\rho}}{\rho\in\Inv[m]{F}}}.
\end{lemma}
\begin{proof}
Put \m{Q\defeq \Invp{F}}. By Lemma~\ref{lem:Polp=semiclone}, \m{\Polp{Q}}
is a semiclone, and, as \m{\Polp{Q}\sups F} contains projections, it is
even a clone (cp.\ Lemma~\ref{lem:semiclones+proj=clones}%
                         \eqref{item:char-clones=semiclones-with-proj}).
Hence, Lemma~\ref{lem:char-Polp-clone} yields that
\m{Q\subs\biguplus_{m\in\N}\lset{\apply{\rho,\rho}}{\rho\in\Rels[m]}},
which implies
\m{Q=\biguplus_{m\in\N}
      \lset{\apply{\rho,\rho}\!}{\!\apply{\rho,\rho}\in\Invp[m]{F}}
    =\biguplus_{m\in\N}
    \lset{\apply{\rho,\rho}}{\rho\in\Inv[m]{F}}}
(cf.\ Lemma~\ref{lem:rel-PolInv-PolpInvp} for the second equality).
\end{proof}

This enables us now to derive the characterisation of the closure
operators \m{\Pol{\Inv[{\leq s}]{}}} and \m{\Pol{\Inv[s]{}}}.
\begin{theorem}\label{thm:pol-inv}
We have \m{\sLoc{\genClone{F}}=\Pol{\Inv[{\leq s}]{F}}} for
\m{F\subs\Ops} and \m{s\in\N}; moreover,
the equality \m{\sLoc{\genClone{F}}=\Pol{\Inv[s]{F}}} holds.
\end{theorem}
\begin{proof}
Using Lemma~\ref{lem:semiclones+proj=clones}%
\eqref{item:char-clone-cl=sclone-cl+proj} we can write
\m{\genClone{F}=\genSclone{F\cup\set{\id_{\CarrierSet}}}}, hence
Theorem~\ref{thm:char-Polp-Invp-leq-s} entails that
\m{\sLoc{\genClone{F}} = \sLoc{\genSclone{F\cup\set{\id_{\CarrierSet}}}}}
coincides with
\m{\Polp{\Invp[{\leq s}]{\apply{F\cup\set{\id_{\CarrierSet}}}}}}, which
by Lemma~\ref{lem:Invp-via-Inv-for-F-with-proj} equals
\begin{align*}
&\Polp{\biguplus_{m=0}^{s}\lset{\apply{\rho,\rho}}{\rho\in\Inv[m]{\apply{F\cup\set{\id_{\CarrierSet}}}}}}\\
&=\Polp{\biguplus_{m=0}^{s} \lset{\apply{\rho,\rho}}{\rho\in\Inv[m]{F}}}
\tag{since \m{\Inv{\apply{F\cup\set{\id_{\CarrierSet}}}} = \Inv{F}}}\\
&=\bigcap_{m=0}^{s}\Polp{\lset{\apply{\rho,\rho}}{\rho\in\Inv[m]{F}}}\\
&=\bigcap_{m=0}^{s}\Pol{\Inv[m]{F}} = \Pol{\Inv[{\leq s}]{F}}.
\tag{cf.\ Lemma~\ref{lem:rel-PolInv-PolpInvp}}
\end{align*}
For \m{\CarrierSet\neq\emptyset} we may replace
Theorem~\ref{thm:char-Polp-Invp-leq-s} by
Corollary~\ref{cor:char-Polp-Invp-s} and therefore the operator
\m{\Polp{\Invp[{\leq s}]{}}} by \m{\Polp{\Invp[s]{}}}. The rest of the
argument is analogous to the above. For \m{\CarrierSet=\emptyset},
\m{\Pol{\Inv[s]{F}}} and \m{\Pol{\Inv[{\leq s}]} = \sLoc{\genClone{F}}}
are both clones on~\m{\emptyset}, but there exists only one such
structure, namely \m{\Op{\emptyset}}.
\end{proof}

As a consequence, we get Theorem~3.17
from~\cite[p.~29]{BehClonesWithNullaryOperations}:
\begin{corollary}\label{cor:char-Pol-Inv}
For \m{F\subs\Ops} we have \m{\Loc{\genClone{F}} = \Pol{\Inv{F}}}.
\end{corollary}
\begin{straightforwardproof}
By definition of the operator \m{\Loc{}} we have
\begin{multline*}
\Loc{\genClone{F}} = \bigcap_{s\in\N}\sLoc{\genClone{F}}
\refeq{thm:pol-inv}\bigcap_{s\in\N} \Pol{\Inv[{\leq s}]{F}}\\
=\Pol{\bigcup_{s\in\N}\Inv[{\leq s}]{F}}
=\Pol{\Inv{F}}.\mbox{\qedhere}
\end{multline*}
\end{straightforwardproof}

In contrast to semiclones, nullary relations are never needed to discern
locally closed clones. Even more generally, invariants of small arity may
always be neglected.
\begin{corollary}\label{cor:nullary-rels-not-needed-for-clones}
For a set of operations \m{F\subs\Ops} and any arity \m{m\in\N}
we have the equality
\m{\Pol{\Inv[{\geq m}]{F}}=\Pol{\Inv{F}}=\Loc{\genClone{F}}}.
\end{corollary}
\begin{easyproof}
As a consequence of Theorem~\ref{thm:pol-inv} we have
\begin{multline*}
\Pol{\Inv{F}} = \bigcap_{s\in\N}\Pol{\Inv[s]{F}}
= \bigcap_{s\in\N}\sLoc{\genClone{F}}\\
=\bigcap_{\substack{s\in\N\\s\geq m}}\sLoc{\genClone{F}}
=\bigcap_{\substack{s\in\N\\s\geq m}}\Pol{\Inv[s]{F}} =
\Pol{\Inv[{\geq m}]{F}}
\end{multline*}
for any \m{F\subs\Ops}.
\end{easyproof}
\par

The following observation will be helpful in deriving the original
formulations (without nullary operations) of the previously presented
results.
\begin{lemma}\label{lem:nullary-part-of-non-nullary-ops}
For operations \m{F\subs\Ops\setminus\Ops[0]} of positive arity and every
\m{s\in\N} the equality
\m{\sLoc[s][0]{\genClone{F}} = \Loc[0]{\genClone{F}} = \genClone[0]{F}
=\emptyset} holds.
\end{lemma}
\begin{easyproof}
Since \m{F\subs\Ops\setminus\Ops[0]}, which is a clone, we obtain
\m{\genClone{F}\subs\Ops\setminus\Ops[0]}, \ie\
\m{\genClone[0]{F}=\emptyset} and thus
\m{\Loc[0]{\genClone{F}}\subs\sLoc[s][0]{\genClone{F}}
=\sLoc{\genClone[0]{F}}=\emptyset}.
\end{easyproof}
\par

\begin{corollary}\label{cor:classical-pol-inv}
Let \m{F\subs\Ops\setminus\Ops[0]} be without nullary operations and
\m{s\in\N}, then we have the equalities
\m{\Pol[{>0}]{\Inv{F}}=\Pol[{>0}]{\Inv[{>0}]{F}} = \Loc{\genClone{F}}}
and \m{\Pol[{>0}]{\Inv[s]{F}}=\sLoc{\genClone{F}}}.
\end{corollary}
\begin{easykeepproof}
Combining Lemma~\ref{lem:nullary-part-of-non-nullary-ops} with
Corollary~\ref{cor:nullary-rels-not-needed-for-clones} (for \m{m=1})
yields emptiness of the set
\m{\Pol[0]{\Inv[{>0}]{F}} = \Pol[0]{\Inv{F}} = \Loc[0]{\genClone{F}}}.
Therefore,
\begin{multline*}
\Pol[{>0}]{\Inv{F}} = \Pol{\Inv{F}} = \Loc{\genClone{F}}\\
  = \Pol{\Inv[{>0}]{F}} = \Pol[{>0}]{\Inv[{>0}]{F}}.
\end{multline*}
In a similar way, we may invoke Theorem~\ref{thm:pol-inv} together with
Lemma~\ref{lem:nullary-part-of-non-nullary-ops} to get
\m{\Pol[0]{\Inv[s]{F}} = \sLoc[s][0]{\genClone{F}} = \emptyset}.
Using again Theorem~\ref{thm:pol-inv} we can infer
\m{\Pol[{>0}]{\Inv[s]{F}} = \Pol{\Inv[s]{F}} = \sLoc{\genClone{F}}}.
\end{easykeepproof}

The equalities \m{\Loc{\genClone{F}}=\Pol[{>0}]{\Inv[{>0}]{F}}}
for \m{F\subs\Ops} without nullary operations, as well as
\m{\sLoc{\genClone{F}} = \Pol[{>0}]{\Inv[s]{F}}} whenever \m{s>0},
express two of the main results regarding \m{\PolOp\text{-}\InvOp} that
one finds in~\cite[Theorem~3.2, p.~260]{%
PoeConcreteRepresentationOfAlgebraicStructuresAndGeneralGaloisTheory},
\cite[Theorem~4.1, p.~31]{%
PoeGeneralGaloisTheoryForOperationsAndRelations}, where neither nullary
operations nor nullary relations were considered.
\par

In order to attack the relational side of the \m{\PolOp\text{-}\InvOp}
\name{Galois} correspondence, we need to express generated relational
clones, \ie\ the closure \m{\genRelClone{\,}} of a set of relations under
general superpositions, in terms of generated relation pair clones.
This is prepared in the following lemma.
\par

\begin{lemma}\label{lem:gen-Relclone--gen-Rpclone}
For any set \m{Q\subs\Rels} we have
\begin{equation*}
\biguplus_{m\in\N}\lset{\apply{\rho,\rho}}{\rho\in\genRelClone[m]{Q}}=
\genRpclone{\biguplus_{m\in\N}\lset{\apply{\rho,\rho}}{\rho\in \Fn[m]{Q}}}.
\end{equation*}
\end{lemma}
\begin{proof}
Since \m{Q\subs\genRelClone{Q}}, the set
\m{P\defeq\biguplus_{m\in\N}\lset{\apply{\rho,\rho}}{\rho\in\Fn[m]{Q}}}
obviously is a subset of
\m{\biguplus_{m\in\N}\lset{\apply{\rho,\rho}}{\rho\in\genRelClone[m]{Q}}}.
By Lemma~\ref{lem:rel-clone-rel-pair-clone}, the latter set is a relation
pair clone, whence it includes \m{\genRpclone{P}}.
\par
Conversely, the set
\m{\lset{\sigma}{\apply{\sigma,\sigma}\in\genRpclone{P}}},
which contains~\m{Q}, is a relational clone by
Lemma~\ref{lem:rel-pair-clone-rel-clone}; hence
\m{\genRelClone{Q}\subs
                   \lset{\sigma}{\apply{\sigma,\sigma}\in\genRpclone{P}}}.
Thus, whenever \m{\rho\in\genRelClone[m]{Q}} for \m{m\in\N}, we find
\m{\apply{\rho,\rho}} in \m{\genRpclone{P}}, which proves ``\m{\subs}''.
\end{proof}
\par

Similarly, we have to relate the \nbdd{s}local closures of sets of
relations (cf.\ page~\pageref{page:sLOC-rel}) and that of sets of relation
pairs.
\par
\begin{lemma}\label{lem:sLOC-vs-sLOC}
For any set \m{Q\subs\Rels} and \m{s\in\Np} we have
\begin{equation*}
\sLOC{\biguplus_{m\in\N}\lset{\apply{\rho,\rho}}{\rho\in\Fn[m]{Q}}}=
\biguplus_{m\in\N}\lset{\apply{\sigma,\sigma}}{\sigma\in\sLOC[s][m]{Q}}.
\end{equation*}
\end{lemma}
\begin{straightforwardproof}
It is sufficient to prove for fixed \m{m\in\N} that the
\nbdd{m}ary part of the left set,
\m{\sLOC[s][m]{\biguplus_{n\in\N}\lset{\apply{\rho,\rho}}{\rho\in\Fn{Q}}}
= \sLOC{\lset{\apply{\rho,\rho}}{\rho\in\Fn[m]{Q}}}}, coincides with
\m{\lset{\apply{\sigma,\sigma}}{\sigma\in\sLOC[s][m]{Q}}}.
By definition of the \nbdd{s}local closure we have that
\m{\sLOC{\lset{\apply{\rho,\rho}}{\rho\in\Fn[m]{Q}}}}
equals
\begin{equation*}
\lset{\apply{\sigma,\sigma'}\in\Relps[m]}{\forall B\subs\sigma,
\abs{B}\leq s\,\exists \rho\in\Fn[m]{Q}\colon
B\subs\rho\land \rho\subs\sigma'}.
\end{equation*}
Due to \m{s>0} any relation pair \m{\apply{\sigma,\sigma'}} belonging to
the previously displayed set satisfies
\m{\sigma=\bigcup\lset{B\subs\sigma}{\abs{B}\leq s}\subs
\sigma'\subs\sigma}, \ie\ \m{\sigma=\sigma'}.
Therefore, we obtain
\begin{align*}
&\sLOC{\lset{\apply{\rho,\rho}}{\rho\in\Fn[m]{Q}}}\\
&= \lset{\apply{\sigma,\sigma}}{\sigma\in\Rels[m]\land
\forall B\subs\sigma,\abs{B}\leq s\,\exists\rho\in\Fn[m]{Q}\colon
B\subs\rho\subs\sigma'}\\
&=\lset{\apply{\sigma,\sigma}}{\sigma\in\Rels[m]\land
\sigma\in\sLOC{Q}}
\end{align*}
as desired.
\end{straightforwardproof}

We may now characterise the closure operator \m{\Inv{\Pol[{\leq s}]{}}}
in terms of the \nbdd{s}local closure and the generated relational clone.
\begin{theorem}\label{thm:inv-pol}
For any parameter \m{s\in\N} and any set of relations \m{Q\subs\Rels} the
equalities \m{\sLOC{\genRelClone{Q}} = \Inv{\Pol[{\leq s}]{Q}}=
\Inv{\Pol[0,s]{Q}}} hold.
\end{theorem}
\begin{proof}
Using the previous results we may calculate for \m{Q\subs\Rels} and
\m{s\in\Np}
\begin{align*}
\Inv{\Pol[{\leq s}]{Q}}
&\refeq{lem:rel-PolInv-PolpInvp}
\Inv{\Polp[{\leq s}]{\biguplus_{m\in\N}
\lset{\apply{\rho,\rho}}{\rho\in\Fn[m]{Q}}}}\\
&\refeq{lem:rel-PolInv-PolpInvp}
\lset{\sigma}{\apply{\sigma,\sigma}\in\Invp{\Polp[{\leq s}]{%
\biguplus_{m\in\N}\lset{\apply{\rho,\rho}}{\rho\in\Fn[m]{Q}}}}}\\
&\refeq{thm:char-Invp-Polp-leq-s}
\lset{\sigma}{\apply{\sigma,\sigma}\in\sLOC{\genRpclone{%
\biguplus_{m\in\N}\lset{\apply{\rho,\rho}}{\rho\in\Fn[m]{Q}}}}}\\
&\refeq{lem:gen-Relclone--gen-Rpclone}
\lset{\sigma}{\apply{\sigma,\sigma}\in\sLOC{%
\biguplus_{m\in\N}\lset{\apply{\rho,\rho}}{\rho\in\genRelClone[m]{Q}}}}\\
&\refeq{lem:sLOC-vs-sLOC}
\lset{\sigma}{\apply{\sigma,\sigma}\in\biguplus_{m\in\N}
  \lset{\apply{\rho,\rho}}{\rho\in\sLOC[s][m]{\genRelClone{Q}}}}\\
&\stackrel{\phantom{\text{\ref{lem:sLOC-vs-sLOC}}}}{=}
\sLOC{\genRelClone{Q}}.
\end{align*}
Employing Corollary~\ref{cor:char-Invp-Polp-s} instead of
Theorem~\ref{thm:char-Invp-Polp-leq-s} in the previous calculation, one
may replace the operator \m{\Inv{\Pol[{\leq s}]{}}} by
\m{\Inv{\Pol[{0,s}]{}}}, and \m{\Invp{\Polp[{\leq s}]{}}} by
\m{\Invp{\Polp[{0,s}]{}}}, respectively, in the manipulations above.
\par
Due to inapplicability of Lemma~\ref{lem:sLOC-vs-sLOC} for \m{s=0}, this
case needs a manual proof. Clearly, we have
\m{\sLOC[0]{\genRelClone{Q}}=
\lset{\sigma\in\Rels}{\exists\rho\in\genRelClone{Q}\colon \sigma\sups\rho}}
and
\m{\Inv{\Pol[0]{Q}}
=\Inv{\lset{\cna[0]{a}}{\forall\rho\in Q\colon\apply{a,\dotsc,a}\in\rho}}
=\lset{\sigma\in\Rels}{\sigma\sups\mu}},
in which \m{\mu\defeq\lset{\apply{a,\dotsc,a}}{a\in A\land
\forall \rho\in Q\colon \apply{a,\dotsc,a}\in \rho}}
and \m{\cna[0]{a}} denotes the nullary operation with image \m{\set{a}}.
It is easy to see that \m{\mu\in\genRelClone{Q}}, namely, for
\m{\sigma\in\Rels}, let \m{\functionhead{\beta}{\arity\apply{\sigma}}{1}}
and \m{\functionhead{\alpha_{\rho}}{\arity\apply{\rho}}{1}} for
\m{\rho\in Q} be the unique constant mappings, then
\m{\mu=\genComp{\beta}{\apply{\alpha_{\rho}}_{\rho\in Q}}%
                              \apply{\rho}_{\rho\in Q}\in\genRelClone{Q}}.
This proves the inclusion
\m{\Inv{\Pol[0]{Q}}\subs\sLOC[0]{\genRelClone{Q}}}. The converse is
simple: if \m{\sigma\in\Rels} includes some \m{\rho\in\genRelClone{Q}},
and \m{\cna[0]{a}\in \Pol[0]{Q}=\Pol[0]{\genRelClone{Q}}}, then
\m{\apply{a,\dotsc,a}\in \rho\subs\sigma}. As this holds for all
constants in \m{\Pol[0]{Q}}, we obtain \m{\sigma\in\Inv{\Pol[0]{Q}}}.
\end{proof}

The previous theorem directly entails Theorem~3.20
of~\cite[p.~31]{BehClonesWithNullaryOperations}:
\begin{corollary}\label{cor:char-inv-pol}
For \m{Q\subs\Rels} we have \m{\LOC{\genRelClone{Q}} = \Inv{\Pol{Q}}}.
\end{corollary}
\begin{straightforwardproof}
By definition of the operator \m{\LOC{}} we have for \m{Q\subs\Rels}:
\begin{align*}
\LOC{\genRelClone{Q}} &= \bigcap_{s\in\N}\sLOC{\genRelClone{Q}} \\
&=\bigcap_{s\in\N}\Inv{\Pol[{\leq s}]{Q}}
=\Inv{\bigcup_{s\in\N}\Pol[{\leq s}]{Q}}
=\Inv{\Pol{Q}}.\qedhere
\end{align*}
\end{straightforwardproof}
\par

The following evident observation will be needed for the next corollaries.
\begin{lemma}\label{lem:Pol-emptyset}
We have \m{\Pol{\set{\emptyset}}=\Ops\setminus\Ops[0]} and thereby
\m{\Pol[s]{\set{\emptyset}} =\Ops[s]}
whenever \m{s\in\Np}; therefore,
\m{\Pol[s]{\apply{Q\cup\set{\emptyset}}} = \Pol[s]{Q}}
for any \m{Q\subs\Rels}.
\end{lemma}
\par

Similarly as in Corollary~\ref{cor:char-Invp-Polp-s-wo-null},
for sets of relations comprising the empty relation, nullary polymorphisms
are not required.
\begin{corollary}\label{cor:char-inv-pol-s-wo-null}
Let \m{Q\subs\Rels} and \m{s\in\N}. Then
\m{\sLOC{\genRelClone{Q}}=\Inv{\Pol[s]{Q}}} if (and, provided that \m{s>0},
also only if) \m{\emptyset\in\Inv{\Pol{Q}}} (which is true in particular
if \m{\emptyset\in Q}).
\end{corollary}
\begin{easyproof}
The conditions \m{\emptyset\in\Inv{\Pol{Q}}} and
\m{\Pol{Q}\subs \Pol{\set{\emptyset}} = \Ops\setminus\Ops[0]} (cf.\
Lemma~\ref{lem:Pol-emptyset} above) are equivalent; moreover, the latter
one holds if and only if \m{\Pol[0]{Q}=\emptyset}. Combining this with the
equality from Theorem~\ref{thm:inv-pol} yields
\m{\sLOC{\genRelClone{Q}} = \Inv{\Pol[s]{Q}}}.
Conversely, if we assume this equality and suppose \m{s>0}, which
entails \m{\Pol[s]{Q}\subs\Ops\setminus\Ops[0]}, then via
Theorem~\ref{thm:inv-pol} we get
\begin{align*}
\emptyset\in\Inv{\apply{\Ops\setminus\Ops[0]}}
&\subs\Inv{\Pol[s]{Q}} =\sLOC{\genRelClone{Q}} = \Inv{\Pol[{\leq s}]{Q}}\\
&\subs \Inv{\Pol[0]{Q}}.
\end{align*}
This is equivalent to
\m{\Pol[0]{Q}\subs\Pol{\set{\emptyset}}=\Ops\setminus\Ops[0]}, \ie\
\m{\Pol[0]{Q}=\emptyset}.
\end{easyproof}
\par

\begin{corollary}\label{cor:char-sLOC[Q+emptyset]}
We have
\m{\LOC{\genRelClone{Q\cup\set{\emptyset}}} =\Inv{\Pol[{>0}]{Q}}}
for \m{Q\subs\Rels}. Moreover, for \m{s\in\Np} the equality
\m{\sLOC{\genRelClone{Q\cup\set{\emptyset}}}=\Inv{\Pol[s]{Q}}} holds.
\end{corollary}
\begin{easyproof}
We have
\m{\Pol{\apply{Q\cup\set{\emptyset}}}
= \Pol{Q}\cap \apply{\Ops\setminus\Ops[0]}= \Pol[{>0}]{Q}},
applying Lemma~\ref{lem:Pol-emptyset}; thus
\m{\LOC{\genRelClone{Q\cup\set{\emptyset}}}
= \Inv{\Pol{\apply{Q\cup\set{\emptyset}}}} = \Inv{\Pol[{>0}]{Q}}} by
Corollary~\ref{cor:char-inv-pol}.
Combining for \m{s\in\Np} the statements of
Corollary~\ref{cor:char-inv-pol-s-wo-null} and Lemma~\ref{lem:Pol-emptyset}
yields the remaining claim.
\end{easyproof}
\par

Restricting the statement of Corollary~\ref{cor:char-sLOC[Q+emptyset]} to
sets \m{Q\subs\Rels\setminus\Rels[0]} and intersecting the equalities on
both sides with \m{\Rels\setminus\Rels[0]} yields the characterisations
\m{\LOC{\apply{\genRelClone[{>0}]{Q\cup\set{\emptyset}}}}
   = \Inv[{> 0}]{\Pol[{>0}]{Q}}}
and, for positive parameters~\m{s},
\m{\sLOC{\apply{\genRelClone[{>0}]{Q\cup\set{\emptyset}}}}
    =\Inv[{>0}]{\Pol[s]{Q}}}.
The closure \m{\genRelClone[{>0}]{Q\cup\set{\emptyset}}} describes the
appropriate notion of generated relational clone (as employed \eg\
in~\cite{PoeGeneralGaloisTheoryForOperationsAndRelations}) if one does
neither consider nullary operations nor relations in connection with
\m{\PolOp\text{-}\InvOp}. With the two stated equalities we have
therefore established the two main results (see Theorem~4.2, p.~32, and
Theorem~3.3, p.~260, respectively)
of~\cite{PoeGeneralGaloisTheoryForOperationsAndRelations,%
PoeConcreteRepresentationOfAlgebraicStructuresAndGeneralGaloisTheory}
regarding the relational side of the mentioned \name{Galois} connection.
\par

\section{Possible applications}\label{sect:applications}

In the literature the \m{\PolOp\text{-}\InvOp} \name{Galois} connection
has been very successfully employed to discover the structure of the
lattice of all clones (\eg~\cite{%
RosenbergMaximalClones,%
JanovMucnik1959,
ZhukCardinalitySetOfClonesContainingMinimalCloneOnThreeElements,
ZhukSelfDualFunctionsInThreeValuedLogic}),
but it is also fundamentally involved in investigating other problems in
algebra and theoretical computer science (\cite{%
BulatovCSPDichotomyOn3,
BartoKozikAbsorbingSubalgebrasCyclicTermsCSP,
BartoDichotomyForConservativeCSP,
BodirskyPinskerPongraczReconstructingTopologyOfClones,
VargasCRelsCclones}).
It is to be expected that the theory developed within this article will
find similar applications \wrt\ semiclones in the future, especially
regarding infinite carrier sets.
\par

In this connection we briefly outline one possible idea, picking up again
the topic of topologically closed (proper) transformation semigroups from
the previous section.
According to Proposition~\ref{prop:char-loc-transf-sg}, for any set
\m{Q\subs\Relps} of relation pairs, the set of all locally
closed transformation semigroups \m{S\subs\Ops[1]} lying properly below
\m{\Polp[1]{Q}} can be described as all those sets \m{S=\Polp[1]{\Sigma}}
(\m{\Sigma\subs\Relps}) satisfying
\m{\Polp[1]{\Sigma}\subsetneq\Polp[1]{Q}}.
If~\m{S} is a maximal member of this collection with respect to
inclusion, then \m{Q\subs\Invp{\Polp[1]{Q}}\subsetneq \Invp{S}}.
One may take any pair
\m{\apply{\rho,\rho'}\in\Invp{S}\setminus\Invp{\Polp[1]{Q}}} and obtains
that
\m{S=\Polp[1]{\Invp{S}}
     \subs\Polp[1]{\apply{Q\cup\set{\apply{\rho,\rho'}}}}
     \subsetneq\Polp[1]{Q}}, which by maximality of~\m{S} entails that
\m{S=\Polp[1]{\apply{Q\cup\set{\apply{\rho,\rho'}}}}}.
In case that \m{\Polp[1]{Q}} is a monoid, \ie\ \m{Q} contains only
relation pairs with identical components and \m{\Polp{Q}} is a real clone,
then one may also be interested in the maximal locally closed proper
transformation semigroups below it. This additional requirement enforces
that the pair \m{\apply{\rho,\rho'}} one had to add above even has to be a
proper relation pair, \ie\ \m{\rho'\subsetneq \rho}.
\par

In a similar way all maximal locally closed (or \nbdd{s}locally closed)
(possibly proper) semiclones (or \nbdd{s}locally closed transformation
semigroups) below one specific structure of the respective sort can be
described by preserving one additional relation pair. It is plausible that
for certain sets~\m{Q} a complete characterisation in analogy
to~\cite{RosenbergMaximalClones} can be attempted. Furthermore, on infinite
carrier sets, the machinery developed in this paper can also be useful to
reveal counterexamples, \eg\ structures having no maximal proper (locally
or \nbdd{s}locally closed) substructures below them.
It is, for example, not hard to prove for \m{Q=\emptyset} that proper
semigroups of the form \m{\Polp[1]{\set{\apply{\rho,\rho'}}}} with
\m{\rho'\subsetneq\rho\subs\CarrierSet} can never be maximal among all
locally closed proper transformation semigroups on any at least
two\dash{}element set~\m{\CarrierSet}.
\par

The author is, moreover, confident that a generalisation of the presented
theory to categories with finite powers is possible along the lines
of~\cite{KerkhoffGeneralGaloisTheoryFunRelInCats}, where a similar project
has been realised for clones and the \m{\PolOp{\text{-}}\InvOp}
\name{Galois} connection (at the same time dualising the involved notions,
which is not in our focus).
Most of our results do not impose any restrictions on the carrier set,
\ie\ the particular object of the category of sets the \name{Galois}
theory is based on. Therefore, the main theorems of this article could be a
guideline and used to hint at what form of results to expect in the general
setting.
Once such a generalisation has been established, the corresponding results
can be instantiated in any category of interest, as long as it has finite
powers, for instance, in that of topological spaces. In this way, it may be
possible to perform similar investigations as sketched above also for
transformation semigroups consisting of continuous functions.

\def\ocirc#1{\ifmmode\setbox0=\hbox{$#1$}\dimen0=\ht0 \advance\dimen0
  by1pt\rlap{\hbox to\wd0{\hss\raise\dimen0
  \hbox{\hskip.2em$\scriptscriptstyle\circ$}\hss}}#1\else {\accent"17 #1}\fi}
  \def\Palatalization#1{\bgroup\fontencoding{T1}\selectfont\v{#1}\egroup}
  \ifx\SetBibliographyCyrillicFontfamily\undefined\def\SetBibliographyCyrillicFontfamily{\relax}\fi
  \def\Cyr#1{\bgroup\SetBibliographyCyrillicFontfamily\fontencoding{T2A}\selectfont{#1}\egroup}
  \ifx\AvailableAt\undefined\def\AvailableAt{available at}\fi
  \ifx\AVAILABLEat\undefined\def\AVAILABLEat{Available at}\fi
  \ifx\OnlineAvblAt\undefined\def\OnlineAvblAt{online}\fi
  \ifx\PreprintAvblAt\undefined\def\PreprintAvblAt{pre-print}\fi
\providecommand{\bysame}{\leavevmode\hbox to3em{\hrulefill}\thinspace}
\providecommand{\MR}{\relax\ifhmode\unskip\space\fi MR }
\providecommand{\MRhref}[2]{%
  \href{http://www.ams.org/mathscinet-getitem?mr=#1}{#2}
}
\providecommand{\href}[2]{#2}

\end{document}